\documentclass[12pt]{amsart}

\usepackage[T1]{fontenc}
\usepackage[utf8]{inputenc}
\usepackage[english]{babel}
\usepackage{amsmath,amssymb,amsthm,amsfonts}
\usepackage{mathtools}
\mathtoolsset{showonlyrefs}
\usepackage[margin=0.80in]{geometry}
\usepackage{booktabs,array}
\usepackage{enumitem}
\usepackage{xcolor}
\usepackage[hidelinks]{hyperref}
\newcommand{\C}{\mathbb{C}}
\newcommand{\Z}{\mathbb{Z}}
\newcommand{\Q}{\mathbb{Q}}
\newcommand{\F}{\mathbb{F}}
\newcommand{\Uone}{\mathrm{U}(1)}
\newcommand{\B}{\mathcal{B}}
\newcommand{\A}{\mathcal{A}}
\newcommand{\T}{\mathcal{T}}
\newcommand{\Htorus}{\mathcal{H}}
\newcommand{\End}{\mathrm{End}}
\newcommand{\Hom}{\mathrm{Hom}}
\newcommand{\Aut}{\mathrm{Aut}}
\newcommand{\GL}{\mathrm{GL}}
\newcommand{\Sp}{\mathrm{Sp}}
\newcommand{\Tr}{\mathrm{Tr}}
\newcommand{\sign}{\mathrm{sign}}
\newcommand{\id}{\mathrm{id}}
\newcommand{\ii}{\mathrm{i}}
\newcommand{\ot}{\otimes}

\newcommand{\pair}[2]{\left\langle#1,#2\right\rangle}
\DeclareMathOperator{\rank}{rank}

\theoremstyle{plain}
\newtheorem{theorem}{Theorem}[section]
\newtheorem{prop}[theorem]{Proposition}
\newtheorem{lemma}[theorem]{Lemma}
\newtheorem{cor}[theorem]{Corollary}
\newtheorem{conj}[theorem]{Conjecture}

\theoremstyle{definition}
\newtheorem{definition}[theorem]{Definition}
\newtheorem{example}[theorem]{Example}
\newtheorem{computation}[theorem]{Computation}

\theoremstyle{remark}
\newtheorem{remark}[theorem]{Remark}

\numberwithin{equation}{section}

\title{Unitary Yang--Baxter Operators: Towards a Classification}
\author{C\'esar Galindo}
\address{Departamento de Matem\'aticas, Universidad de los Andes, Bogot\'a, Colombia}
\email{cn.galindo1116@uniandes.edu.co}
\author{Eric C. Rowell}
\address{Mathematics Department, Texas A\&M University and School of Mathematics, University of Leeds}
\email{rowell@tamu.edu}
\subjclass[2020]{Primary 16T25; Secondary 20F36, 57K14}
\keywords{Yang--Baxter operators, braid group representations, finite quadratic forms, Gaussian representations, Clifford groups, link invariants}
\date{}

\begin{document}

\begin{abstract}
There is a well-known circle of conjectures relating unitary solutions of
the Yang--Baxter equation, unitary braided fusion categories, topological
quantum computation, and link invariants.  Progress is limited by the lack
of a classification of unitary Yang--Baxter operators.

We propose a conjectural classification with three generating sources: monomial
solutions, group-type solutions arising from Yetter--Drinfeld modules, and
solutions from twisted group-algebra towers.  We conjecture
that, up to unit scalars and local unitary basis changes, all solutions are
generated by these  sources.

Our main constructions come from twisted group-algebra towers, first developed by Gustafson-Kimball-Rowell-Zhang.  We apply this method to 
metric groups which yields three unitary families, including Gaussian/metaplectic solutions and a family with nontrivial twisting.  We
provide tensor-power realizations, factorization results, braid
images computations, and exact polynomial-time evaluation of the associated link invariants
for fixed metric data.  We also obtain one-parameter families in dimension \(8\) that are generically
non-Gaussian and non-Clifford.
These yield infinite projective images, but we prove these braid
images are virtually abelian, leaving a conjecture of Rowell-Wang intact.

We provide evidence for our classification conjecture by means of computational searches for unitary Yang--Baxter operators among Clifford groups and a polynomial-in-Pauli ansatz.  This search yields only a few solutions that do not appear to be equivalent to 
monomial or group type solutions.  Direct calculations identify their link invariants
with specializations of the HOMFLYPT and BMW/Kauffman invariants; the latter
is a squared Jones specialization at a primitive sixth root, and the
corresponding qutrit operator strictly localizes the semisimple 
quotient of the BMW algebra specializations $\mathcal{C}_n(q^3,e^{\pi\ii/6})$ associated with the braided fusion category $SO(4)_2$.  The quaternionic Family III
specialization realizes a two-eigenvalue class and gives
tensor-power realization of the semisimple Hecke algebra tower associated
with \(SU(3)_3\). Significantly, this answers a recent existence question of Lechner, and verifies an open case of the Rowell-Wang localization conjecture from 2010.
\end{abstract}

\maketitle

\section{Introduction}
\label{sec:introduction}

Matrix solutions to the Yang--Baxter equation play a unifying role in several areas of study including braid group representations, knot theory, quantum computation and category theory. 
 Although the constraints are simply cubic polynomials, the space of solutions is not well understood--a general classification is only known in dimension $2$ \cite{Hietarinta92}.  In spite of some recent progress on restricted families of solutions (e.g., \cite{MartinRowell,LechnerTwoEigenvalues}), there is no general approach to the problem beyond a computationally infeasible brute-force search. On the other hand, there is some (admittedly, scant) evidence that \emph{unitary} YBOs are less abundant, and may admit a conceptual classification.

A Yang--Baxter operator $R$ on $V\otimes V$ produces a tower of unitary braid group
representations on $V^{\otimes n}$.  When the operator admits an
enhancement, these representations also give invariants of oriented links \cite{Turaev1988,KasselTuraev}.  An important problem is to determine the computational complexity of these invariants.

 Another motivation for our work comes from quantum computation. In the topological model for quantum computation \cite{RowellWangBAMS}, anyon exchange statistics in 2-dimensional media yield global braiding gates on Hilbert spaces modeling anyon configurations. Realizing these braiding gates directly on the circuit model, where gates act locally on a few adjacent copies of a Hilbert space $V$, motivates the \emph{localization} problem in \cite{RW}.  Such anyon configurations are modeled on morphism spaces in unitary braided fusion categories, yielding a mathematical conjecture relating localizability with a basic fusion category statistic known as FP-dimension. All of these provide ample motivation for classifying unitary Yang--Baxter operators.

In this paper we work towards a conceptual classification in terms of a three sources of unitary YBOs and some basic operations on them.  The classification must be interpreted up to an appropriate notion of equivalence, which we will come to presently.

Two familiar sources of unitary YBOs have their own extensive theories.  Monomial solutions
reduce to set-theoretic Yang--Baxter solutions together with multiplicative
cocycle data
\cite{EtingofSchedlerSoloviev,CarterElhamdadiSaito,GalindoRowellBVS,DS},
while group-type solutions arise from Yetter--Drinfeld structures over
discrete groups \cite{GalindoRowellBVS}.  We recall these sources only to fix
the terminology and the equivalence relation used in the paper.  Our main focus
is a third source consisting of solutions arising from twisted group-algebra
towers \cite{GustafsonKimballRowellZhang}.
This source contains the well-studied Gaussian and metaplectic examples
\cite{JonesCMP,GJ,RowellWenzl,GalindoRowellBVS}, but it also produces more
general families.

A twisted group-algebra tower does not by itself produce a matrix solution
of the Yang--Baxter equation; it should instead be viewed as an intermediate
algebraic construction.  The basic data are a finite abelian group, a
normalized \(2\)-cocycle, a bicharacter, and a coefficient function.  They
define an element at every level in a tower of twisted group algebras.
The braid relation becomes a finite identity for the coefficient function,
and unitarity becomes a finite coefficient condition.  A matrix
operator appears only after one chooses a compatible representation of the
tower on tensor powers.  This is the abelian part of the twisted
tensor-product construction of \cite{GustafsonKimballRowellZhang}.  It is
flexible enough to include nontrivial cocycle data while retaining exact
coefficient calculations.

\iffalse
  We denote by
\(\boxtimes\) the external product, which combines independent tower data
and becomes a tensor product of local operators after a shuffle of tensor
factors.  The internal operation \(\boxplus_t\) is the parameterized sum
\(a+t b\) of two elements in the same local algebra. Related operations \(\boxtimes\) and \(\boxplus\) for involutive
Yang--Baxter operators were introduced in \cite{LechnerPennigWood}.  In the
categorical formulation of \cite[Definition~5.2]{MartinRowellTorzewska},
\(\boxtimes\) is called the lashing product. When the mixed braid
identities and orthogonality hypotheses of
Proposition~\ref{prop:boxplus-unitarity} hold, the resulting unitary family
is called an admissible internal orthogonal sum.  An arbitrary sum need not
satisfy the braid relation.  The recursive closure process used below is
made precise in Definition~\ref{def:generated-tower-presentations}.
\fi

We make the following conjecture, the precise sense of the terms will be explained later:
\begin{conj}
\label{conj:unitary-YBO-classification}
Up to equivalence, every unitary Yang--Baxter operator is generated by monomial, group-type and twisted group-algebra solutions.

\end{conj}

The localization conjecture of Rowell--Wang
\cite[Conjecture~4.1]{RW} provides a criterion for localizing braid group representations associated with an anyon model, i.e., a braided fusion category:

\begin{conj}\label{conj:localization}
Let \(X\) be a simple object in a unitary braided fusion category.  The
associated sequence of braid group representations $(\rho_n^X,\End(X^{\otimes n}))$ is localizable by a
unitary Yang--Baxter operator if and only if
\[
  \operatorname{FPdim}(X)^2\in\Z.
\]
\end{conj}

The following conjecture was formulated in
\cite[Conjecture~1.1]{GalindoRowellBVS} as a refinement of one found in \cite{RW}; we restate it in the present
terminology.

\begin{conj}
\label{conj:YBO-virtual-abelian}
For every unitary Yang--Baxter operator \(R\) and every \(n\geq2\), the image
of the associated braid group representation
\[
  \rho_{R,n}:\B_n\longrightarrow\mathrm U(V^{\ot n})
\]
is virtually abelian.
\end{conj}

The results below provide evidence for these three
conjectures.  For the first, the uniform metric families come directly from
twisted-tower presentations, while the dimension-\(8\) cyclic families are
produced by admissible internal orthogonal sums.  Within the declared
low-dimensional search spaces, the qutrit representatives are identified
with signed Family I and Family II operators; after reduction to
\(\F_{1201}\), the remaining ququint orbits are matched by signed Family I
and Family \(\mathrm{II}_E\) data.  For the second, the qutrit Family II
operator strictly localizes the semisimple BMW trace-quotient tower, while
the quaternionic Family III operator strictly localizes the semisimple
Hecke braid-image tower associated with \(SU(3)_3\).  For the third, all
metric-family braid images are finite, and the cyclic-family braid images
are virtually abelian even at parameter values for which they are infinite.

The main uniform constructions start with a finite metric abelian group
\((A,q)\).  It gives three families built from its quadratic phases and
associated cocycle data.  Family I is the basic Gaussian construction and
has two polarization signs.  Family II uses two
coupled quadratic phases; in its tensor realization it shifts two labels in
opposite directions and preserves their sum.  A self-adjoint automorphism
of the metric group gives the reindexed Family \(\mathrm{II}_E\) of
Proposition~\ref{prop:family-II-coordinate-change}.  Family III is
different because its coefficients are constant, while the multiplication
in each group-algebra factor is nontrivially twisted.  Theorem~\ref{thm:metric-families-YBE}
and Corollary~\ref{cor:extended-metric-families-YBE} give unitary
tensor-power realizations for all these constructions.

The metric-group description also reveals structure that is less clear in
matrix form.  Orthogonal decompositions of \((A,q)\) induce factorizations
of the towers, local operators, braid representations, spectra, and
coefficient traces.  The primary factorization is canonical; a subsequent
factorization into indecomposable metric blocks requires a choice, and the
blocks must be stable under \(E\) for the reindexed Family
\(\mathrm{II}_E\); see
Corollary~\ref{cor:metric-primary-factorization} and
Remark~\ref{rem:metric-block-factorization}.  The local
operators also normalize generalized Pauli groups whose projective label
groups are finite.  Their projective images therefore lie in finite
Clifford extensions, and a determinant argument controls the scalar kernel.
Hence the ordinary braid images of the three
metric families, including the signed and reindexed extensions, are
finite by Theorem~\ref{thm:metric-full-finite} and
Corollary~\ref{cor:extended-metric-full-finite}.

The admissible internal orthogonal-sum construction produces a different
phenomenon.  Applied to the two-torsion sectors of the dimension-\(8\)
cyclic tower, it gives eight parametrized families arising from two
templates; see Proposition~\ref{prop:d8-families}.  For a unitary
parameter that is not a root of unity, a local generator has infinite
projective order; hence the corresponding braid image is infinite.
Nevertheless, all these braid images are virtually abelian.  The proof compares every
parameter value with a Gaussian base point and places the generator ratios
in a common commutative two-torsion algebra; see
Theorem~\ref{thm:cyclic-family-virtually-abelian}.  Thus the metric and
cyclic families support two different parts of the proposed picture.  The
metric families have finite image, while virtual abelianness survives in
the cyclic families, which are generically non-Gaussian in the fixed
dimension-\(8\) Pauli realization, even when their braid images are infinite.

Link invariants provide another test of these constructions.  The
identity-coefficient functional on each tower is a trace, and in the Pauli
realizations it agrees with the normalized ordinary matrix trace.
Theorem~\ref{thm:metric-family-enhancement} and
Corollary~\ref{cor:cyclic-family-enhancement} give scalar Turaev
enhancements for the metric and cyclic families.  Orthogonal decompositions of the metric
data become product formulas for the metric-family invariants.  Fourier
spectra give local skein relations, including a universal six-twist
relation for Family III\@.  For the three metric families, the invariants can
also be evaluated without expanding matrices of exponentially growing
size.  For fixed metric data, the trace of a braid word reduces to a finite
quadratic Gauss sum over the kernel of a
homomorphism determined by the word.  Theorem~\ref{thm:metric-efficient-evaluation}
therefore gives an exact deterministic algorithm whose running time is
polynomial in the number of strands and the word length.  General spectral
criteria identify two-eigenvalue Hecke and Temperley--Lieb quotients and
rank-one BMW quotients, while a separate faithful-tower criterion determines
when a tensor-power realization preserves the full braid-image algebra; see
Propositions~\ref{prop:two-eigenvalue-Hecke-criterion},~\ref{prop:rank-one-BMW-criterion},
and~\ref{prop:faithful-tower-localization-criterion}.

The low-dimensional calculations provide evidence within specified search
spaces rather than a classification of all unitary Yang--Baxter operators.
For \(d\in\{2,3,5,6,7,9,10,11\}\),
Computation~\ref{comp:low-dimensional-Gaussian-search} reports only Gaussian
points in the compact unitary locus of the specified full-support ansatz.
Separately, Computation~\ref{comp:d8-family-completeness} reports the
positive-dimensional unitary components in dimension \(8\) as the cyclic
families constructed in Section~\ref{sec:cyclic-families}.  The
dimension-\(4\) polynomial-in-Pauli solutions are omitted because a Zak-basis
change makes them monomial.
Likewise, we do not perform Clifford searches in local dimensions \(4\) or
\(9\): in these square dimensions, the Zak basis makes the relevant
Clifford representations monomial, so every Clifford Yang--Baxter operator
is locally equivalent to a monomial solution
\cite{ApplebyBengtssonBrierleyGrasslGrossLarsson}.
Under the declared reduction pipeline, the two-qutrit search in
Computation~\ref{comp:qutrit-reduction} singles out two explicit
representatives that are not identified as monomial or group type by the
declared tests.  One is a signed Family I operator that factors through
Temperley--Lieb and gives a HOMFLYPT specialization.  The other is a Family
II operator that satisfies the rank-one BMW relations; its normalized
BMW/Kauffman invariant is the square of a Jones specialization at a
primitive sixth root of unity.  These
identifications and invariant calculations are proved directly in
Theorem~\ref{thm:R1-HOMFLYPT}, Corollary~\ref{cor:R1-Temperley-Lieb},
Proposition~\ref{prop:R3-BMW},
Corollary~\ref{cor:R3-BMW-localization}, and
Theorem~\ref{thm:R3-symmetric-square}.

The smallest Family III specialization has a quaternionic twisted-group
description.  Its explicit \(16\times16\) operator has Hecke parameter
\(e^{\pi\ii/3}\), spectral density \(1/2\), and local dimension \(4\).  It
gives a faithful tensor-power realization of the semisimple Hecke braid-image
tower associated with \(SU(3)_3\); see
Theorems~\ref{thm:strict-ybo}
and~\ref{thm:quaternionic-Lechner-localization} and
Corollary~\ref{cor:Lechner-minimal-dimension}.  These results prove that
Lechner's exceptional family \([e^{\pi\ii/3},1/2,2m]\), \(m\geq2\), is
nonempty in its smallest possible local dimension, namely \(4\); see
\cite[Section~3]{LechnerTwoEigenvalues}.  In the two-ququint
calculation, exact characteristic-zero enumeration is followed by
finite-field group-type and gauge tests.  After reduction to
\(\F_{1201}\), the remaining output is matched by the two signed
cyclic Gaussian branches and by Family \(\mathrm{II}_E\) product solutions.
This comparison is a finite-field computational observation and is not a
characteristic-zero classification.

An outline of the manuscript's structure is as follows:  Sections~\ref{sec:background}
and~\ref{sec:twisted-towers} introduce the basic classes and the tower
construction.
Sections~\ref{sec:gaussian} and~\ref{sec:cyclic-families} construct the
metric families and cyclic families that are generically non-Gaussian in
their fixed Pauli realizations.  Section~\ref{sec:conjectures}
makes precise the recursive closure process in
Conjecture~\ref{conj:unitary-YBO-classification}.
Sections~\ref{sec:trace-invariants},
\ref{sec:finite-image}, and~\ref{sec:cyclic-virtual-abelianness} study link
invariants, finite image, and virtual abelianness.  The low-dimensional
Pauli and Clifford calculations are followed in
Section~\ref{sec:spectral-quotients} by general criteria for recognizing
Temperley--Lieb and BMW quotients and strict localizations.  The final
sections apply these criteria to the qutrit outputs and the quaternionic
specialization and give the two-ququint comparison.
Section~\ref{sec:conclusion} summarizes the proved results and the remaining
classification problems.

\section{Basic classes of Yang--Baxter operators}
\label{sec:background}

Let \(V\) be a finite-dimensional complex Hilbert space, and write \(I\)
for the identity operator on the space determined by context.  A
\emph{unitary Yang--Baxter operator} on \(V\) is a unitary map
\(R\in\End(V\otimes V)\) such that
\begin{equation}\label{eq:background-YBE}
  (R\otimes I)(I\otimes R)(R\otimes I)
  =
  (I\otimes R)(R\otimes I)(I\otimes R).
\end{equation}
For \(n\geq1\), let \(\B_n\) denote the braid group on \(n\) strands, with
\(\B_1=\{1\}\); for \(n\geq2\), its Artin generators are
\(\sigma_1,\ldots,\sigma_{n-1}\).  The operator \(R\) defines the trivial
representation at \(n=1\) and the unitary representations
\[
  \rho_{R,n}:\B_n\longrightarrow\mathrm U(V^{\ot n}),
  \qquad
  \rho_{R,n}(\sigma_i)
  =I^{\ot(i-1)}\ot R\ot I^{\ot(n-i-1)}
  \qquad(n\geq2,\ 1\leq i<n).
\]
Throughout the paper, unless explicitly stated otherwise, we regard two
operators as equivalent if they differ by a unit scalar and a local unitary
change of basis as follows.
\begin{equation}\label{eq:unitary-local-equivalence}
  R\longmapsto
  \lambda(S\otimes S)^{-1}R(S\otimes S),
  \qquad \lambda\in\Uone,
  \quad S\in\mathrm U(V).
\end{equation}
We next recall the terminology for monomial and group-type solutions.

\subsection{Monomial solutions}
\label{subsec:background-monomial}

Let
\[
  V=\bigoplus_{i\in I}V_i
\]
be an orthogonal decomposition into one-dimensional subspaces.  We say that
\(R\) is \emph{monomial with respect to this decomposition} if there is a
bijection
\[
  r:I\times I\longrightarrow I\times I,
  \qquad
  r(i,j)=\bigl(r_1(i,j),r_2(i,j)\bigr),
\]
such that
\[
  R(V_i\otimes V_j)
  =V_{r_1(i,j)}\otimes V_{r_2(i,j)}
  \qquad(i,j\in I).
\]
The Yang--Baxter equation for \(R\) implies
\[
  r_{12}r_{23}r_{12}=r_{23}r_{12}r_{23}
\]
on \(I^3\).  Hence \(r\) is a set-theoretic Yang--Baxter solution.

Choose a unit vector \(e_i\in V_i\) for every \(i\in I\).  There is then a
unique function \(c:I\times I\to\Uone\) satisfying
\[
  R(e_i\otimes e_j)
  =c(i,j)e_{r_1(i,j)}\otimes e_{r_2(i,j)}.
\]
The scalar part of the braid relation requires \(c\) to be a multiplicative
Yang--Baxter \(2\)-cocycle.  Rescaling the chosen basis changes \(c\) by a
coboundary; hence the scalar data are naturally cohomological.

With respect to the product basis
\(\{e_i\otimes e_j\mid i,j\in I\}\), let \(P_r\) be the permutation matrix
determined by \(r\).  The operator can then be written
\[
  R=DP_r,
\]
where \(D\) is diagonal and its entry on the output line indexed by
\(r(i,j)\) is \(c(i,j)\).  Thus a monomial Yang--Baxter operator consists of
a set-theoretic solution together with a multiplicative cocycle class
\cite{EtingofSchedlerSoloviev,CarterElhamdadiSaito,GalindoRowellBVS,DS}.

Monomiality depends on the chosen orthogonal decomposition, or equivalently
on the associated product basis.  We therefore say that a local equivalence
class is \emph{of monomial type} if it contains a monomial representative.

\subsection{Group-type solutions}
\label{subsec:background-group-type}

Group-type solutions arise from Yetter--Drinfeld data over discrete groups.
At the matrix level, a solution is of \emph{group type} if there is
a basis \(x_1,\ldots,x_d\) of \(V\) and invertible maps
\(g_1,\ldots,g_d\in\GL(V)\) such that
\begin{equation}\label{eq:background-group-type}
  R(x_j\otimes z)=g_j(z)\otimes x_j
  \qquad(1\leq j\leq d,\ z\in V).
\end{equation}
If the basis is orthonormal and every \(g_j\) is unitary, then \(R\) is
unitary.  As in the monomial case, we say that a local class is of group type
if it has a representative of the form \eqref{eq:background-group-type}.

More explicitly, let \(G=\langle g_1,\ldots,g_d\rangle\) and put
\[
  V_g=\operatorname{span}\{x_j:g_j=g\}
  \qquad(g\in G).
\]
The braid relation is equivalent to
\[
  g_i(V_{g_j})\subseteq V_{g_i g_j g_i^{-1}}
  \qquad(1\leq i,j\leq d).
\]
This is the Yetter--Drinfeld compatibility condition; see
\cite[Definition~4.1 and Proposition~4.2]{GalindoRowellBVS}.

The monomial and group-type classes are not disjoint.  A group-type
operator is monomial whenever the maps \(g_j\) are simultaneously monomial
in the chosen basis.

\section{Twisted group-algebra towers and two operations}
\label{sec:twisted-towers}

The third source in our conjectural classification comes from the iterated
twisted tensor-product construction of
\cite{GustafsonKimballRowellZhang}.  We develop its abelian specialization
in the language of twisted group-algebra towers.  This formulation converts
braid compatibility into finite algebraic data and provides natural
operations for combining solutions.

\subsection{Braid-compatible elements}
\label{subsec:local-braid-data}

Let
\[
  \C=\A_1\subset \A_2\subset\A_3\subset\cdots
\]
be a tower of unital complex algebras.  The braid groups are themselves a
tower under the inclusions
\[
  \B_n\hookrightarrow\B_{n+1},
  \qquad
  \sigma_i\longmapsto\sigma_i,
\]
obtained by adding one strand at the right.  We use the same compatibility
on the algebra side.

A \emph{braid-compatible family} in \((\A_n)\) is a sequence of invertible
elements
\[
  r_i\in \A_{i+1},\qquad i\geq1,
\]
such that, for every \(n\ge2\), their images in \(\A_n\) satisfy
\begin{align}
  r_ir_{i+1}r_i&=r_{i+1}r_ir_{i+1}
  \qquad(1\leq i<n-1), \label{eq:local-braid}\\
  r_ir_j&=r_jr_i
  \qquad(1\leq i,j<n,\ |i-j|>1). \label{eq:local-distant}
\end{align}
The same symbol \(r_i\) denotes its image under the inclusion
\(\A_{i+1}\subset\A_n\) whenever \(n\ge i+1\).  Then
\(\sigma_i\mapsto r_i\) defines homomorphisms
\[
  \rho_n^{\A}:\B_n\longrightarrow \A_n^\times,
  \qquad n\ge2,
\]
and these homomorphisms are compatible with the inclusions
\(\B_n\subset\B_{n+1}\) and \(\A_n\subset\A_{n+1}\).

Suppose now that \(V\) is a finite-dimensional vector space and that the
tower admits representations
\[
  \Phi_n:\A_n\longrightarrow \End(V^{\ot n})
  \qquad(n\geq1)
\]
compatible with the inclusions.  We say that the family is realized by
\(R\in\End(V\ot V)\) if
\[
  \Phi_n(r_i)
  =
  I_V^{\ot(i-1)}\ot R\ot I_V^{\ot(n-i-1)}
  \qquad(n\geq2,\ 1\leq i<n).
\]
In this case \(R\) satisfies the Yang--Baxter equation
\[
  (R\ot I)(I\ot R)(R\ot I)
  =
  (I\ot R)(R\ot I)(I\ot R).
\]
Thus the algebraic problem of producing braid-compatible families precedes the
matrix problem of producing Yang--Baxter operators.

\subsection{Twisted group-algebra towers}

Let \(G\) be a finite abelian group, written additively.  Fix a normalized
unitary \(2\)-cocycle
\[
  \nu:G\times G\longrightarrow\Uone,
\]
that is,
\begin{equation}\label{eq:cocycle-identity}
  \nu(g,h)\nu(g+h,k)=\nu(h,k)\nu(g,h+k),
  \qquad
  \nu(0,g)=\nu(g,0)=1,
\end{equation}
and let
\[
  \tau:G\times G\longrightarrow\Uone
\]
be a bicharacter.  No compatibility between \(\nu\) and \(\tau\) is
required beyond these hypotheses.

For \(n\geq1\), set \(E_n(G)=G^{n-1}\), with \(E_1(G)=\{0\}\).  For
\(\mathbf g=(g_1,\ldots,g_{n-1})\) and
\(\mathbf h=(h_1,\ldots,h_{n-1})\), define
\[
  \kappa_n^\tau(\mathbf g,\mathbf h)
  =
  \prod_{i=1}^{n-2}\tau(h_i,g_{i+1})^{-1}
\]
and
\begin{equation}\label{eq:tower-cocycle}
  \Omega_n^{\nu,\tau}(\mathbf g,\mathbf h)
  =
  \prod_{i=1}^{n-1}\nu(g_i,h_i)
  \kappa_n^\tau(\mathbf g,\mathbf h).
\end{equation}
An empty product is understood to be \(1\).
The product of the coordinate cocycles \(\nu\) is a normalized
\(2\)-cocycle on \(E_n(G)\).  Moreover, \(\kappa_n^\tau\) is
multiplicative in each variable, hence is itself a normalized
\(2\)-cocycle.  Consequently \(\Omega_n^{\nu,\tau}\) is a normalized
unitary \(2\)-cocycle.  We define
\[
  \T_n(G;\nu,\tau)
  =
  \C_{\Omega_n^{\nu,\tau}}[E_n(G)],
\]
the twisted group algebra with basis
\(\{U^{\mathbf g}:\mathbf g\in E_n(G)\}\) and product
\[
  U^{\mathbf g}U^{\mathbf h}
  =
  \Omega_n^{\nu,\tau}(\mathbf g,\mathbf h)
  U^{\mathbf g+\mathbf h}.
\]
For \(n=1\), this gives \(\T_1(G;\nu,\tau)=\C\).  The map
\(\mathbf g\mapsto(\mathbf g,0)\) identifies \(E_n(G)\) with the subgroup
\(E_n(G)\times\{0\}\subset E_{n+1}(G)\).  The cocycle
\(\Omega_{n+1}^{\nu,\tau}\) restricts to \(\Omega_n^{\nu,\tau}\) on this
subgroup, and hence the corresponding basis inclusion gives compatible
unital algebra embeddings
\[
  \T_n(G;\nu,\tau)\subset\T_{n+1}(G;\nu,\tau).
\]
Thus these algebras form a tower.

For \(1\leq i<n\), let \(\epsilon_i(g)\in E_n(G)\) be the element with
\(g\) in the \(i\)-th position and \(0\) elsewhere, and set
\[
  U_i^g=U^{\epsilon_i(g)}.
\]
Then the twisted group algebra product is equivalently described by the
straightening relations
\begin{align}
  U_i^gU_i^h&=\nu(g,h)U_i^{g+h},\label{eq:tower-same}\\
  U_i^gU_j^h&=U_j^hU_i^g\qquad(|i-j|>1),\label{eq:tower-far}\\
  U_i^gU_{i+1}^h&=\tau(g,h)U_{i+1}^hU_i^g. \label{eq:tower-adjacent}
\end{align}
Thus
\[
  U^{\mathbf g}=U_1^{g_1}\cdots U_{n-1}^{g_{n-1}}
\]
is the normal ordered basis.  When \(n=2\), we write \(U^g=U_1^g\).

For an arbitrary unitary cocycle \(\nu\), the canonical involution is
obtained by extending conjugate-linearly the formula
\[
  (U^{\mathbf g})^*
  =
  \Omega_n^{\nu,\tau}(\mathbf g,-\mathbf g)^{-1}U^{-\mathbf g}.
\]
The tower embeddings preserve these involutions.

\begin{remark}\label{rem:cocycle-gauge}
Let \(\lambda:G\to\Uone\) satisfy \(\lambda(0)=1\), and write
\[
  (\delta\lambda)(g,h)
  =
  \frac{\lambda(g)\lambda(h)}{\lambda(g+h)}.
\]
If \(\nu'=\nu\,\delta\lambda\), then the assignments
\begin{equation}\label{eq:cocycle-gauge}
  \Phi_{\lambda,n}:\T_n(G;\nu',\tau)\longrightarrow
  \T_n(G;\nu,\tau),
  \qquad
  U_i^g\longmapsto \lambda(g)U_i^g,
\end{equation}
are compatible algebra isomorphisms and preserve the canonical involutions.
Indeed, the same-position relation is preserved by the identity
\[
  \nu'(g,h)\lambda(g+h)
  =
  \nu(g,h)\lambda(g)\lambda(h),
\]
while the adjacent and distant relations are unchanged.  The inverse is
obtained from \(\lambda^{-1}\), and the involution formula shows
that these maps are \(*\)-preserving.  For any function \(f:G\to\C\), the
same map sends
\[
  \sum_{g\in G}f(g)U_i^g
  \quad\text{to}\quad
  \sum_{g\in G}\lambda(g)f(g)U_i^g.
\]
Thus the tower and its coefficient problem depend only on the cohomology
class of \(\nu\), provided the coefficients are transformed at the same
time.

For a finite abelian group, every \(\Uone\)-valued \(2\)-cocycle is
cohomologous to a bicharacter.  Let \(\operatorname{Bil}(G)\),
\(\operatorname{Sym}(G)\), and \(\operatorname{Alt}(G)\) denote the groups
of bicharacters, symmetric bicharacters, and alternating bicharacters on
\(G\), respectively.  Antisymmetrization induces
the isomorphisms
\[
  \frac{\operatorname{Bil}(G)}{\operatorname{Sym}(G)}
  \cong H^2(G,\Uone)
  \cong \operatorname{Alt}(G),
  \qquad
  [\nu]\longmapsto
  \mathcal A_\nu(g,h)=\frac{\nu(g,h)}{\nu(h,g)};
\]
see \cite[Section~2.2]{GalindoClifford}.  The bicharacter representative is
not canonical, but its commutator bicharacter \(\mathcal A_\nu\) is
independent of the gauge.
\end{remark}

By Remark~\ref{rem:cocycle-gauge}, we may take \(\nu\) to be a unitary
bicharacter without loss of generality.  In this gauge
\(\Omega_n^{\nu,\tau}\) is itself a unitary bicharacter on \(E_n(G)\), and
the canonical involution takes the simple form
\begin{equation}\label{eq:twisted-star}
  (U^{\mathbf g})^*=(U^{\mathbf g})^{-1}
  =
  \Omega_n^{\nu,\tau}(\mathbf g,\mathbf g)U^{-\mathbf g}.
\end{equation}
In particular,
\[
  (U_i^g)^*=\nu(g,g)U_i^{-g}.
\]

We use the abbreviation
\begin{equation}\label{eq:nu-trivial-convention}
  \T_n(G,\tau):=\T_n(G;1,\tau)
\end{equation}
when \(\nu=1\).

\begin{remark}
The tower \(\T_\bullet(G;\nu,\tau)\) is the abelian specialization of the
iterated twisted tensor-product construction in
\cite{GustafsonKimballRowellZhang}.  Its local algebra is
\(\C_\nu[G]\), and the twisting between adjacent factors is
\[
  U^h\ot U^g\longmapsto \tau(g,h)U^g\ot U^h.
\]
\end{remark}

\subsection{The coefficient equation}
\label{subsec:coefficient-functions}

For a function \(f:G\to\C\), define
\[
  r_f=\sum_{g\in G}f(g)U^g\in\T_2(G;\nu,\tau),
  \qquad
  r_{f,i}=\sum_{g\in G}f(g)U_i^g\in\T_n(G;\nu,\tau)
  \quad(1\leq i<n).
\]
The problem is to decide when the elements \(r_{f,i}\) satisfy the braid
relations.  The distant braid relations are automatic from
\eqref{eq:tower-far}; the adjacent braid relation is the condition
\[
  r_{f,1}r_{f,2}r_{f,1}=r_{f,2}r_{f,1}r_{f,2}
  \quad\text{in }\T_3(G;\nu,\tau).
\]
When this holds, we say that \(f\) satisfies the
\((G;\nu,\tau)\) coefficient Yang--Baxter equation.

\begin{lemma}
\label{lem:coefficient-ybe}
The function \(f:G\to\C\) satisfies the \((G;\nu,\tau)\) coefficient
Yang--Baxter equation if and only if, for every \(s,t\in G\),
\begin{equation}\label{eq:coefficient-ybe}
  f(t)\sum_{c\in G}
  f(s-c)f(c)\nu(s-c,c)\tau(c,t)^{-1}
  =
  f(s)\sum_{a\in G}
  f(a)f(t-a)\nu(a,t-a)\tau(s,a)^{-1}.
\end{equation}
\end{lemma}

\begin{proof}
The left-hand side of the braid relation is
\[
  r_{f,1}r_{f,2}r_{f,1}
  =
  \sum_{a,b,c\in G}f(a)f(b)f(c)U_1^aU_2^bU_1^c.
\]
Using \(U_2^bU_1^c=\tau(c,b)^{-1}U_1^cU_2^b\), the coefficient of
\(U_1^sU_2^t\) is
\[
  f(t)\sum_{c\in G}
  f(s-c)f(c)\nu(s-c,c)\tau(c,t)^{-1}.
\]
Similarly,
\[
  r_{f,2}r_{f,1}r_{f,2}
  =
  \sum_{a,b,c\in G}f(a)f(b)f(c)U_2^aU_1^bU_2^c,
\]
and reordering gives
\[
  U_2^aU_1^bU_2^c
  =
  \tau(b,a)^{-1}\nu(a,c)U_1^bU_2^{a+c}.
\]
Thus the coefficient of \(U_1^sU_2^t\) is
\[
  f(s)\sum_{a\in G}
  f(a)f(t-a)\nu(a,t-a)\tau(s,a)^{-1}.
\]
The monomials \(U_1^sU_2^t\) form a basis of
\(\T_3(G;\nu,\tau)\).  Therefore equality of the two products is equivalent to
\eqref{eq:coefficient-ybe}.
\end{proof}

\noindent
To express the unitarity of \(r_f\), for \(f:G\to\C\) set
\begin{equation}\label{eq:coefficient-unitarity-sums}
  C_f^{\nu}(t)
  =
  \sum_{g\in G}\overline{f(g)}f(g+t)\nu(g,t)^{-1}
  \qquad(t\in G).
\end{equation}
We call \(r_f\) \emph{projectively unitary} if
\(r_f^*r_f=c\,1\) for some real scalar \(c>0\).

\begin{lemma}
\label{lem:coefficient-unitarity}
Let \(\nu\) be a unitary bicharacter on the finite abelian group \(G\).
Then the coefficient element \(r_f\in\C_\nu[G]\) satisfies
\[
  r_f^*r_f
  =
  \sum_{t\in G}C_f^{\nu}(t)U^t.
\]
In particular, \(r_f\) is unitary if and only if
\(C_f^{\nu}(t)=\delta_{t,0}\).  If \(f\neq0\), then \(r_f\) is
projectively unitary if and only if
\[
  C_f^{\nu}(t)=c_f\delta_{t,0},
  \qquad
  c_f=\sum_{g\in G}|f(g)|^2>0.
\]
In that case \(c_f^{-1/2}r_f\) is unitary.
\end{lemma}

\begin{proof}
By bilinearity,
\[
  (U^g)^*U^{g+t}
  =
  \nu(g,g)\nu(-g,g+t)U^t
  =
  \nu(g,t)^{-1}U^t.
\]
Consequently,
\[
  r_f^*r_f
  =
  \sum_{g,t\in G}
  \overline{f(g)}f(g+t)\nu(g,t)^{-1}U^t
  =
  \sum_{t\in G}C_f^\nu(t)U^t.
\]
Since \(\{U^t:t\in G\}\) is a basis and \(U^0=1\), (projective) unitarity
is equivalent to the corresponding coefficient conditions.  Moreover,
\[
  C_f^\nu(0)=\sum_{g\in G}|f(g)|^2.
\]
Finally, in a finite-dimensional unital algebra a left inverse is a
two-sided inverse, and therefore \(r_f^*r_f=c_f\,1\) also gives
\(r_fr_f^*=c_f\,1\).
\end{proof}

\begin{cor}\label{cor:coefficient-unitary-braid-criterion}
The elements \(\{r_{f,i}\}\) form a unitary braid-compatible family if and
only if \(f\) satisfies the coefficient equation
\eqref{eq:coefficient-ybe} and
\[
  C_f^{\nu}(t)=\delta_{t,0}
  \qquad(t\in G).
\]
They form a projectively unitary braid-compatible family if and only if
\(f\neq0\), equation \eqref{eq:coefficient-ybe} holds, and
\(C_f^{\nu}(t)=c_f\delta_{t,0}\), where
\(c_f=\sum_{g\in G}|f(g)|^2\).
\end{cor}

\begin{proof}
The adjacent braid relation is Lemma~\ref{lem:coefficient-ybe}, the
distant relations are automatic, and (projective) unitarity is
Lemma~\ref{lem:coefficient-unitarity}.
\end{proof}

When \(\nu=1\), we call the
\((G;1,\tau)\) coefficient Yang--Baxter equation the \((G,\tau)\)
coefficient equation and write \(C_f=C_f^1\).

\subsection{Operations on tower solutions}
\label{subsec:tower-operations}

\subsubsection{External products and monoidality}

If \(\A_\bullet\) and \(\A'_\bullet\) are towers, write
\[
  (\A_\bullet\otimes_{\mathrm{lev}}\A'_\bullet)_n
  =\A_n\otimes\A'_n
\]
for their levelwise tensor product.  For a twisted group-algebra tower, let
\[
  x=\sum_{\mathbf g\in E_n(G)}a_{\mathbf g}U^{\mathbf g}
  \in\T_n(G;\nu,\tau).
\]
Define
\[
  \varepsilon_n(x)=a_{\mathbf 0}.
\]

Let \(G_1,G_2\) be finite abelian groups, let
\(\nu_j:G_j\times G_j\to\Uone\) be normalized unitary \(2\)-cocycles, and
let \(\tau_j:G_j\times G_j\to\Uone\) be bicharacters.  On
\(G=G_1\oplus G_2\), set
\begin{align*}
  \nu((g_1,g_2),(h_1,h_2))
  &=\nu_1(g_1,h_1)\nu_2(g_2,h_2),\\
  \tau((g_1,g_2),(h_1,h_2))
  &=\tau_1(g_1,h_1)\tau_2(g_2,h_2).
\end{align*}
For \(j=1,2\), let \(f_j:G_j\to\C\), and define
\[
  (f_1\boxtimes f_2)(g_1,g_2)=f_1(g_1)f_2(g_2).
\]

\begin{prop}
\label{prop:tower-external-product}
The assignments
\begin{equation}\label{eq:tower-external-product-map}
  \Phi_\bullet:\T_\bullet(G;\nu,\tau)
  \longrightarrow
  \T_\bullet(G_1;\nu_1,\tau_1)
  \otimes_{\mathrm{lev}}
  \T_\bullet(G_2;\nu_2,\tau_2),
  \qquad
  U_i^{(g_1,g_2)}\longmapsto U_i^{g_1}\otimes U_i^{g_2}.
\end{equation}
give a natural isomorphism of towers of unital \(*\)-algebras.  Under this
isomorphism,
\begin{equation}\label{eq:tower-external-product-braid-element}
  \Phi_n(r_{f_1\boxtimes f_2,i})
  =r_{f_1,i}\otimes r_{f_2,i}
  \qquad(n\geq2,\ 1\leq i<n).
\end{equation}
If \(\varepsilon_n^{(j)}\) is the map defined above for the \(j\)-th factor,
then
\begin{equation}\label{eq:tower-external-product-trace}
  \varepsilon_n
  =(\varepsilon_n^{(1)}\otimes\varepsilon_n^{(2)})\circ\Phi_n
  \qquad(n\geq1).
\end{equation}
\end{prop}

\begin{proof}
The identification
\[
  E_n(G_1\oplus G_2)
  \cong E_n(G_1)\oplus E_n(G_2)
\]
sends \(\mathbf g\) to
\((\mathbf g^{(1)},\mathbf g^{(2)})\).  Formula
\eqref{eq:tower-cocycle} then gives
\[
  \Omega_n^{\nu,\tau}(\mathbf g,\mathbf h)
  =
  \Omega_n^{\nu_1,\tau_1}
    (\mathbf g^{(1)},\mathbf h^{(1)})
  \Omega_n^{\nu_2,\tau_2}
    (\mathbf g^{(2)},\mathbf h^{(2)}).
\]
Hence the basis map
\(U^{\mathbf g}\mapsto
U^{\mathbf g^{(1)}}\otimes U^{\mathbf g^{(2)}}\)
preserves multiplication and the canonical involution, and it is compatible
with the embeddings \(\mathbf g\mapsto(\mathbf g,0)\).  This proves the tower
isomorphism.
Summing the local basis map with the product coefficients gives
\eqref{eq:tower-external-product-braid-element}.  The only basis element
mapped to the tensor product of two identity elements is the identity itself,
which proves \eqref{eq:tower-external-product-trace}.
\end{proof}

If the two towers have compatible tensor-power realizations on
\(V_1^{\ot n}\) and \(V_2^{\ot n}\), equation
\eqref{eq:tower-external-product-braid-element} and the canonical shuffle
\[
  V_1^{\ot n}\otimes V_2^{\ot n}
  \cong (V_1\otimes V_2)^{\ot n}
\]
give a realization of \(r_{f_1\boxtimes f_2}\).  If the two local operators
are \(R_1\) and \(R_2\), we denote the resulting operator by
\(R_1\boxtimes R_2\).

\begin{cor}\label{cor:coefficient-products}
If each \(f_j\) satisfies its coefficient Yang--Baxter equation, then
\(f_1\boxtimes f_2\) satisfies the coefficient equation for the direct-sum
data.  If the two local coefficient elements are unitary, their external
product is also unitary.
\end{cor}

\begin{proof}
Under \(\Phi_\bullet\), every adjacent braid relation becomes the tensor
product of the two corresponding braid relations; distant relations behave
the same way.  Equation \eqref{eq:tower-external-product-braid-element} and
Lemma~\ref{lem:coefficient-ybe} give the coefficient equation.  A tensor
product of unitary elements is unitary.
\end{proof}

\subsubsection{Direct sums}

Let \(R\in\End(V\otimes V)\) and \(S\in\End(W\otimes W)\) be
Yang--Baxter operators.  With respect to the decomposition
\[
 (V\oplus W)^{\otimes2}
 = (V\otimes V)\oplus(V\otimes W)
   \oplus(W\otimes V)\oplus(W\otimes W),
\]
define their \emph{direct sum} \(R\boxplus S\) by
\begin{equation}\label{eq:YBO-direct-sum}
 (R\boxplus S)(x\otimes y)
 =
 \begin{cases}
   R(x\otimes y),&x,y\in V,\\
   S(x\otimes y),&x,y\in W,\\
   y\otimes x,&x\in V,\ y\in W\text{ or }x\in W,\ y\in V.
 \end{cases}
\end{equation}
Thus \(R\boxplus S\) acts by the given operators on the two pure summands
and by the ordinary flip on the mixed summands.  For involutive operators,
this is the \(\boxplus\)-construction of \cite{LechnerPennigWood}.

\begin{prop}\label{prop:YBO-direct-sum}
The operator \(R\boxplus S\) is a Yang--Baxter operator on \(V\oplus W\).
If \(R\) and \(S\) are unitary, then \(R\boxplus S\) is unitary.
\end{prop}

\begin{proof}
Decompose \((V\oplus W)^{\otimes3}\) into the eight tensor-product
summands indexed by words of length three in \(V\) and \(W\).  On
\(V^{\otimes3}\) and \(W^{\otimes3}\), the braid relation is that of
\(R\) and \(S\), respectively.  On each mixed summand, both sides of the
braid relation move the unique factor of one type past the other two by
flips and apply the appropriate one of \(R\) or \(S\) once to the two
factors of the same type.  Hence the two sides agree on every summand.

The four summands in \eqref{eq:YBO-direct-sum} are mutually orthogonal.
The restrictions to the pure summands are \(R\) and \(S\), while the mixed
summands are exchanged by a unitary flip.  Thus \(R\boxplus S\) is unitary
whenever \(R\) and \(S\) are unitary.
\end{proof}

\subsubsection{The internal sum operation}

For this subsection, a \emph{local tower} is a tower
\(\A_1\subset\A_2\subset\cdots\) equipped, for every \(i\geq1\), with a
unital algebra homomorphism
\[
  \lambda_i:\A_2\longrightarrow\A_{i+1},
  \qquad x\longmapsto x_i,
  \qquad \lambda_1=\id_{\A_2}.
\]
We view \(x_i\) in every \(\A_n\) with \(n\geq i+1\) by the tower inclusions and require
that copies with disjoint supports commute.
\begin{equation}\label{eq:local-tower-distant-commutation}
  x_i y_j=y_jx_i
  \qquad (x,y\in\A_2,\ |i-j|>1).
\end{equation}
If the tower has an involution, the tower inclusions and every \(\lambda_i\)
are also required to be \(*\)-homomorphisms.
Twisted group-algebra towers and tensor-power towers have this local
structure.  For \(a,b\in\A_2\), write \(a_i,b_i\) for their local copies.

\begin{definition}\label{def:boxplus}
For \(a,b\in\A_2\) and \(t\in\C\), define their parameterized internal sum by
\begin{equation}\label{eq:boxplus-definition}
  a\boxplus_t b:=a+t b.
\end{equation}
\end{definition}

\begin{lemma}\label{lem:boxplus-braid-relations}
For \(a,b\in\A_2\), the local copies
\[
  r_i(t)=a_i\boxplus_t b_i
\]
satisfy all braid relations for every \(t\in\C\) if and only if, for every
\(i\geq1\),
\begin{align*}
 a_i a_{i+1}a_i&=a_{i+1}a_i a_{i+1},\\
 b_i b_{i+1}b_i&=b_{i+1}b_i b_{i+1},
\end{align*}
and
\begin{align*}
 b_i a_{i+1}a_i+a_i b_{i+1}a_i+a_i a_{i+1}b_i
 &=b_{i+1}a_i a_{i+1}+a_{i+1}b_i a_{i+1}
   +a_{i+1}a_i b_{i+1},\\
 a_i b_{i+1}b_i+b_i a_{i+1}b_i+b_i b_{i+1}a_i
 &=a_{i+1}b_i b_{i+1}+b_{i+1}a_i b_{i+1}
   +b_{i+1}b_i a_{i+1}
\end{align*}
hold in \(\A_{i+2}\).
\end{lemma}

\begin{proof}
For each \(i\), expand
\[
 r_i(t)r_{i+1}(t)r_i(t)
 -r_{i+1}(t)r_i(t)r_{i+1}(t)
\]
as a polynomial of degree at most three in \(t\).  Its coefficients of
degrees \(0\) and \(3\) are the braid relations for \(a\) and \(b\), and
the coefficients of degrees \(1\) and \(2\) are precisely the two mixed
identities in the statement.  The polynomial
vanishes identically exactly when these four coefficients vanish.  The
distant relations hold for every \(t\) by
\eqref{eq:local-tower-distant-commutation}.
\end{proof}

The internal sum \(\boxplus_t\) is not automatically braid-compatible.  The
necessary and sufficient identities are those of
Lemma~\ref{lem:boxplus-braid-relations}.  Orthogonality also gives unitarity.

\begin{prop}\label{prop:boxplus-unitarity}
Let \(a,b\) satisfy the identities in
Lemma~\ref{lem:boxplus-braid-relations}.  Suppose the local tower consists of
finite-dimensional unital \(*\)-algebras, and assume that there are
complementary projections
\(e_0,e_1\in\A_2\) and a scalar \(c>0\) such that
\begin{align*}
 a^*a=aa^*&=c e_0,&
 b^*b=bb^*&=c e_1,\\
 a^*b=b^*a&=0,&
 ab^*=ba^*&=0,
\end{align*}
with \(e_0+e_1=1\).  Then
\[
  r(t)=c^{-1/2}(a\boxplus_t b)
\]
has unitary local copies, and these copies form a braid-compatible family,
for every \(|t|=1\).  We call the resulting family an
\emph{admissible internal orthogonal sum} of \(a\) and \(b\).
\end{prop}

\begin{proof}
The orthogonality assumptions give
\[
 (a+t b)^*(a+t b)
 =a^*a+|t|^2b^*b
 =c(e_0+e_1)=c\,1
\]
when \(|t|=1\).  The analogous calculation on the other side gives
\((a+t b)(a+t b)^*=c\,1\).  Applying each unital
\(*\)-homomorphism \(\lambda_i\) shows that every local copy of \(r(t)\) is
unitary, hence invertible.  Lemma~\ref{lem:boxplus-braid-relations} gives the
braid relations; hence the family is braid-compatible.
\end{proof}

\section{Three families from finite metric groups}
\label{sec:gaussian}

We construct three families from finite metric data.  Families I and II use
\(\nu=1\) and are Gaussian, whereas Family III uses a nontrivial \(\nu\).

\subsection{Metric data and Gauss sums}
\label{subsec:metric-gauss-data}

Throughout this section finite abelian groups are written additively.  Fix a
\emph{finite metric abelian group} \((A,q)\), where \(A\) is a finite
abelian group and \(q:A\to\Uone\) satisfies
\[
  q(-a)=q(a)
\]
and
\begin{equation}\label{eq:chi-from-Q}
  \pair{a}{b}_q=\frac{q(a+b)}{q(a)q(b)}
\end{equation}
is a nondegenerate symmetric bicharacter on \(A\).  We suppress the subscript
\(q\) when no confusion can arise.  In particular, evaluating the
bicharacter at \((0,0)\) gives \(q(0)=1\).  The nondegeneracy condition gives
\begin{equation}\label{eq:Q-character-orthogonality}
  \sum_{a\in A}\pair{t}{a}
  =
  \begin{cases}
    |A|,&t=0,\\
    0,&t\neq0.
  \end{cases}
\end{equation}
For every \(a\in A\) and \(m\in\Z\), the quadratic identity gives
\begin{equation}\label{eq:quadratic-homogeneity}
  q(ma)=q(a)^{m^2}.
\end{equation}
Indeed, \eqref{eq:quadratic-homogeneity} follows by induction from
\[
  q(a+b)=q(a)q(b)\pair{a}{b}
\]
and \(1=q(a-a)=q(a)^2\pair{a}{a}^{-1}\).  Since \(A\) is finite,
\eqref{eq:quadratic-homogeneity} implies that all values of \(q\) are roots
of unity.

Let
\[
  \mathcal Q(q)=\{q\chi:\chi\in\Hom(A,\Uone)\}.
\]
We call the elements of \(\mathcal Q(q)\) \emph{quadratic phases with
polarization} \(\pair{\cdot}{\cdot}\).  Thus every \(Q\in\mathcal Q(q)\)
satisfies
\[
  Q(a+b)=Q(a)Q(b)\pair{a}{b},
\]
but it need not satisfy \(Q(-a)=Q(a)\).  The values of \(q\) are roots of
unity by \eqref{eq:quadratic-homogeneity}, and every character of the finite
group \(A\) has finite image.  Hence all values of every
\(Q\in\mathcal Q(q)\) are roots of unity.

There is a second, opposite-polarization torsor
\[
  \mathcal Q^{-}(q)
  =
  \{q^{-1}\chi:\chi\in\Hom(A,\Uone)\}.
\]
Thus, if \(\epsilon\in\{+1,-1\}\), we write
\[
  \mathcal Q^{\epsilon}(q)=
  \begin{cases}
    \mathcal Q(q),&\epsilon=+1,\\
    \mathcal Q^{-}(q),&\epsilon=-1,
  \end{cases}
\]
and every \(Q\in\mathcal Q^{\epsilon}(q)\) satisfies
\begin{equation}\label{eq:signed-quadratic-phase}
  Q(a+b)=Q(a)Q(b)\pair{a}{b}^{\epsilon}.
\end{equation}

Let \(B:A\times A\to\Uone\) be a nondegenerate symmetric bicharacter, and let
\(F:A\to\Uone\) satisfy
\[
  F(x+y)=F(x)F(y)B(x,y).
\]
With the normalization of \cite{Wall,Nikulin1980,TaylorGaussSums}, set
\[
  \gamma(F)=|A|^{-1/2}\sum_{a\in A}F(a).
\]

\begin{lemma}\label{lem:general-shifted-Gauss-sum}
For every \(u\in A\),
\begin{equation}\label{eq:general-shifted-Gauss-sum}
  |A|^{-1/2}\sum_{a\in A}F(a)B(u,a)
  =\gamma(F)F(u)^{-1}.
\end{equation}
In particular, \(|\gamma(F)|=1\).
\end{lemma}

\begin{proof}
Translation by \(u\) preserves the sum, and
\(F(a+u)=F(a)F(u)B(u,a)\).  This gives
\eqref{eq:general-shifted-Gauss-sum}.  Parseval's identity and character
orthogonality give \(|\gamma(F)|=1\).
\end{proof}

\begin{lemma}\label{lem:quadratic-phase-roots-of-unity}
Under the hypotheses of Lemma~\ref{lem:general-shifted-Gauss-sum}, every
value \(F(a)\), the normalized Gauss sum \(\gamma(F)\), and every shifted
value \(\gamma(F)F(u)^{-1}\) are roots of unity.
\end{lemma}

\begin{proof}
The symmetric bicharacter \(B\) admits a homogeneous quadratic refinement
\(q_B:A\to\Uone\) satisfying
\[
  q_B(x+y)=q_B(x)q_B(y)B(x,y),
  \qquad q_B(nx)=q_B(x)^{n^2}\quad(n\in\Z)
\]
\cite{Wall}\cite[Theorem~1.9]{TaylorGaussSums}.  Since \(F/q_B\) is a
character and \(B\) is nondegenerate, there is a unique \(r\in A\) such that
\[
  F(a)=q_B(a)B(r,a)\qquad(a\in A).
\]
If \(Na=0\), then \(q_B(a)^{N^2}=q_B(Na)=1\); the character values
\(B(r,a)\) also have finite order.  Thus every \(F(a)\) is a root of unity.
Applying
Lemma~\ref{lem:general-shifted-Gauss-sum} to \(q_B\) gives
\[
  \gamma(F)
  =|A|^{-1/2}\sum_{a\in A}q_B(a)B(r,a)
  =\gamma(q_B)q_B(r)^{-1}.
\]
Gauss--Milgram gives \(\gamma(q_B)^8=1\)
\cite{Nikulin1980}\cite[discussion following
Remark~1.14]{TaylorGaussSums}, while \(q_B(r)\) has finite order.
Hence \(\gamma(F)\), and therefore
\(\gamma(F)F(u)^{-1}\), are roots of unity.
\end{proof}

Apply Lemma~\ref{lem:general-shifted-Gauss-sum} to
\(Q\in\mathcal Q^\epsilon(q)\), with
\(B(a,b)=\pair{a}{b}^\epsilon\) and translation variable \(\epsilon u\).
This gives the signed shifted Gauss sum identity
\begin{equation}\label{eq:shifted-Gauss-sum}
  |A|^{-1/2}\sum_{a\in A}Q(a)\pair{u}{a}
  =
  \gamma(Q)Q(\epsilon u)^{-1}.
\end{equation}
When \(Q(-a)=Q(a)\) for all \(a\),
\(\gamma(Q)\) is the finite Weil index, also called the multiplicative
central charge.  We write \(\sign(A,Q)\in\Z/8\Z\) for the finite signature
(Brown invariant) of the nonsingular finite quadratic form; for a
discriminant form of an even lattice, this agrees with the lattice signature
modulo \(8\).  With these conventions, Gauss--Milgram gives
\[
  \gamma(Q)=\exp\left(\frac{2\pi\ii}{8}\sign(A,Q)\right)
\]
\cite{Nikulin1980}\cite[Section~2]{TaylorGaussSums}.

\subsection{Family I and the two signs of the Gaussian polarization}

Put
\[
  G_{\mathrm{I}}=A,
  \qquad
  \tau_{\mathrm{I}}(a,b)=\pair{a}{b}.
\]
For \(\epsilon\in\{+1,-1\}\) and
\(Q\in\mathcal Q^\epsilon(q)\), define
\[
  f^{\mathrm{I},\epsilon}_Q(a)=|A|^{-1/2}Q(a),
  \qquad
  r^{\mathrm{I},\epsilon}_{Q,i}
  =
  |A|^{-1/2}\sum_{a\in A}Q(a)U_i^a.
\]
For \(\epsilon=+1\), we use the shorter notation
\(r^{\mathrm I}_{Q,i}=r^{\mathrm{I},+}_{Q,i}\).

\begin{remark}\label{rem:cyclic-quantum-tori}
For odd \(N\), let \(A=\Z/N\Z\) and \(Q(a)=\zeta^{a^2}\), where
\(\zeta=e^{2\pi\ii/N}\).  Its polarization is
\(\pair{a}{b}=\zeta^{2ab}\).  Hence the Family I tower is the finite
quantum torus generated by \(U_i=U_i^1\) with
\[
  U_iU_{i+1}=\zeta^2\,U_{i+1}U_i.
\]
In the odd \(SO(N)_2\) case of Rowell--Wenzl \cite{RowellWenzl}, their
Gaussian element is
\[
  \frac{c}{\sqrt N}\sum_{j=0}^{N-1}\zeta^{j^2}u_i^{2j}
\]
for a unit scalar \(c\in\Uone\).  This is the Family I element after
identifying \(U_i=u_i^2\), where \(u_i\) denotes their quantum-torus
generator, up to the global phase \(c\).  Thus the cyclic
Rowell--Wenzl representation is a
categorical source for this branch of the construction.
\end{remark}

\begin{prop}\label{prop:gaussian-coefficients}
For \(\epsilon\in\{+1,-1\}\) and
\(Q\in\mathcal Q^\epsilon(q)\), the function
\(f^{\mathrm{I},\epsilon}_Q\) satisfies the
\((G_{\mathrm I},\tau_{\mathrm I})\) coefficient Yang--Baxter equation,
and the elements
\(r^{\mathrm{I},\epsilon}_{Q,i}\) are unitary and braid-compatible.  We call
the two branches \(\mathrm{Family\ I}^{+}\) and
\(\mathrm{Family\ I}^{-}\).
\end{prop}

\begin{proof}
For \(u\in A\), the sum \(C_{f^{\mathrm{I},\epsilon}_Q}(u)\) is
\[
  |A|^{-1}Q(u)\sum_{a\in A}\pair{u}{a}^{\epsilon}.
\]
Character orthogonality gives \(1\) at \(u=0\) and \(0\) otherwise.

For the coefficient equation, the signed quadratic identity gives
\[
  Q(s-c)Q(c)
  =Q(s)\pair{c}{c}^{\epsilon}\pair{s}{c}^{-\epsilon}.
\]
After the common scalar is canceled, the left- and right-hand sides of
\eqref{eq:coefficient-ybe} are
\[
  \sum_{c\in A}
  \pair{c}{c}^{\epsilon}
  \pair{s}{c}^{-\epsilon}
  \pair{t}{c}^{-1},
  \qquad
  \sum_{a\in A}
  \pair{a}{a}^{\epsilon}
  \pair{t}{a}^{-\epsilon}
  \pair{s}{a}^{-1}.
\]
For \(\epsilon=+1\), these sums are identical after renaming the variable.
For \(\epsilon=-1\), they agree after \(a=-c\).  The coefficient equation
and Lemma~\ref{lem:coefficient-unitarity} give the result.
\end{proof}

\subsection{Family II with two coupled quadratic phases}

For the second family and \(Q,P\in\mathcal Q(q)\), define
\[
  G_{\mathrm{II}}=A\oplus A,
  \qquad
  \tau_{\mathrm{II}}\bigl((t,c),(t',c')\bigr)
  =
  \pair{c}{t'}\pair{c'}{t}.
\]
The coefficient function is
\begin{equation}\label{eq:family-II-coefficient-function}
  f^{\mathrm{II}}_{Q,P}(t,c)
  =
  \overline{\gamma(P)}\,|A|^{-1}Q(t)P(c),
\end{equation}
Consequently, the braid element in
\(\T_n(G_{\mathrm{II}},\tau_{\mathrm{II}})\) is
\[
  r^{\mathrm{II}}_{Q,P,i}
  =
  \overline{\gamma(P)}\,|A|^{-1}
  \sum_{t,c\in A}Q(t)P(c)U_i^{(t,c)}.
\]

\begin{prop}\label{prop:family-II-reindexing}
For \(Q,P\in\mathcal Q(q)\), the coefficient function
\(f^{\mathrm{II}}_{Q,P}:A\oplus A\to\C\) satisfies the
\((G_{\mathrm{II}},\tau_{\mathrm{II}})\) coefficient Yang--Baxter equation,
and the elements \(r^{\mathrm{II}}_{Q,P,i}\) are
unitary and braid-compatible in the tower
\(\T_n(G_{\mathrm{II}},\tau_{\mathrm{II}})\).
\end{prop}

\begin{proof}
For \((u,v)\in A\oplus A\), the sum
\(C_{f^{\mathrm{II}}_{Q,P}}(u,v)\) is given by
\[
  \sum_{t,c\in A}
  \overline{f^{\mathrm{II}}_{Q,P}(t,c)}
  f^{\mathrm{II}}_{Q,P}(t+u,c+v)
\]
and factors as
\[
  |\gamma(P)|^2|A|^{-2}
  \left(\sum_{t\in A}\overline{Q(t)}Q(t+u)\right)
  \left(\sum_{c\in A}\overline{P(c)}P(c+v)\right),
\]
which is \(1\) if \((u,v)=0\) and \(0\) otherwise.

It remains to verify the coefficient braid equation
\eqref{eq:coefficient-ybe}.  Write
\[
  c_P=\overline{\gamma(P)}\,|A|^{-1},
  \qquad
  f=f^{\mathrm{II}}_{Q,P}.
\]
For \(s=(u,v)\), \(t=(r,w)\), and a summation variable
\((a,b)\in A\oplus A\),
the left-hand side of \eqref{eq:coefficient-ybe} is
\[
  c_P^3Q(r)P(w)
  \sum_{a,b\in A}
  Q(u-a)P(v-b)Q(a)P(b)\pair{b}{r}^{-1}\pair{w}{a}^{-1}.
\]
This separates as the product of two one-variable sums.  Using
\[
  Q(x-y)Q(y)=Q(x)\pair{y}{y}\pair{x}{y}^{-1},
  \qquad
  P(x-y)P(y)=P(x)\pair{y}{y}\pair{x}{y}^{-1},
\]
it becomes
\[
  c_P^3Q(u)Q(r)P(v)P(w)\,S(u+w)S(v+r),
\]
where
\[
  S(z)=\sum_{a\in A}\pair{a}{a}\pair{z}{a}^{-1}.
\]
The right-hand side of \eqref{eq:coefficient-ybe} gives the same expression
with the two factors interchanged.
\[
  c_P^3Q(u)Q(r)P(v)P(w)\,S(r+v)S(w+u).
\]
Since \(A\) is abelian, \(S(u+w)S(v+r)=S(w+u)S(r+v)\).  Thus the
coefficient equation holds for all \(s,t\in A\oplus A\).
Lemmas~\ref{lem:coefficient-ybe} and~\ref{lem:coefficient-unitarity} translate
these two identities into braid compatibility and unitarity of the tower
elements.
\end{proof}

A self-adjoint automorphism of \(A\) gives a reparametrization of Family II.

\begin{prop}
\label{prop:family-II-coordinate-change}
Let \(E\in\Aut(A)\) be self-adjoint with respect to
\(\pair{\cdot}{\cdot}\), meaning that
\[
  \pair{Ex}{y}=\pair{x}{Ey}\qquad(x,y\in A).
\]
Suppose \(Q_E,P_E:A\to\Uone\) satisfy
\begin{align*}
  Q_E(x+y)&=Q_E(x)Q_E(y)\pair{Ex}{y},\\
  P_E(x+y)&=P_E(x)P_E(y)\pair{E^{-1}x}{y}.
\end{align*}
Define
\[
  f^{\mathrm{II},E}_{Q_E,P_E}(t,c)
  =\overline{\gamma(P_E)}\,|A|^{-1}Q_E(t)P_E(c)
\]
and
\[
  r^{\mathrm{II},E}_{Q_E,P_E,i}
  =\sum_{t,c\in A}f^{\mathrm{II},E}_{Q_E,P_E}(t,c)U_i^{(t,c)}.
\]
These elements form a unitary braid-compatible family in the original
Family II tower.  The family, denoted \(\mathrm{Family\ II}_E\), is obtained
from ordinary Family II by the coordinate change \(c\mapsto Ec\).
\end{prop}

\begin{proof}
Set
\[
  B_E(x,y)=\pair{Ex}{y},
  \qquad
  \widetilde P=P_E\circ E.
\]
Self-adjointness makes \(B_E\) symmetric and nondegenerate.  Both \(Q_E\)
and \(\widetilde P\) have polarization \(B_E\), because
\[
  \widetilde P(x+y)
  =\widetilde P(x)\widetilde P(y)\pair{Ex}{y}.
\]
Hence Proposition~\ref{prop:family-II-reindexing}, applied to the metric
bicharacter \(B_E\), gives an ordinary Family II solution with coefficient
\[
  \overline{\gamma(\widetilde P)}\,|A|^{-1}
  Q_E(t)\widetilde P(c).
\]
Let
\[
  \Psi:A\oplus A\longrightarrow A\oplus A,
  \qquad \Psi(t,c)=(t,Ec).
\]
If \(\tau_E\) denotes the Family II adjacent bicharacter formed from
\(B_E\), then
\[
  \tau_{\mathrm{II}}(\Psi(t,c),\Psi(t',c'))
  =B_E(c,t')B_E(c',t)
  =\tau_E((t,c),(t',c')).
\]
Thus \(\Psi\) induces a tower isomorphism
\(U_i^{(t,c)}\mapsto U_i^{(t,Ec)}\).  Since \(E\) permutes \(A\),
\(\gamma(\widetilde P)=\gamma(P_E)\); moreover,
\[
  \sum_{t,c\in A}Q_E(t)\widetilde P(c)U_i^{(t,Ec)}
  =\sum_{t,c'\in A}Q_E(t)P_E(c')U_i^{(t,c')},
  \qquad c'=Ec.
\]
Hence the ordinary Family II element for \((A,B_E)\) is carried exactly to
\(r^{\mathrm{II},E}_{Q_E,P_E,i}\).  Unitarity and braid compatibility are
therefore inherited through this tower isomorphism from
Proposition~\ref{prop:family-II-reindexing}.
\end{proof}

\subsection{Family III}
\label{subsec:family-III}

The third family uses the same metric bicharacter, but places part of the
twisting inside each local algebra.  Set
\[
  G_{\mathrm{III}}=A\oplus A
\]
and define
\begin{equation}\label{eq:family-III-cocycle}
  \nu_{\mathrm{III}}\bigl((a,b),(c,d)\bigr)
  =
  \pair{a}{c}\pair{b}{c}\pair{b}{d}.
\end{equation}
This is a unitary bicharacter and hence a normalized \(2\)-cocycle.  Its
commutator bicharacter is
\begin{equation}\label{eq:family-III-adjacent-bicharacter}
  \tau_{\mathrm{III}}\bigl((a,b),(c,d)\bigr)
  =
  \frac{\nu_{\mathrm{III}}((a,b),(c,d))}
       {\nu_{\mathrm{III}}((c,d),(a,b))}
  =
  \pair{b}{c}\pair{a}{d}^{-1}.
\end{equation}
For nontrivial \(A\), this commutator is nondegenerate; hence
\(\nu_{\mathrm{III}}\) is not cohomologous to \(1\).  Define the constant
coefficient function
\begin{equation}\label{eq:family-III-coefficient-function}
  f^{\mathrm{III}}(a,b)=|A|^{-1}.
\end{equation}
The corresponding braid element in
\(\T_n(G_{\mathrm{III}};\nu_{\mathrm{III}},
\tau_{\mathrm{III}})\) is
\[
  r^{\mathrm{III}}_i
  =
  |A|^{-1}\sum_{a,b\in A}U_i^{(a,b)}.
\]
At level two, we write \(r^{\mathrm{III}}=r^{\mathrm{III}}_1\).

\begin{prop}\label{prop:family-III-coefficients}
The function \(f^{\mathrm{III}}\) satisfies the
\((G_{\mathrm{III}};\nu_{\mathrm{III}},\tau_{\mathrm{III}})\)
coefficient Yang--Baxter equation, and the elements \(r^{\mathrm{III}}_i\) are unitary
and braid-compatible in
\(\T_n(G_{\mathrm{III}};\nu_{\mathrm{III}},
\tau_{\mathrm{III}})\).
\end{prop}

\begin{proof}
For \((c,d)\in A\oplus A\), the relevant coefficient sum is
\begin{align*}
 C_{f^{\mathrm{III}}}^{\nu_{\mathrm{III}}}(c,d)
 &=
 |A|^{-2}\sum_{a,b\in A}
 \nu_{\mathrm{III}}((a,b),(c,d))^{-1}\\
 &=
 |A|^{-2}
 \left(\sum_{a\in A}\pair{a}{c}^{-1}\right)
 \left(\sum_{b\in A}\pair{b}{c+d}^{-1}\right).
\end{align*}
Character orthogonality shows that this is \(1\) at \((c,d)=0\) and
vanishes otherwise.

For the coefficient equation, all values of \(f^{\mathrm{III}}\) are equal;
hence the common coefficient factor cancels.  Put
\[
  K(x,y)=\pair{x}{x}\pair{x}{y}\pair{y}{y}
\]
and, for \(\xi,\eta\in A\), define
\[
  \mathcal F(\xi,\eta)
  =
  \sum_{x,y\in A}
  K(x,y)^{-1}\pair{\xi}{x}\pair{\eta}{y}.
\]
For \(s=(u,v)\), \(t=(r,w)\), and a summation variable \((x,y)\), direct
substitution of
\eqref{eq:family-III-cocycle} and
\eqref{eq:family-III-adjacent-bicharacter} gives
\begin{align*}
 \sum_{x,y\in A}
 \nu_{\mathrm{III}}(s-(x,y),(x,y))
 \tau_{\mathrm{III}}((x,y),t)^{-1}
 &=\mathcal F(u+v+w,v-r),\\
 \sum_{x,y\in A}
 \nu_{\mathrm{III}}((x,y),t-(x,y))
 \tau_{\mathrm{III}}(s,(x,y))^{-1}
 &=\mathcal F(r-v,r+w+u).
\end{align*}
The automorphism
\[
  (x,y)\longmapsto(-x-y,x)
\]
preserves \(K\), since
\[
  K(-x-y,x)
  =
  \pair{x+y}{x+y}\pair{x+y}{x}^{-1}\pair{x}{x}
  =
  K(x,y).
\]
Applying this change of variables gives the identity
\begin{equation}\label{eq:family-III-order-three-Fourier-symmetry}
  \mathcal F(\xi,\eta)
  =
  \mathcal F(-\eta,\xi-\eta).
\end{equation}
Since
\[
  r-v=-(v-r),
  \qquad
  r+w+u=(u+v+w)-(v-r),
\]
the two coefficient sums agree.  The result now follows from
Corollary~\ref{cor:coefficient-unitary-braid-criterion}.
\end{proof}

\subsection{Tensor-power realizations}
\label{subsec:tensor-power-representations}

Let \(V=\C[A]\) have basis \(\{e_x:x\in A\}\), and put
\(W=V\otimes V\).  Families I and II act on \(V^{\ot n}\), while Family III
acts on \(W^{\ot n}\).  The tower maps are defined for \(n\geq1\); their
local formulas use \(n\geq2\) and \(1\leq i<n\).  Let
\(X_a,Z_a\in\End(V)\) be defined by
\[
  X_ae_x=e_{x+a},
  \qquad
  Z_ae_x=\pair{a}{x}e_x.
\]
They satisfy
\[
  X_aX_b=X_{a+b},
  \qquad
  Z_aZ_b=Z_{a+b},
  \qquad
  Z_bX_a=\pair{b}{a}X_aZ_b.
\]
Subscripts indicate the tensor factor on which an operator acts.

For Family I, define
\[
  \Phi_n^{\mathrm{I}}(U_i^a)=Y_i^a,
  \qquad
  Y_i^a=X_{a,i}Z_{a,i+1}
  \qquad(a\in A,\ n\geq2,\ 1\leq i<n).
\]
The Pauli commutation relation gives
\[
  Y_i^aY_{i+1}^b
  =
  \tau_{\mathrm{I}}(a,b)Y_{i+1}^bY_i^a,
\]
and the distant tower relations follow from disjoint tensor support.  Thus
\(\Phi_n^{\mathrm{I}}\) defines a compatible representation of
\(\T_n(G_{\mathrm{I}},\tau_{\mathrm{I}})\).  Since
\((Y_i^a)^*=Y_i^{-a}\), it is a \(*\)-representation.  The image of
\(r^{\mathrm{I}}_{Q,1}\) is
\begin{equation}\label{eq:Gaussian-RQ}
  R^{\mathrm{I}}_Q
  =
  |A|^{-1/2}\sum_{a\in A}Q(a)X_a\otimes Z_a.
\end{equation}
Equivalently,
\begin{equation}\label{eq:RQ-basis-action}
  R^{\mathrm{I}}_Q(e_x\otimes e_y)
  =
  |A|^{-1/2}\sum_{t\in A}
  Q(t)\pair{t}{y}\,e_{x+t}\otimes e_y.
\end{equation}
The same compatible representation applies to the opposite-polarization
branch.  For \(Q\in\mathcal Q^{-}(q)\), define
\begin{equation}\label{eq:opposite-Gaussian-RQ}
  R^{\mathrm{I},-}_Q
  =|A|^{-1/2}\sum_{a\in A}Q(a)X_a\otimes Z_a.
\end{equation}
Proposition~\ref{prop:gaussian-coefficients} implies that
\(R^{\mathrm{I},-}_Q\) is a unitary Yang--Baxter operator.

For Family II, define
\[
  \Phi_n^{\mathrm{II}}(U_i^{(t,c)})=Y_i^{(t,c)},
  \qquad
  Y_i^{(t,c)}
  =
  X_{t,i}X_{-t,i+1}Z_{c,i}Z_{c,i+1}
  \qquad(t,c\in A,\ n\geq2,\ 1\leq i<n).
\]
For products in the same position, the Pauli phase from the \(i\)-th tensor
factor is canceled by the phase from the \((i+1)\)-st tensor factor.
Consequently,
\[
  Y_i^{(t,c)}Y_i^{(t',c')}=Y_i^{(t+t',c+c')}.
\]
For adjacent positions, the only common tensor factor is the middle one, and
the Pauli commutation relation gives
\[
  Y_i^{(t,c)}Y_{i+1}^{(t',c')}
  =
  \tau_{\mathrm{II}}\bigl((t,c),(t',c')\bigr)
  Y_{i+1}^{(t',c')}Y_i^{(t,c)}.
\]
Hence \(\Phi_n^{\mathrm{II}}\) defines a compatible representation of
\(\T_n(G_{\mathrm{II}},\tau_{\mathrm{II}})\).  The adjoint relation
\((Y_i^{(t,c)})^*=Y_i^{(-t,-c)}\) again follows by cancellation of the two
Pauli phases; hence this representation preserves the \(*\)-structure.  The
image of
\(r^{\mathrm{II}}_{Q,P,1}\) is
\begin{equation}\label{eq:family-II-Pauli-expansion}
  R^{\mathrm{II}}_{Q,P}
  =
  \overline{\gamma(P)}\,|A|^{-1}
  \sum_{t,c\in A}Q(t)P(c)\,
  (X_tZ_c)\otimes(X_{-t}Z_c).
\end{equation}
For Family \(\mathrm{II}_E\) of
Proposition~\ref{prop:family-II-coordinate-change}, the same Pauli
realization gives
\begin{equation}\label{eq:family-II-E-Pauli-expansion}
  R^{\mathrm{II},E}_{Q_E,P_E}
  =
  \overline{\gamma(P_E)}\,|A|^{-1}
  \sum_{t,c\in A}Q_E(t)P_E(c)
  (X_tZ_c)\otimes(X_{-t}Z_c).
\end{equation}
This is a unitary Yang--Baxter operator.
For the polarization \(B_E'(x,y)=\pair{E^{-1}x}{y}\), self-adjointness gives
\(B_E'(c,Es)=\pair{c}{s}\).
Lemma~\ref{lem:general-shifted-Gauss-sum}, with translation variable \(Es\), gives
the corresponding basis formula
\begin{equation}\label{eq:family-II-E-basis-action}
  R^{\mathrm{II},E}_{Q_E,P_E}(e_x\otimes e_y)
  =
  |A|^{-1/2}\sum_{t\in A}
  Q_E(t)P_E\bigl(E(x+y)\bigr)^{-1}
  e_{x+t}\otimes e_{y-t}.
\end{equation}
For a function \(h:A\to\C^\times\), let \(M_h\) denote the multiplication
operator on \(V\) defined by \(M_he_s=h(s)e_s\).  Let
\(\Xi(e_x\otimes e_y)=e_{x+y}\otimes e_x\) and put
\(\widetilde P=P_E\circ E\).  In the coordinates \((s,d)=(x+y,x)\),
\eqref{eq:family-II-E-basis-action} becomes
\begin{equation}\label{eq:family-II-E-factorization}
  \Xi R^{\mathrm{II},E}_{Q_E,P_E}\Xi^{-1}
  =M_{\widetilde P^{-1}}\otimes G_{Q_E},
  \qquad
  G_{Q_E}e_d=|A|^{-1/2}\sum_{t\in A}Q_E(t)e_{d+t}.
\end{equation}
Both \(Q_E\) and \(\widetilde P\) have polarization
\(B_E(x,y)=\pair{Ex}{y}\).  Thus Family \(\mathrm{II}_E\) preserves the
total label, and its diagonal phase is pulled back by \(E\).
Returning to ordinary Family II, fix \(Q,P\in\mathcal Q(q)\).  
By the shifted Gauss sum identity \eqref{eq:shifted-Gauss-sum}, the diagonal
operator
\[
  D_P(e_x\otimes e_y)=P(x+y)^{-1}e_x\otimes e_y
\]
satisfies
\[
  D_P
  =
  \overline{\gamma(P)}\,|A|^{-1/2}
  \sum_{c\in A}P(c)\,Z_c\otimes Z_c.
\]
Therefore
\[
  R^{\mathrm{II}}_{Q,P}
  =
  |A|^{-1/2}\sum_{t\in A}Q(t)(X_t\otimes X_{-t})D_P,
\]
or, on basis vectors,
\begin{equation}\label{eq:family-II-Gaussian-R}
  R^{\mathrm{II}}_{Q,P}(e_x\otimes e_y)
  =
  |A|^{-1/2}\sum_{a\in A}
  Q(a)P(x+y)^{-1}\,e_{x+a}\otimes e_{y-a}.
\end{equation}
Equation~\eqref{eq:family-II-Gaussian-R} preserves the total label, since
\((x+a)+(y-a)=x+y\).

After the change of variables
\[
  s=x+y,\qquad d=x,
\]
the operator becomes
\[
  M_{P^{-1}}\otimes G_Q,
  \qquad
  G_Qe_d=|A|^{-1/2}\sum_{a\in A}Q(a)e_{d+a}.
\]

For Family III, on
\[
  W\otimes W=(V\otimes V)\otimes(V\otimes V)
\]
define
\begin{align}
  \mathsf P_a
  &=(I\otimes Z_a)\otimes(Z_a\otimes I),
  \label{eq:family-III-Pauli-P}\\
  \mathsf Q_b
  &=(X_b\otimes I)\otimes(X_b\otimes X_b).
  \label{eq:family-III-Pauli-Q}
\end{align}
For \((a,b)\in A\oplus A\), set
\begin{equation}\label{eq:family-III-local-Pauli-word}
  Y^{(a,b)}
  =
  q(a)^{-1}q(b)^{-1}\mathsf P_a\mathsf Q_{-b}.
\end{equation}
For \((a,b),(c,d)\in A\oplus A\), the Pauli commutation relation and the
quadratic identity give
\begin{equation}\label{eq:family-III-local-multiplication}
  Y^{(a,b)}Y^{(c,d)}
  =
  \pair{a}{c}\pair{b}{c}\pair{b}{d}
  Y^{(a+c,b+d)}
  =
  \nu_{\mathrm{III}}((a,b),(c,d))Y^{(a+c,b+d)}.
\end{equation}
Indeed, commuting \(\mathsf Q_{-b}\) past \(\mathsf P_c\) contributes
\(\pair{b}{c}\), while the scalar quotient is
\[
  \frac{q(a+c)q(b+d)}{q(a)q(b)q(c)q(d)}
  =
  \pair{a}{c}\pair{b}{d}.
\]

Let \(Y_i^{(a,b)}\) denote the operator
\eqref{eq:family-III-local-Pauli-word} on the \(i\)-th and \((i+1)\)-st
\(W\)-factors, where \((a,b)\in A\oplus A\), \(n\geq2\), and
\(1\leq i<n\).  For
adjacent positions, the two Pauli overlaps give
\begin{equation}\label{eq:family-III-adjacent-Pauli-commutation}
  Y_i^{(a,b)}Y_{i+1}^{(c,d)}
  =
  \pair{b}{c}\pair{a}{d}^{-1}
  Y_{i+1}^{(c,d)}Y_i^{(a,b)}.
\end{equation}
Operators at distance greater than one have disjoint tensor supports and
commute.  Each \(Y_i^{(a,b)}\) is unitary.  Equation
\eqref{eq:family-III-local-multiplication} also gives
\[
  (Y_i^{(a,b)})^*
  =\nu_{\mathrm{III}}((a,b),(-a,-b))^{-1}Y_i^{(-a,-b)},
\]
which is the canonical involution in the twisted group algebra.  Thus
\[
  \Phi_n^{\mathrm{III}}(U_i^{(a,b)})=Y_i^{(a,b)}
\]
defines a compatible \(*\)-representation of
\(\T_n(G_{\mathrm{III}};\nu_{\mathrm{III}},
\tau_{\mathrm{III}})\) on \(W^{\ot n}\).  Generalized Pauli words are
orthogonal for the Hilbert--Schmidt inner product.  For a normal-ordered
tower basis monomial
\(\prod_{i=1}^{n-1}U_i^{(\alpha_i,\beta_i)}\), the label \(\alpha_i\) is
read from the \(Z\)-label on the second \(V\)-factor of the \(i\)-th
\(W\)-factor, while \(\beta_i\) is read from the \(X\)-label on the second
\(V\)-factor of the \((i+1)\)-st \(W\)-factor.  Distinct tower monomials therefore give
distinct Pauli words.  Therefore \(\Phi_n^{\mathrm{III}}\) is faithful.

The image of \(r^{\mathrm{III}}_1\) is
\begin{equation}\label{eq:family-III-YBO}
  R^{\mathrm{III}}=R^{\mathrm{III}}_{(A,q)}
  =
  |A|^{-1}\sum_{a,b\in A}
  q(a)^{-1}q(b)^{-1}\mathsf P_a\mathsf Q_{-b}
  \in\End(W\otimes W).
\end{equation}

For any local operator \(R\) defined in this subsection, a subscript \(i\)
denotes the copy
\[
  R_i=I^{\ot(i-1)}\ot R\ot I^{\ot(n-i-1)}
  \qquad(n\geq2,\ 1\leq i<n).
\]

\begin{theorem}\label{thm:metric-families-YBE}
For every \(Q,P\in\mathcal Q(q)\), the three operators
\[
  R^{\mathrm{I}}_Q,
  \qquad
  R^{\mathrm{II}}_{Q,P},
  \qquad
  R^{\mathrm{III}}
\]
are unitary Yang--Baxter operators.
\end{theorem}

\begin{proof}
Propositions~\ref{prop:gaussian-coefficients}
and~\ref{prop:family-II-reindexing} show that the tower elements
\(r^{\mathrm{I}}_{Q,i}\) and \(r^{\mathrm{II}}_{Q,P,i}\) are unitary and
braid-compatible.  Proposition~\ref{prop:family-III-coefficients} gives the
same conclusion for \(r^{\mathrm{III}}_i\).

Each map \(\Phi_n^{\mathrm{I}}\), \(\Phi_n^{\mathrm{II}}\), and
\(\Phi_n^{\mathrm{III}}\) defined in
Subsection~\ref{subsec:tensor-power-representations} is a compatible
\(*\)-representation of the corresponding tower.  Moreover,
\[
  \Phi_n^{\mathrm{I}}(r^{\mathrm{I}}_{Q,i})
  =
  R^{\mathrm{I}}_{Q,i},
  \qquad
  \Phi_n^{\mathrm{II}}(r^{\mathrm{II}}_{Q,P,i})
  =
  R^{\mathrm{II}}_{Q,P,i},
  \qquad
  \Phi_n^{\mathrm{III}}(r^{\mathrm{III}}_i)
  =
  R^{\mathrm{III}}_i.
\]
Applying these representations to the braid relations gives the
Yang--Baxter equation for all three operators.  Since the tower elements are
unitary and the representations preserve the \(*\)-structure, the resulting
operators are unitary.
\end{proof}

\begin{cor}\label{cor:extended-metric-families-YBE}
The conclusion of Theorem~\ref{thm:metric-families-YBE} also holds for
every Family I operator \(R^{\mathrm{I},-}_Q\) with
\(Q\in\mathcal Q^-(q)\) and every Family \(\mathrm{II}_E\) operator
\(R^{\mathrm{II},E}_{Q_E,P_E}\) associated with data satisfying
Proposition~\ref{prop:family-II-coordinate-change}.
\end{cor}

\begin{proof}
Apply the compatible Pauli representations to
Propositions~\ref{prop:gaussian-coefficients}
and~\ref{prop:family-II-coordinate-change}.
\end{proof}

\subsection{Orthogonal and primary factorization}
\label{subsec:metric-factorization}

Orthogonal decompositions induce factorizations of the towers and their
representations.

For the data already defined, write
\begin{align*}
  \T_\bullet^{\mathrm I}(A,q)
  &=\T_\bullet(G_{\mathrm I},\tau_{\mathrm I}),\qquad
  \T_\bullet^{\mathrm{II}}(A,q)
  =\T_\bullet(G_{\mathrm{II}},\tau_{\mathrm{II}}),\\
  \T_\bullet^{\mathrm{III}}(A,q)
  &=\T_\bullet(G_{\mathrm{III}};\nu_{\mathrm{III}},
    \tau_{\mathrm{III}}).
\end{align*}

Suppose \(m\geq1\) and
\[
  (A,q)=\bigoplus_{r=1}^m(A^{(r)},q^{(r)})
\]
is an orthogonal decomposition.  Let \(Q,P\in\mathcal Q(q)\), put
\(Q^{(r)}=Q|_{A^{(r)}}\) and \(P^{(r)}=P|_{A^{(r)}}\).

\begin{prop}\label{prop:metric-factorization}
For each \(\mathrm F\in\{\mathrm I,\mathrm{II},\mathrm{III}\}\), there is a
natural isomorphism
\[
  \T_\bullet^{\mathrm F}(A,q)
  \cong
  \T_\bullet^{\mathrm F}(A^{(1)},q^{(1)})
  \otimes_{\mathrm{lev}}\cdots\otimes_{\mathrm{lev}}
  \T_\bullet^{\mathrm F}(A^{(m)},q^{(m)}),
\]
under which the local braid element is the tensor product of the local braid
elements of the summands.
\end{prop}

\begin{proof}
Write \(a=\sum_ra^{(r)}\) and \(b=\sum_rb^{(r)}\), and let
\(\pair{\cdot}{\cdot}_r\) denote the polarization of \(q^{(r)}\).
Orthogonality gives
\[
  q(a)=\prod_rq^{(r)}(a^{(r)}),
  \qquad
  \pair{a}{b}=\prod_r\pair{a^{(r)}}{b^{(r)}}_r,
\]
and the same product formulas hold for \(Q\) and \(P\).  Thus the adjacent
bicharacters and coefficient functions of Families I and II are external
products.  Moreover,
\(\gamma(P)=\prod_r\gamma(P^{(r)})\); hence the normalization in Family II also
factors.  For Family III, both \(\nu_{\mathrm{III}}\) and
\(\tau_{\mathrm{III}}\) are products of their restrictions, while
\(|A|^{-1}=\prod_r|A^{(r)}|^{-1}\).
Proposition~\ref{prop:tower-external-product} now gives the tower and
braid-element factorizations.
\end{proof}

After the canonical shuffle of tensor factors, the corresponding operators
satisfy
\begin{align*}
  R^{\mathrm I}_Q
    &\cong\bigotimes_rR^{\mathrm I}_{Q^{(r)}},\\
  R^{\mathrm{II}}_{Q,P}
    &\cong\bigotimes_rR^{\mathrm{II}}_{Q^{(r)},P^{(r)}},\\
  R^{\mathrm{III}}_{(A,q)}
    &\cong\bigotimes_rR^{\mathrm{III}}_{(A^{(r)},q^{(r)})}.
\end{align*}
The associated braid representations and the maps \(\varepsilon_n\) factor
in the same way.

Suppose now that data \((E,Q_E,P_E)\) defining Family \(\mathrm{II}_E\) satisfy
Proposition~\ref{prop:family-II-coordinate-change} and that every
summand is \(E\)-stable.  Set
\[
  E^{(r)}=E|_{A^{(r)}},
  \qquad Q_E^{(r)}=Q_E|_{A^{(r)}},
  \qquad P_E^{(r)}=P_E|_{A^{(r)}}.
\]
The restrictions have the required polarizations, and the coordinate-change map
\((t,c)\mapsto(t,Ec)\) respects the decomposition.  Consequently,
\[
  R^{\mathrm{II},E}_{Q_E,P_E}
  \cong
  \bigotimes_r
  R^{\mathrm{II},E^{(r)}}_{Q_E^{(r)},P_E^{(r)}}.
\]
The associated braid representations factor accordingly.

\begin{cor}\label{cor:metric-primary-factorization}
The primary decomposition
\[
  (A,q)=\bigoplus_{p\mid |A|}(A_p,q|_{A_p})
\]
induces canonical factorizations of every tower and braid representation
arising from the metric families.  This includes every admissible Family
\(\mathrm{II}_E\).
\end{cor}

\begin{proof}
If \(a\in A_p\) and \(b\in A_\ell\) with \(p\neq\ell\), then
\(\pair{a}{b}\) has order dividing both a power of \(p\) and a power of
\(\ell\), hence is trivial.  Nondegeneracy restricts to each primary
component.  Each \(A_p\) is characteristic in \(A\) and is therefore
\(E\)-stable for every \(E\in\Aut(A)\).  Proposition~\ref{prop:metric-factorization}
and the tensor-power consequences above give the stated factorizations,
including the Family \(\mathrm{II}_E\) case.
\end{proof}

\begin{remark}\label{rem:metric-block-factorization}
Each \(p\)-primary metric group admits an orthogonal decomposition into
indecomposable blocks \cite{Wall,Durfee,MirandaMorrison}.  After choosing
such decompositions, Proposition~\ref{prop:metric-factorization} gives a
further factorization into the corresponding block representations.  For
Family \(\mathrm{II}_E\), the chosen blocks must be \(E\)-stable.  Under any
such factorization, the local spectrum consists of products of block
eigenvalues, with multiplicities multiplied, and the order of the local
operator is the least common multiple of the orders of those products.
\end{remark}

\section{Generically non-Gaussian cyclic families}
\label{sec:cyclic-families}

We apply the admissible internal orthogonal-sum construction to the
dimension-\(8\) cyclic tower and construct one-parameter families that are
generically non-Gaussian in the fixed cyclic Pauli realization.  Each family
nevertheless contains signed-Gaussian specializations.

\subsection{Two-torsion sectors in cyclic towers}

Let \(d\geq2\) be an integer, let \(G=\Z/d\Z\), set
\(w=e^{2\pi\ii/d}\), and put
\[
  \tau(a,b)=w^{ab}.
\]
For \(n\geq2\) and \(1\leq i<n\), write
\(U_i=U_i^1\) in \(\T_n(G,\tau)\).  Then
\[
 U_i^d=1\qquad(1\leq i<n),
 \qquad
 U_iU_{i+1}=wU_{i+1}U_i\qquad(1\leq i<n-1).
\]

\begin{lemma}\label{lem:cyclic-two-torsion-sectors}
Fix \(n\geq2\), assume \(4\mid d\), and, for \(1\leq i<n\), set
\[
  J_i=U_i^{d/2},
  \qquad
  e_{i,\pm}=\frac{1\pm J_i}{2}.
\]
Then the \(J_i\) are commuting involutions and
\begin{equation}\label{eq:sector-switching}
  J_iU_{i+1}^k=(-1)^kU_{i+1}^kJ_i,
  \qquad
  U_i^kJ_{i+1}=(-1)^kJ_{i+1}U_i^k
  \qquad(k\in\Z,\ 1\leq i<n-1).
\end{equation}
Consequently, even powers preserve the neighboring \(e_\pm\)-sectors and
odd powers interchange them.
\end{lemma}

\begin{proof}
The equality \(J_i^2=U_i^d=1\) follows from the tower relation.  Since
\(4\mid d\), the adjacent
commutation scalar is
\[
 w^{(d/2)^2}=w^{d(d/4)}=1,
\]
Therefore adjacent \(J_i\)'s commute; distant ones commute by
the tower relations.  Moreover,
\[
 w^{(d/2)k}=(-1)^k,
\]
which gives \eqref{eq:sector-switching}.  The statement about the
projections follows by multiplying these identities by
\((1\pm J_i)/2\).
\end{proof}

If \(d\equiv2\pmod4\), then
\(w^{(d/2)^2}=-1\); hence the neighboring involutions anticommute.  Thus the
commuting eigenspace decomposition in
Lemma~\ref{lem:cyclic-two-torsion-sectors} requires \(4\mid d\).  This
observation alone gives no classification when \(4\nmid d\).

\subsection{Dimensions four and eight}

For a coefficient vector
\(\mathbf a=[a_0,\ldots,a_{d-1}]\in\C^d\), write
\[
  U=U_1^1,
  \qquad
  r(\mathbf a)=\sum_{k=0}^{d-1}a_kU^k\in\T_2(\Z/d\Z,\tau).
\]
On \(\C^d\), with basis indexed by \(\Z/d\Z\), let
\(Xe_j=e_{j+1}\) and \(Ze_j=w^je_j\).  Thus \(ZX=wXZ\).  Under the Pauli
realization \(U\mapsto X\otimes Z\), denote the image of
\(r(\mathbf a)\) by \(R(\mathbf a)\), and let \(R_i(\mathbf a)\) denote its
copy on tensor positions \(i\) and \(i+1\).

In dimension \(4\), direct conjugation to the Zak basis shows that every
polynomial-in-Pauli solution produced by the fixed-realization calculation is
monomial, including the positive-dimensional coefficient family.  This is
consistent with the monomial realization of the Clifford group in square
dimensions described in \cite{ApplebyBengtssonBrierleyGrasslGrossLarsson}.
Since these solutions introduce no new local-equivalence classes, we omit
them and turn to dimension \(8\).

For \(d=8\), one has \(w^2=\ii\).  Put
\[
  e_\pm=\frac{1\pm U^4}{2}.
\]
The two decompositions differ by the sector containing the terms multiplied
by \(t\).  For \(t,\alpha,\eta,\delta,\beta\in\C\) satisfying
\(\alpha^2=-1\), \(\eta^2=1\), \(\delta^2=1\), and \(\beta^2=-1\), set
\begin{align}
 I_{\alpha,\eta}(t)
 &=[1,\alpha t,\eta,t,-1,\alpha t,-\eta,t],
 \label{eq:d8-type-I-vector}\\
 II_{\delta,\beta}(t)
 &=[1,\delta t,\beta,-t,1,-\delta t,\beta,t].
 \label{eq:d8-type-II-vector}
\end{align}

We denote the eight resulting families by
\begin{align*}
 C_1&=I_{-\ii,-1},& C_2&=I_{-\ii,1},&
 C_3&=I_{\ii,-1},& C_4&=I_{\ii,1},\\
 C_5&=II_{-1,-\ii},& C_6&=II_{-1,\ii},&
 C_7&=II_{1,-\ii},& C_8&=II_{1,\ii}.
\end{align*}

\begin{prop}\label{prop:d8-families}
For every \(t\in\C\), each vector \(I_{\alpha,\eta}(t)\) with
\(\alpha^2=-1\), \(\eta^2=1\), and each vector
\(II_{\delta,\beta}(t)\) with \(\delta^2=1\), \(\beta^2=-1\), satisfies
the cyclic coefficient Yang--Baxter equation.  Their normalized tower
elements decompose as
\begin{align}
 \frac{r(I_{\alpha,\eta}(t))}{\sqrt8}
 &=\frac{1}{\sqrt2}(1+\eta U^2)e_-
   \boxplus_t
   \frac{1}{\sqrt2}U(\alpha+U^2)e_+,
 \label{eq:d8-type-I-boxplus}\\
 \frac{r(II_{\delta,\beta}(t))}{\sqrt8}
 &=\frac{1}{\sqrt2}(1+\beta U^2)e_+
   \boxplus_t
   \frac{1}{\sqrt2}U(\delta-U^2)e_-.
 \label{eq:d8-type-II-boxplus}
\end{align}
In the Pauli realization, the corresponding normalized operators are
unitary exactly when \(|t|=1\).
\end{prop}

\begin{proof}
The coefficient polynomials of the two templates factor as
\begin{align*}
 r(I_{\alpha,\eta}(t))
 &=(1-U^4)(1+\eta U^2)
   +tU(\alpha+U^2)(1+U^4),\\
 r(II_{\delta,\beta}(t))
 &=(1+U^4)(1+\beta U^2)
   +tU(\delta-U^2)(1-U^4),
\end{align*}
which gives \eqref{eq:d8-type-I-boxplus} and
\eqref{eq:d8-type-II-boxplus}.  On the \(e_-\)-sector one has
\(U^{-2}=-U^2\), while on the \(e_+\)-sector one has
\(U^{-2}=U^2\).  Denote the two pieces in
\eqref{eq:d8-type-I-boxplus} by \(a_I,b_I\), and those in
\eqref{eq:d8-type-II-boxplus} by \(a_{II},b_{II}\), in the order displayed.
Using \(\overline\alpha=-\alpha\), \(\overline\beta=-\beta\), and
\(\eta,\delta\in\{\pm1\}\), direct multiplication gives
\begin{align*}
 a_I^*a_I=a_Ia_I^*&=e_-,&
 b_I^*b_I=b_Ib_I^*&=e_+,\\
 a_{II}^*a_{II}=a_{II}a_{II}^*&=e_+,&
 b_{II}^*b_{II}=b_{II}b_{II}^*&=e_-.
\end{align*}
Moreover, for \(\star\in\{I,II\}\),
\[
 a_\star^*b_\star=b_\star^*a_\star
 =a_\star b_\star^*=b_\star a_\star^*=0,
\]
because \(U\) commutes with \(e_\pm\) and \(e_+e_-=0\).  Thus the two
pieces in each line are orthogonal partial unitaries with complementary
initial and final projections.

It remains to check braid compatibility.  Let \(\Delta_I(t)\) and
\(\Delta_{II}(t)\) denote the two braid differences formed from the
unnormalized coefficient vectors.  After straightening in the basis
\(U_1^jU_2^k\) and reducing by \(w^4=-1\), the only possibly nonzero
coefficient values of \(\Delta_I(t)\), up to sign, are
\begin{align*}
 &2\bigl((\alpha^2+1)t^2+\eta(1-\eta^2)\bigr),
 &&2\bigl((\alpha^2+1)t^2+\eta(\eta^2-1)\bigr),\\
 &2\eta w^2t^2(\alpha^2+1)(w+1),
 &&2\eta w^2t^2(\alpha^2+1)(w-1).
\end{align*}
They vanish when \(\alpha^2=-1\) and \(\eta^2=1\).  For
\(\Delta_{II}(t)\), the corresponding coefficient values are
\begin{align*}
 &2\bigl(\beta(\beta^2+1)+(\delta^2-1)t^2\bigr),
 &&2\bigl(-\beta(\beta^2+1)+(\delta^2-1)t^2\bigr),\\
 &2\beta w^2t^2(\delta^2-1)(w+1),
 &&2\beta w^2t^2(\delta^2-1)(w-1),
\end{align*}
which vanish when \(\beta^2=-1\) and \(\delta^2=1\).  Hence both braid
differences vanish identically in \(t\), proving the four identities in
Lemma~\ref{lem:boxplus-braid-relations}.  The calculation is independent of
the adjacent position.  Writing \(x_I(t)=a_I+tb_I\) and
\(x_{II}(t)=a_{II}+tb_{II}\), orthogonality gives
\[
 x_I(t)^*x_I(t)=x_I(t)x_I(t)^*=e_-+|t|^2e_+,
\]
and
\[
 x_{II}(t)^*x_{II}(t)=x_{II}(t)x_{II}(t)^*=e_++|t|^2e_-.
\]
Either product is the identity exactly when \(|t|=1\), proving both
unitarity and the necessity of the stated condition.
\end{proof}

\begin{cor}\label{cor:d8-positive-Gaussian}
Within the family \(C_8(t)\), the positive cyclic Gaussian specializations
are exactly those satisfying \(t^2=w\).  In particular, if
\(\xi=e^{\pi\ii/8}\), then \(C_8(\xi)\) is the standard positive quadratic
phase \(Q(k)=\xi^{k^2}\).
\end{cor}

\begin{proof}
A cyclic quadratic phase with positive
polarization satisfies
\[
  Q(k)=Q(1)^k w^{k(k-1)/2}.
\]
Here \(Q(1)=t\), while comparison of the second coefficient gives
\(Q(2)=t^2w=w^2\), hence \(t^2=w\).  Conversely, this relation makes all
eight coefficients agree.  At \(t=\xi\), one has
\(Q(k)=\xi^k(\xi^2)^{k(k-1)/2}=\xi^{k^2}\).
\end{proof}

The corollary singles out the standard positive Gaussian phase; it
does not mean that \(C_8\) is the only family meeting the signed Gaussian
locus.  Every \(C_m\) contains a signed Gaussian specialization, with
explicit base points recorded in \eqref{eq:d8-Gaussian-basepoints}.

The \emph{projective order} of an invertible operator is the order of its
class modulo nonzero scalar operators; equivalently, it is the least
\(k\geq1\) for which its \(k\)-th power is scalar, if such a \(k\) exists.

\begin{cor}\label{cor:cyclic-family-infinite-order}
For each \(1\leq m\leq8\) and \(|t|=1\), there are eighth roots of unity
\(\zeta_m^+,\zeta_m^-,\xi_m\) such that the
distinct generic eigenvalues of \(R(C_m(t))/\sqrt8\) are
\begin{equation}\label{eq:d8-family-spectrum}
  \zeta_m^+,\qquad \zeta_m^-,\qquad
  \xi_m t,\qquad -\xi_m t.
\end{equation}
Each has generic multiplicity \(16\).  Hence, if \(t\) is not a root of
unity, every braid representation containing one of these local generators
has infinite image.
\end{cor}

\begin{proof}
Let \(p_I(z)\) and \(p_{II}(z)\) denote the
normalized evaluations of the coefficient polynomials in
\eqref{eq:d8-type-I-vector} and \eqref{eq:d8-type-II-vector}, respectively.
Using the factorizations in the proof of Proposition~\ref{prop:d8-families}
and \(w^2=\ii\), their values at the eight Fourier points are
\[
\begin{array}{@{}c|c|c@{}}
\toprule
k & p_I(w^k) & p_{II}(w^k)\\
\midrule
0 & t(1+\alpha)/\sqrt2 & (1+\beta)/\sqrt2\\
1 & (1+\eta\ii)/\sqrt2 & tw(\delta-\ii)/\sqrt2\\
2 & \ii t(\alpha-1)/\sqrt2 & (1-\beta)/\sqrt2\\
3 & (1-\eta\ii)/\sqrt2 & tw^3(\delta+\ii)/\sqrt2\\
4 & -t(1+\alpha)/\sqrt2 & (1+\beta)/\sqrt2\\
5 & (1+\eta\ii)/\sqrt2 & -tw(\delta-\ii)/\sqrt2\\
6 & -\ii t(\alpha-1)/\sqrt2 & (1-\beta)/\sqrt2\\
7 & (1-\eta\ii)/\sqrt2 & -tw^3(\delta+\ii)/\sqrt2\\
\bottomrule
\end{array}
\]
For type I, the relation \(\alpha^2=-1\) implies
\[
  \ii(\alpha-1)=\pm(1+\alpha),
\]
so the four even Fourier points give a pair of opposite eighth roots
multiplied by \(t\), each occurring twice.  The four odd points give the
two fixed eighth roots
\((1+\eta\ii)/\sqrt2\) and \((1-\eta\ii)/\sqrt2\), again twice each.
For type II, the relation \(\delta^2=1\) similarly gives
\[
  w^3(\delta+\ii)=\pm w(\delta-\ii),
\]
so the odd points give the opposite \(t\)-dependent pair twice each, while
the even points give the two fixed roots
\((1+\beta)/\sqrt2\) and \((1-\beta)/\sqrt2\) twice each.  This proves the
form of the four generic eigenvalues in
\eqref{eq:d8-family-spectrum}.

In the two-qudit Pauli realization \(U=X\otimes Z\), every eigenvalue
\(w^k\) of \(U\) has multiplicity \(8\): among the \(64\) tensor products
of eigenvectors of \(X\) and \(Z\), exactly eight have eigenvalue \(w^k\).
Since each generic value in the table occurs at two Fourier points, its
multiplicity is \(2\cdot8=16\).  At exceptional parameter values where two
displayed eigenvalues coincide, their multiplicities are added.

If the projective order were finite, every ratio of eigenvalues would be a
root of unity; comparing a fixed-sector eigenvalue with a \(t\)-sector
eigenvalue forces \(t\) to be a root of unity.
\end{proof}

\section{Conjectures on classification and braid images}
\label{sec:conjectures}

We first make precise the closure process used in the classification
conjecture.  A \emph{realized local-tower presentation} is a quadruple
\[
  \mathfrak d=(\A_\bullet,\Phi_\bullet,V,r),
\]
where \(\A_\bullet\) is a local tower of finite-dimensional unital
\(*\)-algebras in the sense of Subsection~\ref{subsec:tower-operations},
the maps
\[
  \Phi_n:\A_n\longrightarrow\End(V^{\ot n})
  \qquad(n\geq1)
\]
are compatible unital \(*\)-representations, and \(r\in\A_2\) is unitary
with braid-compatible local copies.  Its associated local operator is
\[
  R_{\mathfrak d}=\Phi_2(r)\in\End(V\otimes V).
\]

For a unitary Yang--Baxter operator \(R\) on \(V\), its
\emph{braid-image presentation} is the realized local-tower presentation
whose \(n\)-th algebra is
\[
  \mathcal D_n(R)
  =\operatorname{alg}^*\langle I,R_1,\ldots,R_{n-1}\rangle
  \subseteq\End(V^{\ot n}),
  \qquad \mathcal D_1(R)=\C,
\]
with the evident inclusions, local maps, and defining representations.

\begin{definition}\label{def:generated-tower-presentations}
Let \(\mathfrak G\) be the smallest class of realized local-tower
presentations satisfying the following rules.
\begin{enumerate}[label=(\roman*)]
\item The class contains the braid-image presentation of every monomial or
group-type unitary Yang--Baxter operator, and every realized presentation
\[
  (\T_\bullet(G;\nu,\tau),\Phi_\bullet,V,r_f)
\]
arising from a unitary braid-compatible coefficient element in a twisted
group-algebra tower.
\item It is closed under external products: if
\[
 \mathfrak d_j=(\A_\bullet^{(j)},\Phi_\bullet^{(j)},V_j,r^{(j)})
 \in\mathfrak G\qquad(j=1,2),
\]
then the levelwise tensor product of their towers and representations, with
distinguished element
\(r^{(1)}\otimes r^{(2)}\) and the canonical shuffle of tensor factors,
belongs to \(\mathfrak G\).
\item It is closed under direct sums: if
\(\mathfrak d_1,\mathfrak d_2\in\mathfrak G\), then the braid-image
presentation of
\[
  R_{\mathfrak d_1}\boxplus R_{\mathfrak d_2}
  \quad\text{on}\quad V_1\oplus V_2
\]
belongs to \(\mathfrak G\).
\item It is closed under admissible internal orthogonal sums within an
already generated tower.  More precisely, if
\(\mathfrak d=(\A_\bullet,\Phi_\bullet,V,r)\in\mathfrak G\) and
\(a,b\in\A_2\) satisfy the braid identities and orthogonality hypotheses of
Proposition~\ref{prop:boxplus-unitarity}, with common scalar \(c>0\), then
\[
  (\A_\bullet,\Phi_\bullet,V,
    c^{-1/2}(a\boxplus_t b))\in\mathfrak G
  \qquad(|t|=1).
\]
\end{enumerate}
We call the operators \(R_{\mathfrak d}\), for
\(\mathfrak d\in\mathfrak G\), \emph{generated operators}.
\end{definition}

External products, Proposition~\ref{prop:YBO-direct-sum}, and
Proposition~\ref{prop:boxplus-unitarity} show inductively that every
generated operator is a unitary Yang--Baxter operator.  The last rule is a
replacement operation inside the local
algebra of an already generated presentation.  In particular,
the internal operation \(\boxplus_t\) is distinct from the direct-sum
operation \(\boxplus\): its inputs \(a\) and \(b\) are allowed to be
noninvertible partial braid elements.

Thus the precise form of
Conjecture~\ref{conj:unitary-YBO-classification} is that every unitary
Yang--Baxter operator is equivalent to \(R_{\mathfrak d}\) for some
\(\mathfrak d\in\mathfrak G\).
Conjecture~\ref{conj:YBO-virtual-abelian}, stated and attributed in the
Introduction, concerns every unitary Yang--Baxter operator and is independent
of the choice of a tower presentation.

\section{Traces, spectra, and link invariants}
\label{sec:trace-invariants}

The coefficient calculus yields Markov traces and, through finite Fourier
transform, local spectra and skein relations.

\subsection{Scalar enhancements and coefficient traces}

Let \(\C=\A_1\subset\A_2\subset\cdots\) be a tower with a
braid-compatible family \(r_i\).  Write
\[
  \rho_n:\B_n\longrightarrow\A_n^\times,
  \qquad n\geq1,
\]
where \(\rho_1\) is trivial and
\(\rho_n(\sigma_i)=r_i\) for \(n\geq2\) and \(1\leq i<n\),
and, for \(\xi\in\B_n\), let \(\widehat\xi\) denote its braid closure.  A
\emph{scalar enhancement} consists of
cyclic linear forms \(\operatorname{tr}_n:\A_n\to\C\), \(n\geq1\), and
nonzero scalars \(\alpha,\beta\) satisfying, for every \(n\geq1\) and
\(x\in\A_n\subset\A_{n+1}\),
\begin{align}
  \operatorname{tr}_{n+1}(xr_n)
    &=\alpha\beta\operatorname{tr}_n(x),
    \label{eq:tower-enhancement-positive}\\
  \operatorname{tr}_{n+1}(xr_n^{-1})
    &=\alpha^{-1}\beta\operatorname{tr}_n(x).
    \label{eq:tower-enhancement-negative}
\end{align}
Here cyclicity means
\(\operatorname{tr}_n(xy)=\operatorname{tr}_n(yx)\) for all
\(x,y\in\A_n\).
For a braid \(\xi\in\B_n\), let \(\operatorname{wr}(\xi)\) be its writhe
and define
\begin{align}
  F(\widehat\xi)
    &=\beta^{-n}\operatorname{tr}_n(\rho_n(\xi)),
    \label{eq:tower-framed-invariant}\\
  T(\widehat\xi)
    &=\alpha^{-\operatorname{wr}(\xi)}\beta^{-n}
      \operatorname{tr}_n(\rho_n(\xi)).
    \label{eq:tower-oriented-invariant}
\end{align}
Cyclicity gives invariance under braid conjugacy, while
\eqref{eq:tower-enhancement-positive}--\eqref{eq:tower-enhancement-negative}
show that \(F\) changes by \(\alpha^{\pm1}\) under stabilization and that
the writhe correction makes \(T\) an oriented link invariant.  This is the
scalar part of Turaev's enhanced Yang--Baxter construction
\cite{Turaev1988,KasselTuraev}.

\begin{prop}\label{prop:local-annihilators}
Fix \(n\geq2\), \(1\leq i<n\), and \(x,y\in\B_n\), and set
\(L_j=\widehat{x\sigma_i^jy}\).  If
\(p(t)=\sum_{j\in\Z}c_jt^j\) is a Laurent polynomial with finite support
and \(p(r_i)=0\), then
\[
  \sum_jc_jF(L_j)=0,
  \qquad
  \sum_jc_j\alpha^jT(L_j)=0.
\]
\end{prop}

\begin{proof}
The first identity follows from
\[
  \sum_jc_j\operatorname{tr}_n\bigl(\rho_n(x)r_i^j\rho_n(y)\bigr)
  =\operatorname{tr}_n\bigl(\rho_n(x)p(r_i)\rho_n(y)\bigr)=0.
\]
The writhe of \(x\sigma_i^jy\) differs from the fixed outside writhe by
\(j\), which gives the factor \(\alpha^j\) in the oriented relation.
\end{proof}

Fix a twisted tower \(\T_\bullet(G;\nu,\tau)\), and let \(\varepsilon_n\)
be its identity-coefficient functional.  For \(D\in\C^\times\), put
\[
  \operatorname{tr}^{(D)}_n=D^n\varepsilon_n.
\]

\begin{prop}\label{prop:coefficient-trace-enhancement}
Suppose \(f:G\to\C\) satisfies the coefficient Yang--Baxter equation,
\(r_f\) is unitary, and \(f(0)\neq0\).  Then
\(\operatorname{tr}^{(D)}_n\) is a scalar enhancement with
\begin{equation}\label{eq:coefficient-enhancement-parameters}
  \alpha_f=\frac{f(0)}{|f(0)|},
  \qquad
  \beta_f=D|f(0)|.
\end{equation}
\end{prop}

\begin{proof}
Write \(\Omega_n=\Omega_n^{\nu,\tau}\) and
\(\T_n=\T_n(G;\nu,\tau)\).
For basis elements \(U^{\mathbf g},U^{\mathbf h}\), both
\(\varepsilon_n(U^{\mathbf g}U^{\mathbf h})\) and
\(\varepsilon_n(U^{\mathbf h}U^{\mathbf g})\) vanish unless
\(\mathbf h=-\mathbf g\).  In the remaining case, the cocycle identity at
\((\mathbf g,-\mathbf g,\mathbf g)\) gives
\[
  \Omega_n(\mathbf g,-\mathbf g)
  =\Omega_n(-\mathbf g,\mathbf g),
\]
and therefore \(\varepsilon_n\) is cyclic.  If \(x\in\T_n\subset\T_{n+1}\), the
identity coefficient of \(xr_{f,n}\) can only use the identity coefficient
of \(x\) and the term \(g=0\) of \(r_{f,n}\).  Hence
\[
  \varepsilon_{n+1}(xr_{f,n})=f(0)\varepsilon_n(x).
\]
Since \(r_f^{-1}=r_f^*\), the corresponding coefficient for the inverse is
\(\overline{f(0)}\).  Multiplication by \(D^{n+1}\) gives
\eqref{eq:tower-enhancement-positive}--\eqref{eq:tower-enhancement-negative}
with the parameters in \eqref{eq:coefficient-enhancement-parameters}.
\end{proof}

\begin{theorem}\label{thm:metric-family-enhancement}
Let \((A,q)\) be a finite metric group.  For Family I, choose
\(\epsilon\in\{\pm1\}\) and \(Q\in\mathcal Q^\epsilon(q)\).  For
Family \(\mathrm{II}_E\), choose datum \((E,Q_E,P_E)\) satisfying
Proposition~\ref{prop:family-II-coordinate-change}.  Consider also
Family III associated with \((A,q)\).  For each of these three metric
families, the coefficient trace agrees with the ordinary matrix trace in
the corresponding Pauli realization, and the scalar-enhancement parameters
are displayed in the following table.
\[
\begin{array}{@{}ccccc@{}}
\toprule
\textup{family} & D & f(0) & \alpha & \beta\\
\midrule
\mathrm{I}^{\epsilon}
  & |A| & |A|^{-1/2} & 1 & \sqrt{|A|}\\
\mathrm{II}_E
  & |A| & \overline{\gamma(P_E)}|A|^{-1}
  & \overline{\gamma(P_E)} & 1\\
\mathrm{III}
  & |A|^2 & |A|^{-1} & 1 & |A|\\
\bottomrule
\end{array}
\]
Ordinary Family II is recovered by taking \(E=\id_A\).
\end{theorem}

\begin{proof}
Substituting the values of \(f(0)\) into
Proposition~\ref{prop:coefficient-trace-enhancement} gives the parameters in
the table.  We verify explicitly that the tower basis is separated by Pauli
labels.  For Family I, if
\(\mathbf a=(a_1,\ldots,a_{n-1})\), then
\[
  \Phi_n^{\mathrm I}(U^{\mathbf a})
  =\prod_{i=1}^{n-1}X_{a_i,i}Z_{a_i,i+1}.
\]
For every \(i<n\), the \(X\)-label on the \(i\)-th physical tensor factor is
\(a_i\).  Hence the Pauli label of the image determines \(\mathbf a\).

For Family II, write
\(\mathbf t=(t_1,\ldots,t_{n-1})\) and
\(\mathbf c=(c_1,\ldots,c_{n-1})\).  The first physical tensor factor of
\[
  \Phi_n^{\mathrm{II}}
  \left(\prod_{i=1}^{n-1}U_i^{(t_i,c_i)}\right)
  =\prod_{i=1}^{n-1}
    (X_{t_i}Z_{c_i})_i(X_{-t_i}Z_{c_i})_{i+1}
\]
has Pauli label \((t_1,c_1)\).  For \(2\leq j<n\), the label on the
\(j\)-th factor is
\[
  (t_j-t_{j-1},\,c_j+c_{j-1}).
\]
Once \((t_{j-1},c_{j-1})\) is known, this label determines
\((t_j,c_j)\).  Induction therefore recovers all tower labels.  The same
argument applies to Family \(\mathrm{II}_E\), whose Pauli basis words are the
same as in \eqref{eq:family-II-E-Pauli-expansion}.  For Family III, the
construction of \(\Phi_n^{\mathrm{III}}\) recovers the two labels at each
tower position from, respectively, the indicated \(Z\)- and \(X\)-labels on
the physical \(W\)-factors.

The nondegeneracy of the metric bicharacter makes the generalized Pauli
labels faithful, and Pauli words with distinct labels are linearly
independent.  Thus each \(\Phi_n\) is faithful and only the identity tower
monomial maps to a scalar identity word.  Every nonidentity Pauli word has
trace zero, whereas the identity has trace \(D^n\).  Consequently,
\[
  \Tr\circ\Phi_n=D^n\varepsilon_n.
\]
\end{proof}

\begin{cor}\label{cor:cyclic-family-enhancement}
For \(|t|=1\) and \(1\leq m\leq8\), the normalized dimension-eight family
\(R(C_m(t))/\sqrt8\) has coefficient trace equal to the ordinary matrix
trace and scalar-enhancement parameters
\[
  (D,\alpha,\beta)=(8,1,\sqrt8).
\]
\end{cor}

\begin{proof}
Proposition~\ref{prop:d8-families} gives unitarity on the stated parameter
locus, and the coefficient function has \(f(0)=8^{-1/2}\).
Proposition~\ref{prop:coefficient-trace-enhancement} therefore gives
\((D,\alpha,\beta)=(8,1,\sqrt8)\).  The same
one-generator Pauli-label argument used for Family I gives
\(\Tr\circ\Phi_n=8^n\varepsilon_n\).
\end{proof}

Fix \(\epsilon\in\{\pm1\}\), \(Q\in\mathcal Q^\epsilon(q)\), and a
datum \((E,Q_E,P_E)\) defining Family \(\mathrm{II}_E\) and satisfying
Proposition~\ref{prop:family-II-coordinate-change}.  For \(n\geq1\),
let
\(\rho^{\T,\mathrm I,\epsilon}_{Q,n}\),
\(\rho^{\T,\mathrm{II},E}_{Q_E,P_E,n}\), and
\(\rho^{\T,\mathrm{III}}_n\) denote the tower representations defined by
\[
  \sigma_i\longmapsto r^{\mathrm I,\epsilon}_{Q,i},
  \qquad
  \sigma_i\longmapsto r^{\mathrm{II},E}_{Q_E,P_E,i},
  \qquad
  \sigma_i\longmapsto r^{\mathrm{III}}_i
  \qquad(n\geq2,\ 1\leq i<n),
\]
respectively, with the representations at \(n=1\) understood to be trivial.
For \(\xi\in\B_n\), the resulting oriented invariants are
\begin{align}
  \mathcal I^{\mathrm I,\epsilon}_Q(\widehat\xi)
  &=|A|^{n/2}\,
    \varepsilon_n(\rho^{\T,\mathrm I,\epsilon}_{Q,n}(\xi)),
  \label{eq:family-I-trace-invariant}\\
  \mathcal I^{\mathrm{II},E}_{Q_E,P_E}(\widehat\xi)
  &=\gamma(P_E)^{\operatorname{wr}(\xi)}|A|^n\,
    \varepsilon_n(\rho^{\T,\mathrm{II},E}_{Q_E,P_E,n}(\xi)),
  \label{eq:family-II-trace-invariant}\\
  \mathcal I^{\mathrm{III}}(\widehat\xi)
  &=|A|^n\,
    \varepsilon_n(\rho^{\T,\mathrm{III}}_n(\xi)).
  \label{eq:family-III-trace-invariant}
\end{align}
When the metric data must be displayed, we write
\(\mathcal I^{\mathrm{III}}_{(A,q)}=\mathcal I^{\mathrm{III}}\).  For
ordinary Family II, we write
\(\mathcal I^{\mathrm{II}}_{Q_0,P_0}
:=\mathcal I^{\mathrm{II},\id_A}_{Q_0,P_0}\) for
\(Q_0,P_0\in\mathcal Q(q)\).
Their values on the unknot are \(\sqrt{|A|}\), \(|A|\), and \(|A|\),
respectively.

\begin{cor}\label{cor:metric-invariant-factorization}
Let \(L\) be an oriented link, \(\epsilon\in\{\pm1\}\),
\(Q_\epsilon\in\mathcal Q^\epsilon(q)\), and
\(Q_+,P\in\mathcal Q(q)\).  Put
\(Q_\epsilon^{(r)}=Q_\epsilon|_{A^{(r)}}\),
\(Q_+^{(r)}=Q_+|_{A^{(r)}}\), and
\(P^{(r)}=P|_{A^{(r)}}\).  Under the orthogonal decomposition in
Proposition~\ref{prop:metric-factorization}, the three invariants satisfy
\begin{align*}
  \mathcal I^{\mathrm I,\epsilon}_{Q_\epsilon}(L)
    &=\prod_r\mathcal I^{\mathrm I,\epsilon}_{Q_\epsilon^{(r)}}(L),\\
  \mathcal I^{\mathrm{II}}_{Q_+,P}(L)
    &=\prod_r\mathcal I^{\mathrm{II}}_{Q_+^{(r)},P^{(r)}}(L),\\
  \mathcal I^{\mathrm{III}}_{(A,q)}(L)
    &=\prod_r\mathcal I^{\mathrm{III}}_{(A^{(r)},q^{(r)})}(L).
\end{align*}
The same formula holds for Family \(\mathrm{II}_E\) when the summands are
\(E\)-stable and the phases are restricted to them.
\end{cor}

\begin{proof}
Equation~\eqref{eq:tower-external-product-trace} factors the coefficient
trace.  The group orders, normalized Gauss sums, and writhe factors in
\eqref{eq:family-I-trace-invariant}--\eqref{eq:family-III-trace-invariant}
factor over the orthogonal summands.
\end{proof}

\subsection{Fourier spectra and skein relations}

The local algebras of Families I and II are ordinary group algebras.
Fourier diagonalization therefore controls spectra, local periods, and skein
relations simultaneously.

For a finite abelian group \(G\), let
\(\widehat G=\Hom(G,\Uone)\).  The Fourier isomorphism
\[
  \C[G]\longrightarrow\C^{\widehat G},
  \qquad
  \sum_{g\in G}f(g)U^g
  \longmapsto
  \left(\chi\longmapsto\sum_{g\in G}f(g)\chi(g)\right)
\]
identifies the spectrum of \(r_f\) with the values of the function on the
right.  The polynomial whose distinct roots are these values annihilates
\(r_f\).  Proposition~\ref{prop:local-annihilators} converts its
coefficients into the corresponding local skein relation for any scalar
enhancement.

For \(\epsilon\in\{\pm1\}\) and \(Q\in\mathcal Q^\epsilon(q)\), the
signed shifted Gauss sum gives
\begin{equation}\label{eq:family-I-Gaussian-spectrum-set}
  \Lambda_{\mathrm I,\epsilon}(Q)
  =\{\gamma(Q)Q(\epsilon c)^{-1}:c\in A\}.
\end{equation}
For Family \(\mathrm{II}_E\) data \((E,Q_E,P_E)\) satisfying
Proposition~\ref{prop:family-II-coordinate-change} and \(b,s\in A\),
evaluation at the character
\(\chi_{b,s}(t,c)=\pair{b}{t}\pair{s}{c}\) gives
\begin{equation}\label{eq:family-II-Gaussian-spectrum-set}
  \Lambda_{\mathrm{II},E}(Q_E,P_E)
  =\{\gamma(Q_E)Q_E(E^{-1}b)^{-1}P_E(Es)^{-1}:b,s\in A\}.
\end{equation}
Repeated values are ignored when forming the annihilating polynomials.  Write
\[
  \Pi_{\mathrm I,\epsilon}(t)
  :=\prod_{\lambda\in\Lambda_{\mathrm I,\epsilon}(Q)}(t-\lambda)
  =\sum_j k_j^{\mathrm I,\epsilon}t^j,
  \qquad
  \Pi_{\mathrm{II},E}(t)
  :=\prod_{\lambda\in\Lambda_{\mathrm{II},E}(Q_E,P_E)}(t-\lambda)
  =\sum_j k_j^{\mathrm{II},E}t^j.
\]
Then every braid-local twist family satisfies
\begin{align}
  \sum_jk_j^{\mathrm I,\epsilon}
    \mathcal I^{\mathrm I,\epsilon}_Q(L_j)&=0,
  \label{eq:family-I-skein}\\
  \sum_jk_j^{\mathrm{II},E}\overline{\gamma(P_E)}^{\,j}
    \mathcal I^{\mathrm{II},E}_{Q_E,P_E}(L_j)&=0.
  \label{eq:family-II-skein}
\end{align}
Equations~\eqref{eq:family-I-Gaussian-spectrum-set} and
\eqref{eq:family-II-Gaussian-spectrum-set} follow by applying
Lemma~\ref{lem:general-shifted-Gauss-sum} to the relevant polarizations.
The preceding Fourier identification shows that \(\Pi_{\mathrm I,\epsilon}\)
and \(\Pi_{\mathrm{II},E}\) annihilate the corresponding local elements.
Proposition~\ref{prop:local-annihilators} and the parameters in
Theorem~\ref{thm:metric-family-enhancement} then gives the skein relations.

Family III has a noncommutative local algebra, but it admits an equally
explicit Fourier model.

Let \(V=\C[A]\) and define
\[
  \pi_0(U^{(a,b)})=q(a)^{-1}q(b)^{-1}Z_aX_{-b}
  \qquad(a,b\in A).
\]
Also define
\[
  \mathcal F_Ae_x=|A|^{-1/2}\sum_{y\in A}\pair{x}{y}e_y,
  \qquad
  M_{q^{-1}}e_x=q(x)^{-1}e_x,
\]
and set \(S_{\mathrm{III}}=\pi_0(r^{\mathrm{III}})\).

\begin{prop}\label{prop:family-III-local-Fourier-model}
The map \(\pi_0\) is a faithful \(*\)-representation of
\(\C_{\nu_{\mathrm{III}}}[A\oplus A]\), and
\begin{equation}\label{eq:family-III-Fourier-model}
  S_{\mathrm{III}}
  =\overline{\gamma(q)}\,\mathcal F_A M_{q^{-1}}.
\end{equation}
\end{prop}

\begin{proof}
The Pauli commutation relation gives
\[
  \pi_0(U^{(a,b)})\pi_0(U^{(c,d)})
  =\nu_{\mathrm{III}}((a,b),(c,d))
    \pi_0(U^{(a+c,b+d)})
  \qquad(a,b,c,d\in A),
\]
and the operators \(Z_aX_{-b}\) are linearly independent, proving
faithfulness.  Each \(\pi_0(U^{(a,b)})\) is unitary.  The canonical
involution of the unitary twisted group algebra is characterized by
\((U^{(a,b)})^*=(U^{(a,b)})^{-1}\), and the multiplication formula therefore
gives
\[
  \pi_0\bigl((U^{(a,b)})^*\bigr)
  =\pi_0(U^{(a,b)})^*.
\]
Hence \(\pi_0\) is a \(*\)-representation.  For \(x\in A\), direct application of
\(S_{\mathrm{III}}\), followed by the change of variable \(y=x-b\), gives
\[
  S_{\mathrm{III}}e_x
  =\overline{\gamma(q)}|A|^{-1/2}
    \sum_{y\in A} q(x)^{-1}\pair{x}{y}e_y,
\]
which is \eqref{eq:family-III-Fourier-model}.
\end{proof}

Let \(Je_x=e_{-x}\).  We first compute the square of the local operator.
From \eqref{eq:family-III-Fourier-model}, bicharacter multiplicativity, and
the signed shifted Gauss-sum identity for \(q^{-1}\),
\begin{align*}
  S_{\mathrm{III}}^2e_x
  &=\overline{\gamma(q)}^2|A|^{-1}q(x)^{-1}
    \sum_{y,z\in A}
    q(y)^{-1}\pair{x}{y}\pair{y}{z}e_z\\
  &=\overline{\gamma(q)}^2|A|^{-1}q(x)^{-1}
    \sum_{z\in A}
    \left(\sum_{y\in A}q(y)^{-1}\pair{x+z}{y}\right)e_z\\
  &=\overline{\gamma(q)}^3|A|^{-1/2}q(x)^{-1}
    \sum_{z\in A}q(x+z)e_z\\
  &=\overline{\gamma(q)}^3|A|^{-1/2}
    \sum_{z\in A}q(z)\pair{x}{z}e_z.
\end{align*}
Applying \(S_{\mathrm{III}}\) once more and using character orthogonality
now gives
\begin{align*}
  S_{\mathrm{III}}^3e_x
  &=\overline{\gamma(q)}^4|A|^{-1}
    \sum_{z,w\in A}\pair{x}{z}\pair{z}{w}e_w\\
  &=\overline{\gamma(q)}^4|A|^{-1}
    \sum_{w\in A}\left(\sum_{z\in A}\pair{z}{x+w}\right)e_w\\
  &=\overline{\gamma(q)}^4e_{-x}.
\end{align*}
Consequently,
\begin{equation}\label{eq:family-III-cube-inversion}
  S_{\mathrm{III}}^3=\overline{\gamma(q)}^{4}J,
  \qquad
  S_{\mathrm{III}}^6=I.
\end{equation}
Since \(J^2=I\) and Gauss--Milgram gives \(\gamma(q)^8=1\), the sixth power
is the identity.  Faithfulness transfers this relation to the twisted local
algebra and every position in the tower.  Consequently,
\[
  (r_i^{\mathrm{III}})^6=1,
  \qquad
  (R^{\mathrm{III}})^6=I.
\]

\begin{cor}\label{cor:family-III-skein-period}
For \(m\in\Z\) and \(L_m=\widehat{x\sigma_i^m y}\) as in
Proposition~\ref{prop:local-annihilators},
\begin{equation}\label{eq:family-III-six-twist-skein}
  \mathcal I^{\mathrm{III}}(L_{m+6})=\mathcal I^{\mathrm{III}}(L_m).
\end{equation}
If \(2a=0\) for every \(a\in A\), equivalently if \(2A=0\), then
\begin{equation}\label{eq:family-III-three-twist-skein}
  \mathcal I^{\mathrm{III}}(L_{m+3})
  =\overline{\gamma(q)}^{4}\mathcal I^{\mathrm{III}}(L_m).
\end{equation}
In particular, for \(A=\Z/2\Z\) and \(q(1)=-\ii\),
\[
  \mathcal I^{\mathrm{III}}(L_{m+3})+\mathcal I^{\mathrm{III}}(L_m)=0.
\]
\end{cor}

\begin{proof}
Apply Proposition~\ref{prop:local-annihilators} to \(t^6-1\).  Since
\(\alpha_{\mathrm{III}}=1\), no writhe factor occurs.  If \(2A=0\), then
the inversion \(J\) in \eqref{eq:family-III-cube-inversion} is the identity,
which gives the cubic relation.  In the stated two-element case,
\(\overline{\gamma(q)}^4=-1\).
\end{proof}

\begin{cor}\label{cor:metric-local-finite-order}
For every \(\epsilon\in\{\pm1\}\),
\(Q\in\mathcal Q^\epsilon(q)\), and Family \(\mathrm{II}_E\) datum
\((E,Q_E,P_E)\) satisfying
Proposition~\ref{prop:family-II-coordinate-change}, the corresponding
local operators in Families \(\mathrm I^\epsilon\), \(\mathrm{II}_E\), and
III have finite order.  The orders in the first two families are the least
common multiples of the orders of the scalars in
\eqref{eq:family-I-Gaussian-spectrum-set} and
\eqref{eq:family-II-Gaussian-spectrum-set}; the Family III order divides
\(6\).
\end{cor}

\begin{proof}
Lemma~\ref{lem:quadratic-phase-roots-of-unity} applies to
\(Q\in\mathcal Q^\epsilon(q)\) and to the nondegenerate polarizations of
\(Q_E\) and \(P_E\).  Hence every scalar in
\eqref{eq:family-I-Gaussian-spectrum-set} and
\eqref{eq:family-II-Gaussian-spectrum-set} is a root of unity.  The first
two operators are unitary and therefore semisimple; their orders are the
least common multiples of their eigenvalue orders.  The last statement is
\eqref{eq:family-III-cube-inversion}.
\end{proof}

Fix \(\epsilon\in\{\pm1\}\), \(Q\in\mathcal Q^\epsilon(q)\), and a
Family \(\mathrm{II}_E\) datum \((E,Q_E,P_E)\) satisfying
Proposition~\ref{prop:family-II-coordinate-change}.  For \(m\in\Z\),
let \(T(2,m)\) denote the closure of the two-strand braid \(\sigma_1^m\).
The corresponding values are
\begin{align}
  \mathcal I^{\mathrm I,\epsilon}_Q(T(2,m))
  &=\gamma(Q)^m\sum_{c\in A}Q(c)^{-m},
  \label{eq:family-I-torus-two}\\
  \mathcal I^{\mathrm{II},E}_{Q_E,P_E}(T(2,m))
  &=\gamma(P_E)^m\gamma(Q_E)^m
    \left(\sum_{b\in A}Q_E(b)^{-m}\right)
    \left(\sum_{s\in A}P_E(s)^{-m}\right),
  \label{eq:family-II-torus-two}\\
  \mathcal I^{\mathrm{III}}(T(2,m))
  &=|A|\,\Tr(S_{\mathrm{III}}^m).
  \label{eq:family-III-torus-two}
\end{align}
Equations~\eqref{eq:family-I-torus-two}
and~\eqref{eq:family-II-torus-two} follow by averaging the \(m\)-th powers of
the character values, while \eqref{eq:family-III-torus-two} follows from the
faithful local model of
Proposition~\ref{prop:family-III-local-Fourier-model}.

\begin{example}\label{ex:odd-prime-cyclic-Gaussian}
Let \(p\) be an odd prime, \(A=\F_p\),
\(\zeta_p=e^{2\pi\ii/p}\), and \(u\in\F_p^\times\).  For
\[
  q_u(a)=\zeta_p^{u a^2},
  \qquad
  Q_b(a)=\zeta_p^{u a^2+ba},
  \qquad b\in\F_p,
\]
\(q_u\) has polarization
\(\pair{a}{c}=\zeta_p^{2uac}\), and
\(Q_b\in\mathcal Q(q_u)\).
Let \(\left(\frac{u}{p}\right)\) denote the Legendre symbol, and compute
inverses in \(\F_p\).  Completion of the square gives
\begin{equation}\label{eq:odd-prime-Gauss-sum}
  \gamma(Q_b)
  =\varepsilon_p\left(\frac{u}{p}\right)
   \zeta_p^{-b^2(4u)^{-1}},
  \qquad
  \varepsilon_p=
  \begin{cases}
    1,&p\equiv1\pmod4,\\
    \ii,&p\equiv3\pmod4.
  \end{cases}
\end{equation}
This formula also shows that \(\gamma(Q_{-b})=\gamma(Q_b)\).  For the Family
II pair \((Q_b,Q_{-b})\), every element of \(\F_p\) is a sum of two squares.
For this pair, abbreviate
\[
  \Lambda_{\mathrm{II}}
  :=\Lambda_{\mathrm{II},\id}(Q_b,Q_{-b}),
  \qquad
  \Pi_{\mathrm{II}}:=\Pi_{\mathrm{II},\id}.
\]
Translating the two variables in the spectral formula gives
\[
  \Lambda_{\mathrm{II}}
  =\gamma(Q_b)\{\zeta_p^r:r\in\F_p\},
  \qquad
  \Pi_{\mathrm{II}}(t)=t^p-\gamma(Q_b)^p.
\]
Consequently, for every \(m\in\Z\), the braid-local twist family satisfies
\begin{equation}\label{eq:odd-prime-family-II-skein}
  \overline{\gamma(Q_b)}^{\,p}
  \mathcal I^{\mathrm{II}}_{Q_b,Q_{-b}}(L_{m+p})
  -\gamma(Q_b)^p
  \mathcal I^{\mathrm{II}}_{Q_b,Q_{-b}}(L_m)=0.
\end{equation}
\end{example}

\subsection{Exact evaluation}

\begin{theorem}\label{thm:metric-efficient-evaluation}
Fix the metric group and the phases defining Family I with either
polarization sign, any Family \(\mathrm{II}_E\), or Family III\@.
Given a braid word \(\xi\in\B_n\) of length \(\ell\), the
corresponding coefficient-trace link invariant can be computed exactly by a
deterministic classical algorithm in time polynomial in \(n+\ell\).
\end{theorem}

\begin{proof}
The case \(n=1\) is immediate, so assume \(n\geq2\).
Fix one of these families and write
\((G_\star,\nu_\star,\tau_\star,f_\star)\) for its tower and coefficient
data, where \(G_\star=A\) for Family I and \(G_\star=A\oplus A\) for Families
II and III\@.  Put
\[
  r_{\star,i}=\sum_{g\in G_\star}f_\star(g)U_i^g,
  \qquad
  \rho_n^{\T,\star}(\sigma_i)=r_{\star,i}.
\]
These data are fixed; only the braid word varies.

Let
\[
  \xi=\sigma_{i_\ell}^{\delta_\ell}\cdots
       \sigma_{i_1}^{\delta_1},
  \qquad 1\leq i_k<n,
  \quad \delta_k\in\{\pm1\},
  \quad 1\leq k\leq\ell.
\]
Expand each local factor in the tower basis.  For \(g_k\in G_\star\), a
positive crossing contributes \(f_\star(g_k)U_{i_k}^{g_k}\).  Since the
braid elements are unitary, a
negative crossing contributes
\[
  \overline{f_\star(-g_k)}\nu_\star(g_k,g_k)U_{i_k}^{g_k}.
\]
After multiplying the \(\ell\) factors and straightening them into normal
order, one obtains
\begin{equation}\label{eq:evaluation-tower-expansion}
  \rho_n^{\T,\star}(\xi)
  =\sum_{\mathbf g\in G_\star^\ell}
    C_\xi(\mathbf g)U^{L_\xi(\mathbf g)}.
\end{equation}
The homomorphism
\[
  L_\xi:G_\star^\ell\longrightarrow G_\star^{n-1}
\]
is the homomorphism recording the total exponent at every tower position.
The factor \(C_\xi(\mathbf g)\) is obtained from the crossing coefficients
and the factors contributed by \(\nu_\star\) and \(\tau_\star\).  Thus it is
a fixed scalar times a quadratic phase in the variables \(\mathbf g\).  Both
\(L_\xi\) and the quadratic exponent defining \(C_\xi\) can be constructed
using a polynomial number of operations in fixed finite rings.

Consequently, applying the coefficient trace gives
\begin{equation}\label{eq:evaluation-kernel-Gauss-sum}
  \varepsilon_n\bigl(\rho_n^{\T,\star}(\xi)\bigr)
  =\sum_{\mathbf g\in\ker L_\xi}C_\xi(\mathbf g).
\end{equation}
It remains to evaluate this finite quadratic Gauss sum efficiently.

Fix a cyclic decomposition
\[
  G_\star\cong\bigoplus_{a=1}^s\Z/d_a\Z.
\]
The group and all phases are fixed.  Lemma~\ref{lem:quadratic-phase-roots-of-unity}
shows that the values of the coefficient phases and their normalized Gauss
sums are roots of unity; the values of the bicharacters
\(\nu_\star\) and \(\tau_\star\) are roots of unity as well.  Choose an
integer \(L_0\) divisible by the orders of all these values and by every
\(d_a\), and put
\[
  M=2L_0,
  \qquad \zeta_M=e^{2\pi\ii/M}.
\]
This modulus depends only on the fixed metric data.

We record explicitly how the phases become quadratic polynomials.  Let
\(F:G_\star\to\Uone\) be any of the fixed quadratic phases, with
polarization \(B_F\), and let \(e_a\) be the chosen cyclic generators.  For
integer representatives \(x_a\), repeated use of the quadratic identity
gives
\begin{equation}\label{eq:evaluation-quadratic-coordinate-formula}
  F\left(\sum_ax_ae_a\right)
  =\prod_aF(e_a)^{x_a}
   \prod_aB_F(e_a,e_a)^{\binom{x_a}{2}}
   \prod_{a<b}B_F(e_a,e_b)^{x_ax_b}.
\end{equation}
A bicharacter has the analogous bilinear coordinate formula.  When its
values are written as powers of \(\zeta_M\), every coefficient multiplying
a binomial term in
\eqref{eq:evaluation-quadratic-coordinate-formula} is even, because
\(M/m\) is even for every phase order \(m\mid L_0\).  Hence each binomial
term is an ordinary integral quadratic polynomial modulo \(M\).  Applying
these formulas to all crossing and straightening factors produces an
explicit scalar \(\kappa_\xi\), which is a product of fixed normalization
constants, and an integral quadratic polynomial \(q_\xi\) such that
\begin{equation}\label{eq:evaluation-quadratic-phase-polynomial}
  C_\xi(\mathbf g)
  =\kappa_\xi\zeta_M^{q_\xi(\mathbf g)}.
\end{equation}
The polynomial is understood modulo \(M\) in the chosen cyclic coordinates.

We now remove the kernel condition without computing generators for the
kernel.  For \(\mathbf y\in G_\star^{n-1}\), character orthogonality gives
\[
  \boldsymbol{1}_{\{\mathbf y=0\}}
  =\frac{1}{|G_\star|^{n-1}}
    \sum_{\boldsymbol\chi\in\widehat G_\star^{\,n-1}}
      \boldsymbol\chi(\mathbf y),
\]
where
\(\boldsymbol\chi(\mathbf y)=\prod_{j=1}^{n-1}\chi_j(y_j)\).
Substituting \(\mathbf y=L_\xi(\mathbf g)\) into
\eqref{eq:evaluation-kernel-Gauss-sum} yields
\begin{equation}\label{eq:evaluation-unconstrained-Gauss-sum}
  \varepsilon_n\bigl(\rho_n^{\T,\star}(\xi)\bigr)
  =\frac{\kappa_\xi}{|G_\star|^{n-1}}
   \sum_{\substack{\mathbf g\in G_\star^\ell\\
                    \boldsymbol\chi\in\widehat G_\star^{\,n-1}}}
   \zeta_M^{q_\xi(\mathbf g)}
   \boldsymbol\chi\bigl(L_\xi(\mathbf g)\bigr).
\end{equation}
The character factor is bilinear in the coordinates of \(\mathbf g\) and
\(\boldsymbol\chi\).  Thus the full summand is the exponential of an
integral quadratic polynomial modulo \(M\).

Using the cyclic coordinates of \(G_\star\) and the dual coordinates of
\(\widehat G_\star\), the domain in
\eqref{eq:evaluation-unconstrained-Gauss-sum} has the form
\[
  G_\star^\ell\oplus\widehat G_\star^{\,n-1}
  \cong\bigoplus_{j=1}^N\Z/d_j'\Z,
  \qquad N=s(\ell+n-1),
\]
where every \(d_j'\) is one of the fixed integers \(d_a\) and divides
\(M\).  The polynomial obtained above is periodic modulo \(d_j'\) in its
\(j\)-th variable, since it represents a well-defined phase on this finite
group.  For any such periodic function \(h\),
\[
  \sum_{x\in\Z/d_j'\Z}h(x)
  =\frac{d_j'}{M}\sum_{\widetilde x\in\Z/M\Z}h(\widetilde x).
\]
Applying this identity in every variable converts
\eqref{eq:evaluation-unconstrained-Gauss-sum} into
\begin{equation}\label{eq:evaluation-standard-quadratic-sum}
  \varepsilon_n\bigl(\rho_n^{\T,\star}(\xi)\bigr)
  =\lambda_\xi
    \sum_{\mathbf z\in(\Z/M\Z)^N}
       \zeta_M^{F_\xi(\mathbf z)},
  \qquad
  \lambda_\xi
  =\frac{\kappa_\xi}{|G_\star|^{n-1}}
    \prod_{j=1}^N\frac{d_j'}{M},
\end{equation}
where \(F_\xi\) is an explicitly constructed integral quadratic polynomial.

The coefficient matrix and linear term of \(F_\xi\) require
\(O((n+\ell)^2)\) entries in the fixed ring \(\Z/M\Z\).  The algorithm of
\cite[Theorem~1.1]{CaiChenLiptonLu} evaluates the quadratic exponential sum in
\eqref{eq:evaluation-standard-quadratic-sum} exactly in time polynomial in
the number of variables and the input bit length.  Here \(M\) is fixed and
\(N=O(n+\ell)\), so this running time is polynomial in \(n+\ell\).

Finally, fix a number field containing \(\zeta_M\), the normalization
constants, and the enhancement parameters.  Its degree is independent of
the input.  The cited algorithm returns the exponential sum in a canonical
representation in \(\Q(\zeta_M)\) of polynomial bit length.  All remaining
scalar powers are computed by repeated squaring, and arithmetic in the fixed
number field preserves polynomial bit complexity.  Finally,
\[
  \alpha^{-\operatorname{wr}(\xi)}\beta^{-n}D^n
  \varepsilon_n\bigl(\rho_n^{\T,\star}(\xi)\bigr)
\]
is the oriented coefficient-trace invariant by
\eqref{eq:tower-oriented-invariant}.  This completes the exact
polynomial-time evaluation.
\end{proof}

\section{Clifford normalization and finite image}
\label{sec:finite-image}

We begin by recalling the generalized Pauli and Clifford groups associated
with a finite abelian group.  We then show that the local operators of the
three metric families lie in these Clifford groups and use this to prove
finiteness of their braid images.

For a finite abelian group \(B\), put \(V_B=\C[B]\) and
\(\widehat B=\Hom(B,\Uone)\).  On the standard basis of \(V_B\), define
\[
  X_be_x=e_{x+b},
  \qquad Z_\chi e_x=\chi(x)e_x
  \qquad(b,x\in B,\ \chi\in\widehat B).
\]
For \(n\geq1\), the \(n\)-fold Pauli group \(\mathcal P_n(B)\) is generated
by the unit scalar operators \(\Uone I\) and the local copies
\(X_{b,j}\), \(Z_{\chi,j}\), where \(1\leq j\leq n\).  Its projective
Clifford group is
\[
  \mathcal C_n(B)=
  N_{\mathrm U(V_B^{\ot n})}(\mathcal P_n(B))\,/\,\Uone I.
\]
When \(B=A\) is the fixed metric group, we identify
\(a\in A\) with the character \(\chi_a(x)=\pair{a}{x}\) and write
\(Z_a=Z_{\chi_a}\), as in Subsection~\ref{subsec:tensor-power-representations}.
The Pauli label group is
\[
  \mathbb V_{B,n}
  =
  B^{\oplus n}\oplus\widehat{B^{\oplus n}},
\]
equipped with the nondegenerate commutator bicharacter
\[
  \omega_{B,n}\bigl((b,\chi),(b',\chi')\bigr)
  =\chi(b')\chi'(b)^{-1}
  \qquad
  \bigl(b,b'\in B^{\oplus n},\ 
    \chi,\chi'\in\widehat{B^{\oplus n}}\bigr).
\]
Conjugation defines
\[
  \vartheta_{B,n}:\mathcal C_n(B)
  \longrightarrow
  \Aut(\mathbb V_{B,n},\omega_{B,n}),
  \qquad
  \Sigma_{B,n}=\operatorname{im}(\vartheta_{B,n}).
\]
The Clifford exact sequence is
\begin{equation}\label{eq:Clifford-exact-sequence}
  1\longrightarrow \mathbb V_{B,n}
  \longrightarrow \mathcal C_n(B)
  \longrightarrow \Sigma_{B,n}
  \longrightarrow 1.
\end{equation}
The quotient map records the induced action on Pauli labels, and
the kernel consists of projective Pauli translations, identified with
\(\mathbb V_{B,n}\) by the symplectic bicharacter.  This is the finite
Heisenberg--Clifford exact sequence; see \cite{GalindoClifford}.  Both
\(\mathbb V_{B,n}\) and \(\Sigma_{B,n}\) are finite.

For the three families, set
\[
  B_{\mathrm{I}}=B_{\mathrm{II}}=A,
  \qquad
  B_{\mathrm{III}}=A\oplus A.
\]
Thus \(V_{B_{\mathrm{I}}}=V_{B_{\mathrm{II}}}=\C[A]\) and
\(V_{B_{\mathrm{III}}}=W=\C[A]\otimes\C[A]\).

For \(Q,P\in\mathcal Q(q)\) and \(n\geq2\), let
\(\rho_n^{\mathrm I}\), \(\rho_n^{\mathrm{II}}\), and
\(\rho_n^{\mathrm{III}}\) denote the braid representations determined by
\[
  \rho_n^{\mathrm I}(\sigma_i)=R^{\mathrm I}_{Q,i},
  \qquad
  \rho_n^{\mathrm{II}}(\sigma_i)=R^{\mathrm{II}}_{Q,P,i},
  \qquad
  \rho_n^{\mathrm{III}}(\sigma_i)=R^{\mathrm{III}}_i.
\]

\begin{prop}\label{prop:metric-normalizes-Pauli}
Fix \(n\geq2\), \(1\leq i<n\), and \(Q,P\in\mathcal Q(q)\).  For every
Family \(\mathrm{II}_E\) datum \((E,Q_E,P_E)\) satisfying
Proposition~\ref{prop:family-II-coordinate-change}, each local operator
\[
  R^{\mathrm{I}}_{Q,i},
  \qquad
  R^{\mathrm{II}}_{Q,P,i},
  \qquad
  R^{\mathrm{II},E}_{Q_E,P_E,i},
  \qquad
  R^{\mathrm{III}}_i
\]
normalizes the corresponding Pauli group.  More precisely, Family I normalizes
\(\mathcal P_n(B_{\mathrm{I}})\), Families II and \(\mathrm{II}_E\) normalize
\(\mathcal P_n(B_{\mathrm{II}})\), and Family III normalizes
\(\mathcal P_n(B_{\mathrm{III}})\).
\end{prop}

\begin{proof}
It is enough to check the two tensor factors on which the local operator is
supported.  For the Family I operator write
\[
  U_i^t=X_{t,i}Z_{t,i+1},
  \qquad
  R^{\mathrm{I}}_{Q,i}=|A|^{-1/2}\sum_{t\in A}Q(t)U_i^t.
\]
The elementary commutation rules are
\[
  X_{a,i}U_i^t=U_i^tX_{a,i},
  \qquad
  Z_{a,i}U_i^t=\pair{a}{t}U_i^tZ_{a,i},
\]
and
\[
  X_{a,i+1}U_i^t=\pair{a}{t}^{-1}U_i^tX_{a,i+1},
  \qquad
  Z_{a,i+1}U_i^t=U_i^tZ_{a,i+1}.
\]
Thus \(X_{a,i}\) and \(Z_{a,i+1}\) commute with the local Gaussian.  For
\(Z_{a,i}\), moving the character factor into the Gaussian sum gives
\[
  Q(t)\pair{a}{t}
  =
  Q(t+a)Q(a)^{-1},
\]
After reindexing the sum, one obtains, up to the scalar \(Q(a)^{-1}\), a
Pauli operator times \(R^{\mathrm{I}}_{Q,i}\).  Therefore
\[
  R^{\mathrm{I}}_{Q,i}Z_{a,i}(R^{\mathrm{I}}_{Q,i})^{-1}
\]
is again a Pauli operator up to scalar.  The calculation for
\(X_{a,i+1}\) is the same with \(a\) replaced by \(-a\).  Hence the
Family I operator normalizes \(\mathcal P_n(B_{\mathrm{I}})\).

For the Family II operator use the change of variables
\[
  s=x+y,\qquad d=x,
\]
under which
\[
  R^{\mathrm{II}}_{Q,P}\cong M_{P^{-1}}\otimes G_Q,
  \qquad
  G_Qe_d=|A|^{-1/2}\sum_{t\in A}Q(t)e_{d+t}.
\]
The diagonal multiplier \(M_{P^{-1}}\) satisfies
\[
  M_{P^{-1}}X_aM_{P^{-1}}^{-1}
  =
  P(a)^{-1}X_aZ_{-a},
  \qquad
  M_{P^{-1}}Z_bM_{P^{-1}}^{-1}=Z_b,
\]
and therefore it normalizes the one-particle Pauli group.  Similarly \(G_Q\) commutes
with all \(X_a\), and reindexing its Gaussian sum gives
\[
  G_QZ_bG_Q^{-1}
  =
  Q(b)X_bZ_b.
\]
Thus both tensor factors are Clifford operators.  The coordinate
permutation \((x,y)\mapsto(x+y,x)\) also sends Pauli operators to Pauli
operators.  Undoing it shows that \(R^{\mathrm{II}}_{Q,P,i}\) normalizes
\(\mathcal P_n(B_{\mathrm{II}})\).

For Family \(\mathrm{II}_E\), use
\eqref{eq:family-II-E-factorization}.  The map
\(a\mapsto B_E(a,-)\) identifies \(A\) with \(\widehat A\).  More
explicitly,
\[
  B_E(a,x)=\pair{Ea}{x},
  \qquad
  Z_{B_E(a,-)}=Z_{Ea}
  \qquad(a,x\in A),
\]
and the inverse character corresponds to \(Z_{-Ea}\).  Thus all character
multipliers occurring in the translation and reindexing calculations lie
in the original generalized Pauli group.  Those calculations show that both
\(M_{\widetilde P^{-1}}\) and \(G_{Q_E}\) normalize the generalized Pauli
group.  The coordinate permutation \(\Xi\) also preserves that group; hence
\(R^{\mathrm{II},E}_{Q_E,P_E,i}\) normalizes
\(\mathcal P_n(B_{\mathrm{II}})\).

It remains to treat Family III\@.  Write \(\doteq\) for equality up to a
scalar.  In the faithful local model of
Proposition~\ref{prop:family-III-local-Fourier-model}, one has
\[
  S_{\mathrm{III}}
  =\overline{\gamma(q)}\,\mathcal F_AM_{q^{-1}}.
\]
The elementary Fourier conjugation formulas give
\[
  S_{\mathrm{III}}Z_aS_{\mathrm{III}}^{-1}=X_{-a},
  \qquad
  S_{\mathrm{III}}X_aS_{\mathrm{III}}^{-1}
  \doteq Z_aX_a.
\]
Since
\[
  \pi_0(U^{(a,b)})
  =q(a)^{-1}q(b)^{-1}Z_aX_{-b},
\]
it follows that
\begin{equation}\label{eq:family-III-local-symplectic-action}
  r^{\mathrm{III}}U^{(a,b)}(r^{\mathrm{III}})^{-1}
  \doteq
  U^{\mathsf T_{\mathrm{III}}(a,b)},
  \qquad
  \mathsf T_{\mathrm{III}}(a,b)=(-b,a+b).
\end{equation}
Faithfulness of \(\pi_0\) transfers this conjugation identity back to the
local twisted group algebra.  Moreover,
\[
  \mathsf T_{\mathrm{III}}^3=-I,
  \qquad
  \mathsf T_{\mathrm{III}}^6=I.
\]

Inside the physical Pauli label group
\(\mathbb V_{B_{\mathrm{III}},2}\), let \(L\) be the subgroup represented
by the Pauli words \(Y^{(a,b)}\).  Its restricted commutator bicharacter is
\(\tau_{\mathrm{III}}\), which is nondegenerate.  Put
\(\mathbb V=\mathbb V_{B_{\mathrm{III}},2}\).  Nondegeneracy of the
restriction gives
\[
  L\cap L^\perp=0.
\]
The ambient commutator bicharacter on \(\mathbb V\) is also nondegenerate.
Consequently, the homomorphism
\[
  \mathbb V\longrightarrow\widehat L,
  \qquad
  v\longmapsto\omega_{B_{\mathrm{III}},2}(v,-)|_L,
\]
is surjective and has kernel \(L^\perp\).  Indeed, ambient nondegeneracy
identifies \(\mathbb V\) with \(\widehat{\mathbb V}\), and every character
of the subgroup \(L\) extends to \(\mathbb V\).  Therefore
\[
  |L^\perp|=\frac{|\mathbb V|}{|L|},
  \qquad
  |L||L^\perp|=|\mathbb V|.
\]
Together with the trivial intersection, this proves
\[
  \mathbb V=L\oplus L^\perp.
\]
Equation~\eqref{eq:family-III-local-symplectic-action} shows that
\(R^{\mathrm{III}}\) acts on \(L\) by \(\mathsf T_{\mathrm{III}}\).  Every Pauli operator with
label in \(L^\perp\) commutes with all \(Y^{(a,b)}\), and therefore with
\(R^{\mathrm{III}}\); thus the action on \(L^\perp\) is trivial.  It follows
that \(R^{\mathrm{III}}\) normalizes the full physical Pauli group.  The
same local argument applies at every position \(i\).
\end{proof}

\begin{prop}\label{prop:metric-symplectic-extension}
For \(n\geq2\), \(Q,P\in\mathcal Q(q)\), and
\(\star\in\{\mathrm I,\mathrm{II},\mathrm{III}\}\), let
\[
  H_n^\star=\rho_n^\star(\B_n)\cdot\Uone I\,/\,\Uone I
\]
be its projective braid image, and let \(\Gamma_n^\star\) be its image in
\(\Sigma_{B_\star,n}\).  There is a short exact sequence
\[
  1\longrightarrow K_n^\star
  \longrightarrow H_n^\star
  \longrightarrow \Gamma_n^\star
  \longrightarrow 1,
\]
where \(K_n^\star\) is a finite abelian group of projective Pauli
translations.  In particular \(H_n^\star\) is finite.
\end{prop}

\begin{proof}
By Proposition~\ref{prop:metric-normalizes-Pauli}, \(H_n^\star\) is a
subgroup of \(\mathcal C_n(B_\star)\).  Restrict the quotient map in
\eqref{eq:Clifford-exact-sequence} to \(H_n^\star\).  Its image is
\(\Gamma_n^\star\) by definition, and its kernel is
\[
  K_n^\star=H_n^\star\cap\mathbb V_{B_\star,n}.
\]
Exactness gives the sequence in the proposition.  Since
\(\mathbb V_{B_\star,n}\) is finite abelian, its subgroup \(K_n^\star\) is
finite abelian as well.  The
group \(\Gamma_n^\star\) is a subgroup of the finite group
\(\Sigma_{B_\star,n}\).  Hence \(H_n^\star\) is finite.
\end{proof}

\begin{lemma}\label{lem:finite-projective-lift}
Fix \(n\geq2\).  Let \(R\in\mathrm U(V\otimes V)\) be a Yang--Baxter operator of finite
order.  If the projective image of the associated representation
\(\rho_{R,n}:\B_n\to\mathrm U(V^{\ot n})\) is finite, then its ordinary
unitary image is finite.
\end{lemma}

\begin{proof}
Choose \(N\geq1\) with \(R^N=I\), and put \(D=(\dim V)^n\).  All local
copies \(R_i\) have the same determinant, say \(\delta\), and
\(\delta^N=1\).  Therefore, for every braid \(\xi\),
\[
  \det\rho_{R,n}(\xi)=\delta^{\operatorname{wr}(\xi)}.
\]
If \(\rho_{R,n}(\xi)=\lambda I_D\), then
\(\lambda^D=\delta^{\operatorname{wr}(\xi)}\), and hence
\(\lambda^{DN}=1\).  Thus the
scalar subgroup of the braid image is contained in the finite group
\(\mu_{DN}I_D\).  Since the quotient by this scalar subgroup is the finite
projective image, the ordinary image is finite as well.
\end{proof}

\begin{theorem}\label{thm:metric-full-finite}
For \(n\geq2\), \(Q,P\in\mathcal Q(q)\), and
\(\star\in\{\mathrm I,\mathrm{II},\mathrm{III}\}\), the ordinary
unitary braid image \(\rho_n^\star(\B_n)\) is finite for the
corresponding representation \(\rho_n^\star\).
\end{theorem}

\begin{proof}
The projective image is finite by
Proposition~\ref{prop:metric-symplectic-extension}, and the local operator
has finite order by Corollary~\ref{cor:metric-local-finite-order}.  Apply
Lemma~\ref{lem:finite-projective-lift}.
\end{proof}

\begin{cor}\label{cor:extended-metric-full-finite}
For every \(n\geq2\), each Family I representation determined by
\(Q\in\mathcal Q^-(q)\) and each Family \(\mathrm{II}_E\)
representation determined by a datum \((E,Q_E,P_E)\) satisfying
Proposition~\ref{prop:family-II-coordinate-change} have finite ordinary
braid image.
\end{cor}

\begin{proof}
The signed Family I conjugation calculation uses the same quadratic
reindexing with the inverse character; hence its local operators normalize the
same finite Pauli group.  The Family \(\mathrm{II}_E\) operators normalize the
Pauli group by Proposition~\ref{prop:metric-normalizes-Pauli}.  Their local
orders are finite by Corollary~\ref{cor:metric-local-finite-order}.  The
Clifford exact sequence and Lemma~\ref{lem:finite-projective-lift} therefore
apply to both extensions.
\end{proof}

\section{Virtual abelianness of the cyclic families}
\label{sec:cyclic-virtual-abelianness}

We compare each parameter value with a Gaussian point on the same component.
The corresponding generator ratios lie in a commutative
two-torsion sector even when the phase parameter has infinite order.

\begin{lemma}
\label{lem:commutative-sector-deformation}
Fix \(n\geq2\) and two parameter values \(t,t_0\).  Let \(\A\) be a
finite-dimensional algebra, let \(H\subset\A\) be a commutative unital
subalgebra, and suppose that
\[
  x_i(t),x_i(t_0)\in\A^\times\qquad(1\leq i<n)
\]
normalize \(H^\times\).  Assume that
\[
  x_i(t)x_i(t_0)^{-1}\in H^\times
  \qquad(1\leq i<n).
\]
If \(\Gamma_{t_0}=\langle x_1(t_0),\ldots,x_{n-1}(t_0)\rangle\) is finite,
then \(\Gamma_t=\langle x_1(t),\ldots,x_{n-1}(t)\rangle\) is virtually
abelian.  More precisely,
\[
  \Gamma_t\cap H^\times
\]
is an abelian normal subgroup and
\[
  [\Gamma_t:\Gamma_t\cap H^\times]\leq |\Gamma_{t_0}|.
\]
\end{lemma}

\begin{proof}
Let \(N=N_{\A^\times}(H^\times)\) be the normalizer of \(H^\times\).  The hypothesis
says that \(x_i(t)\) and \(x_i(t_0)\) have the same image in
\(N/H^\times\).  Hence the images of \(\Gamma_t\) and \(\Gamma_{t_0}\)
in this quotient coincide.  The image of \(\Gamma_t\) is therefore finite, of order at
most \(|\Gamma_{t_0}|\).  Its kernel is
\(\Gamma_t\cap H^\times\), which is normal in \(\Gamma_t\) and abelian
because \(H\) is commutative.
\end{proof}

We apply Lemma~\ref{lem:commutative-sector-deformation} in the cyclic tower.
Fix \(n\geq2\), let \(d=8\), and put
\[
  J_i=U_i^{d/2},
  \qquad
  \Htorus_{n-1}=\C[J_1,\ldots,J_{n-1}],
\]
where \(1\leq i<n\).  Lemma~\ref{lem:cyclic-two-torsion-sectors} shows that the \(J_i\)
are commuting involutions.  Thus
\begin{equation}\label{eq:cyclic-sector-algebra}
  \Htorus_{n-1}\cong\C[(\Z/2\Z)^{n-1}]
\end{equation}
is commutative.  Let
\[
  \pi_n^{\mathrm{cyc}}:\T_n(\Z/d\Z,\tau)
  \longrightarrow\End((\C^d)^{\ot n}),
  \qquad U_i\longmapsto X_iZ_{i+1},
\]
be the Pauli realization used in Proposition~\ref{prop:d8-families}, and put
\(\Htorus^{\mathrm{mat}}_{n-1}
=\pi_n^{\mathrm{cyc}}(\Htorus_{n-1})\).

\begin{lemma}\label{lem:cyclic-family-normalizer}
For \(1\leq m\leq8\), \(t\in\C^\times\), and \(1\leq i<n\), the local
generator \(R_i(C_m(t))\) normalizes
\((\Htorus^{\mathrm{mat}}_{n-1})^\times\).  Moreover, for
\(1\leq m\leq8\), \(t,t_0\in\C^\times\), and \(1\leq i<n\),
\begin{align}
  R_i(C_m(t))R_i(C_m(t_0))^{-1}
  &\in(\Htorus^{\mathrm{mat}}_{n-1})^\times.
  \label{eq:d8-sector-ratio}
\end{align}
The same statements hold for any nonzero scalar normalization of the local
generators.
\end{lemma}

\begin{proof}
For \(1\leq m\leq8\), let \(f_{m,t}(y)\) be the coefficient polynomial of
\(C_m(t)\), and put \(f(y)=f_{m,t}(y)\).  Its
evaluation at \(U_i\) commutes with \(J_i\) and with every distant \(J_j\).
Since
\[
  U_i^kJ_{i+1}=(-1)^kJ_{i+1}U_i^k,
\]
one has
\[
  f(U_i)J_{i+1}=J_{i+1}f(-U_i),
\]
and the analogous formula holds with \(i-1\) in place of \(i+1\).
For the two templates, the factorizations in the proof of
Proposition~\ref{prop:d8-families} give
\begin{equation}\label{eq:d8-sign-change}
  f_{m,t}(-y)=-y^4f_{m,t}(y)
  \quad\text{for type I},
  \qquad
  f_{m,t}(-y)=y^4f_{m,t}(y)
  \quad\text{for type II}.
\end{equation}

We first justify invertibility for every nonzero parameter.  In the
orthogonal decompositions from Proposition~\ref{prop:d8-families}, write
\[
  x_i(z)=a_i+zb_i,
\]
where \(a_i,b_i\) are orthogonal partial unitaries with complementary
initial and final projections
\(e_{i,0},e_{i,1}\in\Htorus^{\mathrm{mat}}_{n-1}\).  Thus, for
\(z\neq0\),
\begin{equation}\label{eq:sector-parameter-inverse}
  x_i(z)^{-1}=a_i^*+z^{-1}b_i^*.
\end{equation}
Indeed, orthogonality and the complementary projection identities give
\[
  x_i(z)(a_i^*+z^{-1}b_i^*)
  =a_ia_i^*+b_ib_i^*=e_{i,0}+e_{i,1}=1,
\]
and the product in the opposite order is the identity as well.  Up to the
fixed nonzero normalization scalar, each \(f(U_i)\) is one of these
operators and is therefore invertible for every parameter in \(\C^\times\).

In \eqref{eq:d8-sign-change}, write
\[
  f(-U_i)=\epsilon J_if(U_i),
  \qquad \epsilon\in\{\pm1\}.
\]
For either neighboring index \(j=i\pm1\) that occurs in the tower,
\begin{equation}\label{eq:cyclic-sector-conjugation}
\begin{aligned}
  f(U_i)J_jf(U_i)^{-1}
  &=J_jf(-U_i)f(U_i)^{-1}\\
  &=\epsilon J_iJ_j.
\end{aligned}
\end{equation}
The local generator commutes with \(J_i\) and every distant \(J_j\).
Consequently, conjugation carries every generator of
\(\Htorus^{\mathrm{mat}}_{n-1}\) into that algebra.  Its restriction is an
injective endomorphism of a finite-dimensional algebra and hence an
automorphism.  Therefore every local generator normalizes
\((\Htorus^{\mathrm{mat}}_{n-1})^\times\).

It remains to compare two parameters.  For \(z,z_0\neq0\),
\begin{equation}\label{eq:sector-parameter-ratio}
  x_i(z)x_i(z_0)^{-1}
  =a_ia_i^*+\frac{z}{z_0}b_ib_i^*
  =e_{i,0}+\frac{z}{z_0}e_{i,1}
  \in(\Htorus^{\mathrm{mat}}_{n-1})^\times.
\end{equation}
The inverse of the last element is
\(e_{i,0}+(z_0/z)e_{i,1}\).
This proves the ratio statement.  Scalar normalizations cancel in the
same calculation up to a nonzero scalar, which also belongs to
\((\Htorus^{\mathrm{mat}}_{n-1})^\times\).
\end{proof}

The signed Gaussian family supplies the required finite-image base points.
Recall that \(\xi=e^{\pi\ii/8}\) and \(w=\xi^2\).  The relevant base points
are
\begin{equation}\label{eq:d8-Gaussian-basepoints}
\begin{array}{@{}ccccccccc@{}}
\toprule
 m&1&2&3&4&5&6&7&8\\
\midrule
 t_0&\xi&\xi^3&\xi^{-1}&\xi^{-3}&\xi^{-3}&\xi^3&\xi^{-1}&\xi\\
 \epsilon&-&+&+&-&+&-&-&+\\
\bottomrule
\end{array}
\end{equation}
For each \(m\), these choices make \(C_m(t_0)\) a signed quadratic phase.
Indeed, for \(\epsilon\in\{\pm1\}\), a cyclic quadratic phase is
uniquely determined by \(Q(1)\) and satisfies
\[
  Q(k)=Q(1)^k w^{\epsilon k(k-1)/2};
\]
for the type I template \(Q(1)=\alpha t_0\), while for the type II template
\(Q(1)=\delta t_0\).  Substitution gives the eight coefficient vectors of
Proposition~\ref{prop:d8-families}.  The braid images at all these base
points are finite by Theorem~\ref{thm:metric-full-finite} and
Corollary~\ref{cor:extended-metric-full-finite}.

\begin{theorem}\label{thm:cyclic-family-virtually-abelian}
For every \(n\geq2\), every \(t\in\C^\times\), and \(1\leq m\leq8\), the
braid representation generated by
\(\widehat R_{m,t}=R(C_m(t))/\sqrt8\) is virtually abelian.  Let
\(\Gamma_n\) denote its braid image.  Then
\[
  \Gamma_n\cap(\Htorus^{\mathrm{mat}}_{n-1})^\times
\]
is an abelian normal subgroup of finite index.  The same conclusion holds
for the unnormalized generator.  The normalized generator is unitary when
\(|t|=1\).
\end{theorem}

\begin{proof}
Apply Lemma~\ref{lem:cyclic-family-normalizer} to the fixed normalization.
It verifies the normalizer and sector-ratio hypotheses of
Lemma~\ref{lem:commutative-sector-deformation}.  Use the corresponding base
point in \eqref{eq:d8-Gaussian-basepoints}.  The normalized base-point
images are finite; hence the conclusion follows.  An unnormalized image is a
subgroup of the group generated by the normalized image and one central
cyclic scalar; virtual abelianness is therefore unchanged.
\end{proof}

Thus the unitary families may have infinite image, as occurs when a
unit-modulus parameter is not a root of unity by
Corollary~\ref{cor:cyclic-family-infinite-order}.  For any image
\(\Gamma_n\) in Theorem~\ref{thm:cyclic-family-virtually-abelian}, however,
\[
  \Gamma_n\big/
  \bigl(\Gamma_n\cap(\Htorus^{\mathrm{mat}}_{n-1})^{\times}\bigr)
\]
is finite, and the denominator is an abelian normal subgroup.

\section{Pauli-polynomial solutions in low dimensions}
\label{sec:pauli-polynomial}

We study full-support solutions for specified low-dimensional Pauli
directions.
The exact algebraic search and the subsequent imposition of unitarity are
treated as separate steps.  Each search specifies its Pauli direction and allowed
equivalences.  Statements labeled \emph{Computation} in this and
the subsequent low-dimensional sections record the output of exact
calculations in the stated arithmetic models and under the declared
identifications.  They are used as evidence for the proposed classification.
All assertions about explicitly displayed operators are verified directly
and do not depend on completeness of the computational searches.

\subsection{Equivalences used in the computations}
\label{subsec:computational-equivalences}

The identifications used in the searches are broader than the local unitary
equivalence \eqref{eq:unitary-local-equivalence}, and every computational
statement specifies which ones are allowed.  In these statements, a
\emph{strict Yang--Baxter operator} satisfies
\eqref{eq:background-YBE} exactly, rather than only up to a scalar.  In a
Clifford enumeration for a finite label group \(B\), a
\emph{symplectic solution} is an element of \(\Sigma_{B,2}\) whose local
copies satisfy the braid relation.  A \emph{projective lift} is an element of
\(\mathcal C_2(B)\) that induces such a symplectic solution and satisfies the
braid relation in the projective Clifford group.  At coefficient level we use
projective rescaling, multiplication by a character, and automorphisms of
\(G\) preserving \(\tau\).  Fix \(d\geq2\), choose a primitive \(d\)-th
root \(w\), put \(G=\Z/d\Z\), and set
\(\tau(a,b)=w^{ab}\) for \(a,b\in G\).
For a coefficient function
\(a:G\to\C\), write \(a_t=a(t)\).  When \(G=\Z/d\Z\), these operations
act by
\[
  a_t\longmapsto c\,x^ta_{\psi^{-1}(t)}.
\]
Here \(c\in\C^\times\), \(x\in\C^\times\) satisfies \(x^d=1\), and
\(\psi\in\Aut(G)\) preserves \(\tau\).
We omit Clifford enumerations in local dimensions \(4\) and \(9\).  For
these square dimensions, conjugation to the Zak basis gives a monomial
realization of the Clifford group
\cite{ApplebyBengtssonBrierleyGrasslGrossLarsson}.  Applying this basis
change on every local tensor factor makes each two-site Clifford candidate
monomial, so these dimensions cannot produce a new local-equivalence class
within the Clifford ansatz.
At matrix level, local gauge sends \(R\) to
\((S\otimes S)^{-1}R(S\otimes S)\).  We distinguish local Clifford gauge,
where \(S\) is a one-particle Clifford operator, from arbitrary local
\(\GL(V)\)-gauge, where \(S\in\GL(V)\), and from Galois conjugation.  Galois
conjugation preserves the
cyclotomic equations and projective unitarity at cyclotomic points, but it
is not a unitary equivalence in a fixed complex Hilbert space.

We also use the conditional Doikou--Smoktunowicz \(T\)-move.  For a
Yang--Baxter operator \(R\in\End(V\otimes V)\), if \(T\in\GL(V)\)
satisfies \((T\otimes T)R=R(T\otimes T)\), set
\[
  R^T=(T\otimes I)R(T^{-1}\otimes I)
     =(I\otimes T^{-1})R(I\otimes T).
\]
Then \(R^T\) is a Yang--Baxter operator and gives an equivalent braid
representation.  For every \(n\geq2\), on \(V^{\ot n}\),
\[
  \Theta_n=T^{n-1}\otimes T^{n-2}\otimes\cdots\otimes T\otimes I
\]
conjugates every local copy of \(R\) to the corresponding copy of \(R^T\).
See \cite{DS} for the involutive case.

\subsection{Cyclic Pauli-polynomial ansatz}

\subsubsection{The Pauli direction \(X\otimes Z\)}

Let \(X,Z\in\mathrm U(\C^d)\) be the shift and clock matrices
\[
  Xe_j=e_{j+1},
  \qquad
  Ze_j=w^j e_j,
  \qquad
  ZX=wXZ
  \qquad(j\in\Z/d\Z).
\]
A one-generator
two-qudit Pauli-polynomial ansatz has the form
\begin{equation}\label{eq:Pauli-polynomial-ansatz}
  R(a)=\sum_{t=0}^{d-1}a_tP^t,
  \qquad
  P=X\otimes Z
\end{equation}
or, before Galois normalization,
\[
  \sum_{t=0}^{d-1}a_t(X^t\otimes Z^{\alpha t}),
  \qquad
  \alpha\in(\Z/d\Z)^\times.
\]
The Yang--Baxter equation is homogeneous in \(a=(a_0,\ldots,a_{d-1})\).
We therefore work projectively.

\begin{definition}
A coefficient vector \(a=(a_0,\ldots,a_{d-1})\) has \emph{full support} if
all \(a_t\) are nonzero.  On the full-support locus we normalize
\(a_0=1\).  The \emph{compact coefficient phase torus} is
\begin{equation}\label{eq:compact-coefficient-phase-torus}
  \mathbb T_d^{\mathrm{coef}}
  =
  \{(a_1,\ldots,a_{d-1})\in(\C^\times)^{d-1}:
    |a_t|=1\text{ for }1\leq t<d\}.
\end{equation}
Coefficientwise unimodularity is an additional restriction, distinct from
projective unitarity of \(R(a)\).  The latter requires
\(C_a(u)=0\) for every nonzero shift \(u\), but neither condition by itself
implies the other.
\end{definition}

For \(i,j,k,\ell\in\Z/d\Z\), define the Pauli direction
\[
  D=X^iZ^j\otimes X^kZ^\ell,
  \qquad
  \Delta(D)=i\ell-jk.
\]
A common symplectic change of one-qudit Pauli coordinates sends the two
exponent columns to \((1,0)^{\mathsf T}\) and
\((0,\Delta(D))^{\mathsf T}\) whenever \(\Delta(D)\) is a unit.  For odd
\(d\), this change is implemented by a one-qudit Clifford gate; hence the
corresponding direction is locally Clifford equivalent, up to scalar Pauli
phase, to \(X\otimes Z^{\Delta(D)}\).  In even composite dimensions the
lift requires additional choices; the \(d=8\) completeness statement is
therefore made directly for the Pauli direction \(X\otimes Z\).

A cyclotomic Galois automorphism taking \(w\) to \(w^{\alpha^{-1}}\)
maps the coefficients for \(X\otimes Z^\alpha\) to those for
\(X\otimes Z\).  We use this
only as an arithmetic normalization, never as a local unitary equivalence.

\subsubsection{The coefficient torus and the elimination ideal}

After this normalization, Lemma~\ref{lem:coefficient-ybe} becomes
\begin{equation}\label{eq:cyclic-coefficient-ybe}
  a_t\sum_{c\in\Z/d\Z}a_{s-c}a_cw^{-ct}
  =
  a_s\sum_{c\in\Z/d\Z}a_ca_{t-c}w^{-sc}
  \qquad(s,t\in\Z/d\Z).
\end{equation}
For exact elimination, choose a cyclotomic coefficient field, introduce
variables \(b_0,\ldots,b_{d-1}\), set \(a_0=b_0=1\), and adjoin
\begin{align}
  a_tb_t-1&=0\qquad(1\leq t<d),
  \label{eq:formal-inverse-equations}\\
  \sum_{x\in\Z/d\Z}b_xa_{x+u}&=0
    \qquad(u\in\Z/d\Z\setminus\{0\}).
  \label{eq:formal-unitarity-equations}
\end{align}
The \(b_t\) are formal inverses in the auxiliary complex variety; they are
not conjugates until one intersects with the real form
\begin{equation}\label{eq:unitary-real-form}
  b_t=\overline{a_t}\qquad(1\leq t<d).
\end{equation}
On this real form, the first equations impose coefficientwise unimodularity
and the second impose the remaining coefficient conditions for projective
unitarity.  Thus primary
decomposition of the complex ideal, intersection with
\eqref{eq:unitary-real-form}, and quotienting by the declared symmetries are
three separate stages.  The parameters of an algebraic family range over
\(\C^\times\); its unitary locus is the corresponding circle.

\subsubsection{Gaussian solutions for the normalized cyclic ansatz}

For the normalized cyclic ansatz, a \emph{Gaussian coefficient vector} is a
vector obtained from a function \(Q\) satisfying
\[
  Q:\Z/d\Z\longrightarrow\C^\times,
  \qquad
  Q(s+t)=Q(s)Q(t)w^{st},
  \qquad Q(0)=1,
\]
by setting \(a_t=Q(t)\).  The values lie in a suitable cyclotomic extension.
Multiplication by a character gives
all quadratic phases with the same polarization.  Together with their
cyclotomic Galois conjugates, these character rescalings give the coefficient
normalizations of the cyclic Family I locus.

\begin{remark}\label{rem:signed-family-I-fixed-realization}
For \(a,b,t\in\Z/d\Z\), the quadratic coefficients
\(Q_{a,b}(t)=w^{at^2+bt}\) have the two polarization signs precisely when
\(2a\equiv+1\pmod d\) or \(2a\equiv-1\pmod d\), respectively.  Thus the
two branches are intrinsic to the tower and should not be identified merely
by Galois conjugation.
\end{remark}

\subsubsection{Low-dimensional computations}

\begin{computation}
\label{comp:low-dimensional-Gaussian-search}
For \(d\in\{2,3,5,6,7,9,10,11\}\), our exact primary-decomposition
calculation for the
full-support ideals defined by
\eqref{eq:cyclic-coefficient-ybe},
\eqref{eq:formal-inverse-equations} and
\eqref{eq:formal-unitarity-equations}, followed by intersection with
the unitary real form \eqref{eq:unitary-real-form}, returned only Gaussian
points in the compact coefficient phase torus for the Pauli direction
\(X\otimes Z\), up to character rescaling.  Cyclotomic Galois conjugation
gives the corresponding Galois-conjugate Gaussian vectors.  The computation
concerns full-support solutions for this Pauli direction; sparse solutions,
other Pauli directions, and general unitary Yang--Baxter operators are
outside its scope.
\end{computation}

\subsubsection{\texorpdfstring{Positive-dimensional output in \(d=8\)}
{Positive-dimensional output in d=8}}

\begin{computation}\label{comp:d8-family-completeness}
Within the same fixed-direction search space, and before quotienting by
character rescaling or Galois conjugation, the primary-decomposition
calculation returned the following dimension-\(8\) unitary real loci.
\begin{enumerate}[label=(\roman*)]
\item The positive-dimensional components meeting the compact
coefficient phase torus are the eight algebraic families \(C_m(t)\) of
Proposition~\ref{prop:d8-families}.  Their unitary real loci are
\(|t|=1\).
\item The unsliced \(d=8\) ideal also has sixteen isolated formal
components.  None meets the unitary real form
\eqref{eq:unitary-real-form}.
\end{enumerate}
\end{computation}

\begin{prop}\label{prop:d8-family-equivalence-orbits}
For the fixed \(d=8\) Pauli direction $X\otimes Z$, character rescaling and
reparametrization of
\(t\) divide the eight families into the two orbits
\[
  \{C_1,C_4,C_6,C_7\},
  \qquad
  \{C_2,C_3,C_5,C_8\}.
\]
The Galois automorphism \(w\mapsto w^3\) exchanges these two character
orbits.  Hence the eight components before quotienting form one orbit under
character rescaling, parameter reparametrization, and Galois conjugation.
\end{prop}

\begin{proof}
For the character \(k\mapsto w^{rk}\), direct substitution gives
\begin{align*}
 C_1(t)&\longmapsto
 C_7(w^7t),\ C_4(w^6t),\ C_6(w^5t)
 &&\text{for }r=1,2,3,\\
 C_2(t)&\longmapsto
 C_8(w^7t),\ C_3(w^6t),\ C_5(w^5t)
 &&\text{for }r=1,2,3.
\end{align*}
These transformations give the two character orbits in the proposition.  The
character with exponent \(r+4\) differs from the one with exponent \(r\)
only by the reparametrization \(t\mapsto-t\).  Hence the identity together
with the three displayed substitutions for each starting family exhausts all
eight characters and does not join the two sets.  Finally,
the automorphism \(w\mapsto w^3\), acting trivially on the formal parameter
\(t\), sends \(\ii=w^2\) to \(-\ii\), and in
particular sends \(C_1(t)\) to \(C_3(t)\).  It therefore joins the two
orbits.
\end{proof}

\subsubsection{Composite products}
\label{subsubsec:Pauli-composite-products}

Corollary~\ref{cor:coefficient-products} identifies several composite
solutions outside the full-support search space.  Using
\[
  \Z/6\Z\cong\Z/2\Z\oplus\Z/3\Z,
\]
let \(R_d^{\mathrm G}\) denote a fixed cyclic Gaussian representative in
dimension \(d\), and let \(\mathbf 1_d\) denote the identity Yang--Baxter
operator on \(\C^d\otimes\C^d\).  The sparse \(d=6\) points are the
represented products
\(R_2^{\mathrm G}\boxtimes\mathbf 1_3\) and
\(\mathbf 1_2\boxtimes R_3^{\mathrm G}\), while the full-support Gaussian
point matches \(R_2^{\mathrm G}\boxtimes R_3^{\mathrm G}\) up to scalar and
Galois conjugation.

\subsubsection{A conjecture for fixed Pauli directions}

\begin{conj}\label{conj:Pauli-rigidity}
Let \(d\geq2\) with \(4\nmid d\), let
\(\alpha\in(\Z/d\Z)^\times\), and put
\[
  P_\alpha=X\otimes Z^\alpha,
  \qquad
  \tau_\alpha(s,t)=w^{\alpha st}
  \qquad(s,t\in\Z/d\Z).
\]
Suppose that
\[
  a=(1,a_1,\ldots,a_{d-1})\in\mathbb T_d^{\mathrm{coef}}
\]
and that
\[
  R_\alpha(a)=\sum_{t\in\Z/d\Z}a_tP_\alpha^t
\]
satisfies the Yang--Baxter equation and is projectively unitary.  Then
\(a\) belongs to one of the two signed Gaussian torsors for
\(\tau_\alpha\): there exists \(\epsilon\in\{+1,-1\}\) such that
\[
  a_{s+t}=a_sa_t\tau_\alpha(s,t)^\epsilon
  \qquad(s,t\in\Z/d\Z).
\]
Equivalently, up to multiplication by a character and precomposition by an
automorphism \(\psi\in\Aut(\Z/d\Z)\) preserving \(\tau_\alpha\), the
coefficient vector is a signed Family I Gaussian vector.
\end{conj}

\begin{remark}\label{rem:Pauli-rigidity-equivalences}
A cyclotomic Galois automorphism sending \(w\) to
\(w^{\alpha^{-1}}\) maps the entire cyclotomic datum for the Pauli direction
\(X\otimes Z^\alpha\) to the normalized direction \(X\otimes Z\).  This is an
arithmetic normalization of the preceding conjecture, not a unitary
equivalence in a fixed complex Hilbert space.  The conditional
Doikou--Smoktunowicz \(T\)-move likewise belongs at the level of represented
Yang--Baxter operators rather than coefficient vectors: it preserves the
braid representation up to conjugacy but need not preserve the chosen
one-generator Pauli direction.  Thus allowing \(T\)-moves gives a broader
operator-level comparison, separate from the fixed-direction coefficient
rigidity asserted in Conjecture~\ref{conj:Pauli-rigidity}.
\end{remark}

\section{Spectral quotients and strict localizations}
\label{sec:spectral-quotients}

The spectral formulas and Fourier models of Section~\ref{sec:trace-invariants}
give polynomial relations for the local braid generators.  We use these
relations in two complementary ways.  First, spectral projections detect
factorizations through the Temperley--Lieb and
Birman--Murakami--Wenzl (BMW) quotients.  Second, a faithful representation
of the ambient tower shows that a tensor-power realization introduces no
additional quotient.

The first two criteria concern quotient recognition from operators with two
or three eigenvalues.  The final criterion concerns faithfulness of an
already identified braid-image tower.

\subsection{Spectral quotient criteria}

\subsubsection{Two eigenvalues and Temperley--Lieb projections}

Let \(R\) be a unitary Yang--Baxter operator with two distinct eigenvalues.
After multiplying \(R\) by a unit scalar, assume that
\[
  (R+I)(R-qI)=0,
  \qquad q\neq-1.
\]
For \(n\geq2\), let \(R_i\), \(1\leq i<n\), be its local copies and set
\[
  p_i=\frac{qI-R_i}{q+1},
\]
the spectral projection for the eigenvalue \(-1\).  The braid and quadratic
relations define a representation of the Hecke algebra with generators
satisfying \((g_i+1)(g_i-q)=0\).

\begin{prop}\label{prop:two-eigenvalue-Hecke-criterion}
For every \(n\geq2\), let
\[
  \rho_n:H_n(q)\longrightarrow\End(V^{\ot n}),
  \qquad
  \rho_n(g_i)=R_i,
\]
be the Hecke representation determined by the local copies of \(R\).  Then
the following are equivalent.
\begin{enumerate}[label=(\roman*)]
\item For every \(n\geq2\), the representation \(\rho_n\) factors through
the Temperley--Lieb quotient \(\mathrm{TL}_n(q)\).
\item On \(V^{\ot3}\),
\[
  p_1p_2p_1=\frac{q}{(1+q)^2}p_1,
  \qquad
  p_2p_1p_2=\frac{q}{(1+q)^2}p_2.
\]
\item For every \(n\geq3\) and every
\(i,j\in\{1,\ldots,n-1\}\) with \(|i-j|=1\),
\begin{equation}\label{eq:TL-spectral-projection-criterion}
  p_ip_jp_i=\frac{q}{(1+q)^2}p_i.
\end{equation}
\end{enumerate}
\end{prop}

\begin{proof}
Under the change of generators
\[
  e_i=\frac{q-g_i}{q+1},
\]
the Temperley--Lieb quotient of \(H_n(q)\) is obtained by imposing
\[
  e_ie_je_i=\frac{q}{(1+q)^2}e_i
  \qquad(|i-j|=1);
\]
see \cite{JonesHecke}\cite[Lemma~3.5]{LechnerTwoEigenvalues}.  Since
\(\rho_n(e_i)=p_i\), the representation factors through this quotient if
and only if the identities in~\eqref{eq:TL-spectral-projection-criterion}
hold.  By tensor locality, all neighboring identities are translates of
the two identities on \(V^{\ot3}\).  This proves the equivalence of the
three conditions.
\end{proof}

For a unitary two-eigenvalue operator in this normalization, put
\[
  p=\frac{qI-R}{q+1}
\]
and define its spectral density by
\begin{equation}\label{eq:two-eigenvalue-density}
  \eta(R)=\frac{\rank(p)}{(\dim V)^2}.
\end{equation}
This is the parameter used in the two-eigenvalue classification of
\cite[Theorem~3.4]{LechnerTwoEigenvalues}.  It distinguishes Hecke representations with
the same two eigenvalues but different trace quotients.

\subsubsection{Three eigenvalues and a rank-one BMW criterion}

Let \(q,r\in\C^\times\), assume \(q^2\neq1\), and let
\(G\in\End(V\otimes V)\) satisfy the Yang--Baxter equation and
\[
  (G-r^{-1}I)(G-qI)(G+q^{-1}I)=0,
\]
where \(r^{-1},q,-q^{-1}\) are pairwise distinct.  Put
\[
  x=\frac{r-r^{-1}}{q-q^{-1}}+1,
  \qquad
  E=I-\frac{G-G^{-1}}{q-q^{-1}}.
\]
Let \(E_i\) denote the local copies of \(E\).
Suppose \(x\neq0\) and
\[
  E=x\,|\psi\rangle\langle\psi|
\]
for a unit vector \(\psi\in V\otimes V\).

\begin{prop}\label{prop:rank-one-BMW-criterion}
Under the preceding hypotheses, suppose that
\begin{equation}\label{eq:rank-one-BMW-contraction}
  (\langle\psi|_{12}\otimes I)
  G_{23}^{\,\pm1}
  (|\psi\rangle_{12}\otimes I)
  =\frac{r^{\pm1}}{x}I,
\end{equation}
Then, for every \(n\geq2\), the assignments
\[
  g_i\longmapsto G_i,
  \qquad
  e_i\longmapsto E_i
  \qquad(1\leq i<n)
\]
extend to a representation
\[
  C_n(r,q)\longrightarrow\End(V^{\ot n})
\]
of the BMW algebra in the convention
\[
  g_i-g_i^{-1}=(q-q^{-1})(1-e_i),
  \qquad
  e_ig_i=r^{-1}e_i,
  \qquad
  e_i^2=xe_i.
\]
\end{prop}

\begin{proof}
The three roots of the polynomial of \(G\) are nonzero and pairwise
distinct, so \(G\) is invertible and diagonalizable.  The polynomial
defining \(E\) vanishes on the \(q\)- and
\(-q^{-1}\)-eigenspaces and has value \(x\) on the
\(r^{-1}\)-eigenspace.  Hence
\[
  E^2=xE,
  \qquad
  EG=GE=r^{-1}E.
\]
The definition of \(E\) also gives the skein relation
\[
  G-G^{-1}=(q-q^{-1})(I-E).
\]
Equation~\eqref{eq:rank-one-BMW-contraction}, after multiplying by the two
factors of \(x\), gives
\(E_1G_2^{\pm1}E_1=r^{\pm1}E_1\).  On three strands, put
\[
  \Delta=G_1G_2G_1=G_2G_1G_2.
\]
Then \(\Delta G_1\Delta^{-1}=G_2\) and
\(\Delta G_2\Delta^{-1}=G_1\).  Since \(E_i\) is a polynomial in \(G_i\),
conjugating the first relation gives
\(E_2G_1^{\pm1}E_2=r^{\pm1}E_2\).  Translation along the tensor factors
gives both neighboring untwisting relations.  If \(|i-j|=1\), these
relations and the skein relation imply
\begin{align*}
  E_iE_jE_i
  &=E_i^2-
    \frac{E_iG_jE_i-E_iG_j^{-1}E_i}{q-q^{-1}}\\
  &=\left(x-\frac{r-r^{-1}}{q-q^{-1}}\right)E_i
   =E_i.
\end{align*}
The Yang--Baxter equation gives the braid relations, while tensor locality
gives commutativity at distance.  Thus the images satisfy the skein,
idempotent, eigenvalue, and neighboring untwisting relations in the standard
presentation of \(C_n(r,q)\); see \cite{BirmanWenzl,LarsenRowell}.  Hence the
displayed assignments extend to the claimed representation.
\end{proof}

\subsection{Faithful tensor realizations}

For the purposes of this paper, a Yang--Baxter operator \(R\) is a
\emph{strict localization} of a compatible tower of braid-image algebras
\((\mathcal D_n)\) if its local copies on tensor powers of one fixed space
induce a faithful representation of \(\mathcal D_n\) for every \(n\); compare
\cite{RW}.

Let \(W\) be a finite-dimensional vector space.  For every \(n\geq2\), let
\(\mathcal D_n\) be a braid-image algebra generated by
\(g_1,\ldots,g_{n-1}\), put
\[
  \mathcal B_n=\langle r_1,\ldots,r_{n-1}\rangle\subseteq\A_n,
\]
and suppose that
\[
  \theta_n:\mathcal D_n\longrightarrow\mathcal B_n,
  \qquad
  \theta_n(g_i)=r_i,
\]
is an algebra isomorphism.  Let
\[
  \Phi_n:\A_n\longrightarrow\End(W^{\ot n})
\]
be compatible representations.  Put \(R=\Phi_2(r_1)\), and assume that
\[
  \Phi_n(r_i)
  =I_W^{\ot(i-1)}\otimes R\otimes I_W^{\ot(n-i-1)}.
\]
The braid relations among the \(r_i\) imply that \(R\) is a Yang--Baxter
operator.

\begin{prop}\label{prop:faithful-tower-localization-criterion}
The operator \(R\) is a strict localization of the tower
\((\mathcal D_n)\) if and only if
\(\Phi_n|_{\mathcal B_n}\) is injective for every \(n\geq2\).  In
particular, this holds whenever every \(\Phi_n\) is faithful on \(\A_n\).
\end{prop}

\begin{proof}
The localized representation is
\[
  \rho_n=\Phi_n|_{\mathcal B_n}\circ\theta_n.
\]
Since \(\theta_n\) is an isomorphism, \(\rho_n\) is faithful if and only if
\(\Phi_n|_{\mathcal B_n}\) is injective.  Faithfulness of \(\Phi_n\) on
\(\A_n\) implies faithfulness of its restriction to \(\mathcal B_n\).
\end{proof}

The criteria above will be used in Section~\ref{sec:qutrit} to identify the
Temperley--Lieb quotient associated with \(R_1\) and the rank-one BMW quotient
associated with \(R_3\).  In Section~\ref{sec:quaternionic-localization}, the
faithful-tower criterion gives the strict localization of the quaternionic
Family III specialization.

\section{Two-qutrit Clifford outputs and link invariants}
\label{sec:qutrit}

We report an enumeration of the two-qutrit projective Clifford solutions and
a reduction by the declared computational identifications.  The calculation
singles out two explicit projective representatives, denoted \(M_1\) and
\(M_3\) below.  Their unitary normalizations, denoted \(R_1\) and \(R_3\),
are identified with signed Family I and Family II, respectively.  The
specializations of their invariants to HOMFLYPT and to a BMW/Kauffman
invariant that is a squared Jones specialization are proved directly from
their displayed formulas.

\subsection{The two representatives}

Let \(V=\C^3\), let \(\omega=e^{2\pi\ii/3}\), and let
\(\{e_x:x\in\F_3\}\) be the standard basis of \(V\).  Write
\(e_{xy}=e_x\otimes e_y\) and use the ordered basis
\[
  e_{00},e_{01},e_{02},e_{10},e_{11},e_{12},e_{20},e_{21},e_{22}
\]
of \(V\otimes V\).

\begin{computation}
\label{comp:qutrit-reduction}
The projective two-qutrit Clifford group is
\[
  \F_3^4\rtimes \Sp(4,3),
  \qquad |\Sp(4,3)|=51{,}840.
\]
Our exact enumeration returned \(409\) symplectic solutions, \(4465\)
projective lifts, and \(1750\) strict Yang--Baxter operators.  Among the
strict operators, \(559\) are monomial.

Local one-qutrit Clifford conjugation divides the strict operators into
\(139\) orbits, of which \(34\) contain no monomial representative.  The
graph on these \(34\) orbits generated by the admissible conditional
Doikou--Smoktunowicz \(T\)-moves has \(7\) connected components.  The
subsequent Galois and local \(\GL_3\) tests leave \(5\) candidate outputs.
The declared group-type tests retain two outputs that are neither identified
as monomial nor identified as group type.

We choose projective representatives \(M_1\) and \(M_3\) for these two
outputs.  They are outputs of the stated reduction pipeline, not a
classification under local-unitary equivalence.
\end{computation}

We define the two representatives by formulas that determine their
\(9\times9\) matrices.  Let
\[
  A=\F_3,
  \qquad
  q_A(a)=\omega^{a^2},
  \qquad
  \pair{a}{b}=\omega^{2ab}.
\]
Let \(Xe_j=e_{j+1}\), \(Ze_j=\omega^je_j\) for \(j\in\F_3\), and put
\(U=Z\otimes X\).  The Fourier matrix
\[
  F_0e_x=3^{-1/2}\sum_{y\in\F_3}\omega^{xy}e_y
\]
satisfies \(F_0XF_0^{-1}=Z\) and \(F_0Z^2F_0^{-1}=X\).  Thus \(U\) is
locally Clifford-conjugate to the canonical Family I generator
\(X\otimes Z^2\).
For \(n\geq3\) and \(1\leq i<n-1\), the local copies of \(U\) satisfy
\[
  U_iU_{i+1}=\omega^2U_{i+1}U_i.
\]
The phase \(Q_-(a)=q_A(a)^{-1}=\omega^{2a^2}\) has polarization
\(\pair{a}{b}^{-1}=\omega^{ab}\).  Hence the conjugate
\[
  \widetilde R^{\mathrm{I},-}_{Q_-}
  \coloneqq
  (F_0\otimes F_0)R^{\mathrm{I},-}_{Q_-}(F_0\otimes F_0)^{-1}
  =\frac1{\sqrt3}\bigl(I+\omega^2(U+U^2)\bigr)
\]
is locally equivalent to the Family I operator with \(\epsilon=-1\) from
Proposition~\ref{prop:gaussian-coefficients}.  Define the first
representative by
\begin{equation}\label{eq:M1-signed-family-I-gauge}
  M_1\coloneqq
  \sqrt3\,(Z\otimes Z)\widetilde R^{\mathrm{I},-}_{Q_-}
  (Z\otimes Z)^{-1}.
\end{equation}
Set \(R_1=M_1/\sqrt3\).  Equation~\eqref{eq:M1-signed-family-I-gauge}
then shows that \(R_1\) is locally unitarily equivalent to the signed
Family I operator.

For Family II, take \(Q_3,P_3:\F_3\to\Uone\) given by
\[
  Q_3(t)=\omega^{t^2+t},
  \qquad
  P_3(s)=\omega^{s^2-s}.
\]
Then
\begin{equation}\label{eq:M3-family-II-Gaussian}
  R^{\mathrm{II}}_{Q_3,P_3}(e_x\otimes e_y)
  =
  3^{-1/2}\sum_{t\in\F_3}
  \omega^{t^2+t-(x+y)^2+(x+y)}
  e_{x+t}\otimes e_{y-t}.
\end{equation}
Define the second representative by
\[
  M_3\coloneqq\sqrt3\,R^{\mathrm{II}}_{Q_3,P_3}.
\]
Set \(R_3=M_3/\sqrt3\).  Thus \(R_3\) is precisely the Family II qutrit
operator associated with \(Q_3\) and \(P_3\); by
\eqref{eq:family-II-Gaussian-R}, it preserves the total label \(x+y\).

The matrices \(M_1\) and \(M_3\) are the projective representatives retained
by the computational normalization and are not unitary.  Their unitary normalizations are
\(R_1=M_1/\sqrt3\) and \(R_3=M_3/\sqrt3\).  In particular, multiplication by
\(\sqrt3\) is a projective rescaling, not an equivalence of the form
\eqref{eq:unitary-local-equivalence}.  All statements below concerning
unitary local equivalence refer to \(R_1\) and \(R_3\).  Put
\[
  q_{\mathrm B}=e^{\pi\ii/6}.
\]

Equations~\eqref{eq:M1-signed-family-I-gauge} and
\eqref{eq:M3-family-II-Gaussian}, together with
Corollary~\ref{cor:extended-metric-families-YBE}, show that
\(R_1\) and \(R_3\) are unitary Yang--Baxter operators.  Define the spectral
multisets
\[
  \Sigma_1=\{-\ii\ (3),\ q_{\mathrm B}\ (6)\},
  \qquad
  \Sigma_3=
  \{q_{\mathrm B}^{-1}\ (4),\ \ii\ (4),\ -q_{\mathrm B}\ (1)\},
\]
where the numbers in parentheses denote multiplicities, and define
\begin{align}
  m_1(t)
  &=t^2+\frac{\omega-1}{\sqrt3}t-\omega,
  \label{eq:R1-minpoly}\\
  m_3(t)
  &=t^3+\ii.
  \label{eq:R3-minpoly}
\end{align}

\begin{prop}\label{prop:qutrit-local-data}
For \(j\in\{1,3\}\), one has
\[
  \operatorname{Spec}(R_j)=\Sigma_j,
  \qquad
  m_{R_j}=m_j.
\]
\end{prop}

\begin{proof}
In the ungauged Family I model, put
\[
  h(z)=\frac{1+\omega^2(z+z^2)}{\sqrt3}.
\]
The operator \(U=Z\otimes X\) has each of \(1,\omega,\omega^2\) as an
eigenvalue of multiplicity \(3\), and
\[
  h(1)=-\ii,
  \qquad
  h(\omega)=h(\omega^2)=q_{\mathrm B}.
\]
Local conjugation in \eqref{eq:M1-signed-family-I-gauge} does not change the
spectrum.  Hence \(R_1\) has eigenvalues \(-\ii\) and \(q_{\mathrm B}\),
with multiplicities \(3\) and \(6\), respectively.  Since
\[
  (t+\ii)(t-q_{\mathrm B})
  =t^2+\frac{\omega-1}{\sqrt3}t-\omega,
\]
the first spectral multiset is \(\Sigma_1\), and the two distinct
eigenvalues show that the minimal polynomial is \(m_1\).

For \(R_3\), use the coordinates \(s=x+y\), \(d=x\), in which
\[
  R_3\cong M_{P_3^{-1}}\otimes G_{Q_3}.
\]
The multiplier \(M_{P_3^{-1}}\) has eigenvalues \(1,1,\omega\).  On the
three Fourier eigenspaces of \(X\), the convolution operator \(G_{Q_3}\)
has eigenvalues obtained from
\[
  g(z)=\frac{1+\omega^2z+z^2}{\sqrt3},
  \qquad
  g(1)=g(\omega)=q_{\mathrm B}^{-1},
  \qquad
  g(\omega^2)=\ii.
\]
Taking the pairwise products gives the spectrum
\[
  q_{\mathrm B}^{-1}\ (4),
  \qquad
  \ii\ (4),
  \qquad
  -q_{\mathrm B}\ (1).
\]
All three eigenvalues have cube \(-\ii\), so \(R_3^3+\ii I=0\); since they
are distinct, the second spectral multiset is \(\Sigma_3\) and the minimal
polynomial is \(m_3\).  Equivalently, for the unnormalized representatives,
\[
  M_1^2+(\omega-1)M_1-3\omega I=0,
  \qquad
  M_3^3+3(1+2\omega)I=0,
\]
where \(1+2\omega=\ii\sqrt3\).
\end{proof}

\begin{cor}\label{cor:qutrit-enhancement-data}
The partial traces satisfy
\begin{equation}\label{eq:qutrit-partial-traces}
  \Tr_2(R_1^{\pm1})=\sqrt3 I,
  \qquad
  \Tr_1(R_3^{\pm1})=\Tr_2(R_3^{\pm1})=q_{\mathrm B}^{\pm1}I.
\end{equation}
Consequently, the scalar enhancement parameters for the ordinary matrix
trace are
\[
  (\alpha_1,\beta_1)=(1,\sqrt3),
  \qquad
  (\alpha_3,\beta_3)=(q_{\mathrm B},1).
\]
\end{cor}

\begin{proof}
For the ungauged Family I operator
\[
  \widetilde R^{\mathrm{I},-}_{Q_-}
  =\frac1{\sqrt3}\bigl(I+\omega^2(U+U^2)\bigr),
\]
the identities \(\Tr(X)=\Tr(X^2)=0\) give
\[
  \Tr_2\bigl(\widetilde R^{\mathrm{I},-}_{Q_-}\bigr)=\sqrt3 I.
\]
The adjoint has the same identity coefficient, so its partial trace is also
\(\sqrt3 I\).  The local conjugation in
\eqref{eq:M1-signed-family-I-gauge} preserves these scalar partial traces.

For \(R_3\), formula \eqref{eq:M3-family-II-Gaussian} gives, for every
\(x\in\F_3\),
\begin{align*}
  \Tr_2(R_3)e_x
  &=\sum_{y\in\F_3}(I\otimes\langle e_y|)
    R_3(e_x\otimes e_y)\\
  &=\frac1{\sqrt3}\sum_{y\in\F_3}
    \omega^{-(x+y)^2+(x+y)}e_x\\
  &=\frac{2+\omega}{\sqrt3}e_x
   =q_{\mathrm B}e_x.
\end{align*}
Indeed, the contraction forces the summation variable \(t\) in
\eqref{eq:M3-family-II-Gaussian} to be zero.  The same calculation on the
first tensor factor gives \(\Tr_1(R_3)=q_{\mathrm B}I\).  Taking adjoints
gives the inverse identities.  These are precisely the scalar enhancement
identities for
\((\alpha_1,\beta_1)=(1,\sqrt3)\) and
\((\alpha_3,\beta_3)=(q_{\mathrm B},1)\).
\end{proof}

\subsection{\texorpdfstring{The \(R_1\) invariant}{The R1 invariant}}

Let \(T_1\) be the invariant associated with \(R_1\), the ordinary matrix
trace, and scalar parameters \((\alpha_1,\beta_1)=(1,\sqrt3)\).  Thus
\[
  T_1(\widehat\xi)=3^{-n/2}\Tr(\rho_{R_1,n}(\xi)),
  \qquad \xi\in\B_n,\quad n\geq2.
\]
For every oriented link \(L\), set
\(P_1(L)=3^{-1/2}T_1(L)\).  Indeed,
\[
  T_1(\bigcirc)=\beta_1^{-1}\dim(V)=\sqrt3,
\]
so \(P_1(\bigcirc)=1\).

Let \(P_{\mathrm H}(L;a,z)\) denote the normalized HOMFLYPT polynomial
determined by
\[
  aP_{\mathrm H}(L_+;a,z)-a^{-1}P_{\mathrm H}(L_-;a,z)
  =zP_{\mathrm H}(L_0;a,z),
  \qquad P_{\mathrm H}(\bigcirc;a,z)=1.
\]

\begin{theorem}\label{thm:R1-HOMFLYPT}
For every oriented link \(L\),
\[
  P_1(L)=P_{\mathrm H}(L;\omega,\ii).
\]
\end{theorem}

\begin{proof}
The quadratic relation \eqref{eq:R1-minpoly} is equivalent to
\[
  R_1-\omega R_1^{-1}=\frac{1-\omega}{\sqrt3}I.
\]
Indeed, multiply \eqref{eq:R1-minpoly} on the right by \(R_1^{-1}\) and
rearrange.  Applying Proposition~\ref{prop:local-annihilators} to this
Laurent relation identifies the three local replacements \(L_+\), \(L_-\), and
\(L_0\).  Since the enhancement has \(\alpha=1\), no writhe factor appears,
and hence
\[
  T_1(L_+)-\omega T_1(L_-)
  =
  \frac{1-\omega}{\sqrt3}T_1(L_0).
\]
Multiplying by the global factor \(3^{-1/2}\) does not change the skein
coefficients; therefore the same relation holds for \(P_1\).

Dividing the HOMFLYPT skein relation by \(a\) and substituting
\(a=\omega\), \(z=\ii\) gives
\[
  P_{\mathrm H}(L_+;\omega,\ii)
  -\omega P_{\mathrm H}(L_-;\omega,\ii)
  =
  \ii\omega^{-1}P_{\mathrm H}(L_0;\omega,\ii).
\]
Since \(\ii\omega^{-1}=(1-\omega)/\sqrt3\), the two normalized invariants
have the same skein relation and the same unknot value.  The skein relation
together with the unknot normalization determines the invariant, so
\(P_1(L)=P_{\mathrm H}(L;\omega,\ii)\).
\end{proof}

The Hecke representation defined by \(R_1\) factors further through the
Temperley--Lieb quotient.

\begin{cor}\label{cor:R1-Temperley-Lieb}
Put
\[
  h_{\mathrm{TL}}=e^{-\pi\ii/3},
  \qquad
  \widetilde R_1=-\ii R_1.
\]
Then
\[
  (\widetilde R_1+I)(\widetilde R_1-h_{\mathrm{TL}}I)=0,
  \qquad
  \eta(\widetilde R_1)=\frac13,
\]
and the corresponding Hecke representation factors through the
Temperley--Lieb quotient.  Thus \(R_1\) gives the qutrit Gaussian
Temperley--Lieb parameter appearing in the Rowell--Wang localization
\cite{RW}.
\end{cor}

\begin{proof}
Before the gauge transformation
\eqref{eq:M1-signed-family-I-gauge}, the spectral projection for the
eigenvalue \(-1\) of \(-\ii\widetilde R^{\mathrm{I},-}_{Q_-}\) is
\[
  p=\frac{I+U+U^2}{3}.
\]
It has rank \(3\) on \(V\otimes V\), and its density is therefore
\(3/9=1/3\).
For neighboring generators, the relation
\(U_iU_{i\pm1}=\omega^{\pm1}U_{i\pm1}U_i\) and character orthogonality give
\[
  p_ip_{i\pm1}p_i=\frac13p_i.
\]
Since
\[
  \frac{h_{\mathrm{TL}}}{(1+h_{\mathrm{TL}})^2}=\frac13,
\]
Proposition~\ref{prop:two-eigenvalue-Hecke-criterion} gives the
Temperley--Lieb factorization.  Local gauge preserves these relations.
\end{proof}

\subsection{\texorpdfstring{The \(R_3\) invariant}{The R3 invariant}}

For Family II, the enhancement scalar is determined by the Gauss sum of
\(P_3\).  For the phases in \eqref{eq:M3-family-II-Gaussian},
\[
  \gamma(Q_3)=\gamma(P_3)=3^{-1/2}(2+\omega^2)=e^{-\pi\ii/6}.
\]
The Family II enhancement of Theorem~\ref{thm:metric-family-enhancement}
therefore gives
\[
  \alpha_3=\overline{\gamma(P_3)}=e^{\pi\ii/6}.
\]
The general spectrum formula also gives the three eigenvalues
\[
  \gamma(Q_3),\qquad \gamma(Q_3)\omega,\qquad \gamma(Q_3)\omega^2,
\]
and therefore \(R_3^3=-\ii I\), equivalently \(R_3^3+\ii I=0\).

Let \(T_3\) be the invariant associated with \(R_3\), the ordinary matrix
trace, and scalar parameters
\((\alpha_3,\beta_3)=(q_{\mathrm B},1)\).  Then
\[
  T_3(\widehat\xi)=q_{\mathrm B}^{-\operatorname{wr}(\xi)}
  \Tr(\rho_{R_3,n}(\xi)),
  \qquad \xi\in\B_n,\quad n\geq2,
  \qquad
  T_3(\bigcirc)=3.
\]
The cubic relation \(R_3^3+\ii I=0\), together with
\(q_{\mathrm B}^3=\ii\), gives the local
periodic relation
\[
  T_3(L_{m+3})+T_3(L_m)=0
\]
for every \(m\in\Z\) and every local two-strand twist family.

Set
\[
  r=q_{\mathrm B}^3=\ii,
  \qquad G=q_{\mathrm B}^2R_3,
  \qquad
  E=I-\frac{G-G^{-1}}{q_{\mathrm B}-q_{\mathrm B}^{-1}}.
\]
Scalar renormalization does not change Turaev's invariant, but it changes
the enhancement parameter from \(q_{\mathrm B}\) to
\(q_{\mathrm B}^3=r\).

For \(n\geq2\), let \(G_i\) and \(E_i\) denote the local copies of \(G\)
and \(E\) on \(V^{\ot n}\).  By definition,
\[
  (q_{\mathrm B}-q_{\mathrm B}^{-1})(I-E_i)
  =G_i-G_i^{-1}.
\]
The corresponding BMW loop parameter is
\[
  x=\frac{r-r^{-1}}{q_{\mathrm B}-q_{\mathrm B}^{-1}}+1=3.
\]

\begin{prop}\label{prop:R3-BMW}
For every \(n\geq2\), the assignments
\[
  g_i\longmapsto G_i,
  \qquad
  e_i\longmapsto E_i
  \qquad(1\leq i<n)
\]
extend to a representation
\[
  C_n(r,q_{\mathrm B})\longrightarrow\End(V^{\ot n}),
  \qquad r=q_{\mathrm B}^3=\ii.
\]
In particular, \(E_i^2=3E_i\).  Moreover, the pullbacks of
\(3^{-n}\Tr\) along these representations form the normalized BMW Markov
trace.
\end{prop}

\begin{proof}
The braid relations for the generators
\(G_i=q_{\mathrm B}^2R_{3,i}\) follow from the Yang--Baxter equation for
\(R_3\), since scalar multiplication does not change the braid relation.
The roots of \(t^3+\ii\), the minimal polynomial of \(R_3\), are
\[
  e^{-\pi\ii/6},
  \qquad
  \ii,
  \qquad
  -e^{\pi\ii/6}.
\]
Multiplication by \(q_{\mathrm B}^2=e^{\pi\ii/3}\) gives the eigenvalues
\[
  q_{\mathrm B},\qquad -q_{\mathrm B}^{-1},\qquad r^{-1}.
\]
Thus the BMW cubic relation
\[
  (G_i-r^{-1})(G_i-q_{\mathrm B})(G_i+q_{\mathrm B}^{-1})=0
\]
holds.

Proposition~\ref{prop:qutrit-local-data} shows that the eigenvalue
\(-q_{\mathrm B}\) of \(R_3\) has multiplicity one.  Define
\begin{equation}\label{eq:M3-BMW-vector}
  \psi
  =\frac1{\sqrt3}
  \bigl(\omega^2e_{02}+e_{11}+\omega e_{20}\bigr).
\end{equation}
Direct substitution in \eqref{eq:M3-family-II-Gaussian} gives
\[
  R_3\psi=-q_{\mathrm B}\psi,
  \qquad
  G\psi=-q_{\mathrm B}^3\psi=r^{-1}\psi.
\]
Since the eigenvalue \(-q_{\mathrm B}\) has multiplicity one, the
\(r^{-1}\)-eigenspace of \(G=q_{\mathrm B}^2R_3\) is exactly \(\C\psi\).
Spectral interpolation therefore gives
\begin{equation}\label{eq:M3-BMW-projector}
  E=3|\psi\rangle\langle\psi|.
\end{equation}
Its local copies are the operators \(E_i\) defined above.
The coefficient matrix of \(\psi\) is
\[
  C_\psi
  =\frac1{\sqrt3}
  \begin{pmatrix}
    0&0&\omega^2\\
    0&1&0\\
    \omega&0&0
  \end{pmatrix},
  \qquad
  C_\psi^*C_\psi=C_\psi C_\psi^*=\frac13I.
\]
Thus \(\sqrt3C_\psi\) is unitary, so \(\psi\) is maximally entangled and
\[
  \Tr_1(|\psi\rangle\langle\psi|)
  =\Tr_2(|\psi\rangle\langle\psi|)=\frac13I.
\]

Equation~\eqref{eq:qutrit-partial-traces} implies
\[
  \Tr_1(G^{\pm1})=r^{\pm1}I.
\]
Write \(\psi=\sum_{i,j}c_{ij}e_i\otimes e_j\) and \(C=(c_{ij})\).  Since
\(C^*C=I/3\), the \((k,l)\)-matrix coefficient of the contraction by any
\(T\in\End(V\otimes V)\) is
\[
  \sum_{i,j,j'}\overline{c_{ij}}c_{ij'}T_{jk,j'l}
  =\frac13\sum_jT_{jk,jl}.
\]
Thus contraction by \(\psi\) equals \(\Tr_1(T)/3\).  Applying this to
\(T=G^{\pm1}\) gives
\begin{equation}\label{eq:M3-BMW-contraction}
  (\langle\psi|_{12}\otimes I)
  G_{23}^{\,\pm1}
  (|\psi\rangle_{12}\otimes I)
  =\frac{r^{\pm1}}3I.
\end{equation}
Proposition~\ref{prop:rank-one-BMW-criterion}, with loop parameter
\[
  \frac{r-r^{-1}}{q_{\mathrm B}-q_{\mathrm B}^{-1}}+1=3,
\]
therefore gives all BMW relations, including
\[
  E_i^2=3E_i,
  \qquad
  E_iG_{i\pm1}^{\pm1}E_i=r^{\pm1}E_i.
\]

Finally, for \(n\geq1\), put
\(\operatorname{tr}^{\mathrm{BMW}}_n=3^{-n}\Tr\).
The identities
\(\Tr_2(G^{\pm1})=r^{\pm1}I\) and \(\Tr_2(E)=I\) imply, for every operator
\(x\) on the first \(n\) tensor factors,
\[
  \operatorname{tr}^{\mathrm{BMW}}_{n+1}(xG_n^{\pm1})
  =\frac{r^{\pm1}}3\operatorname{tr}^{\mathrm{BMW}}_n(x),
  \qquad
  \operatorname{tr}^{\mathrm{BMW}}_{n+1}(xE_n)
  =\frac13\operatorname{tr}^{\mathrm{BMW}}_n(x).
\]
Together with cyclicity and
\(\operatorname{tr}^{\mathrm{BMW}}_n(I)=1\), these are the defining BMW
Markov trace identities with loop parameter \(3\), in the convention of
\cite[Proposition~4.3]{LarsenRowell}.
\end{proof}

Let \(C_n^{\mathrm{ss}}(r,q_{\mathrm B})\) denote the quotient of
\(C_n(r,q_{\mathrm B})\) by the radical of its Markov trace.

\begin{cor}\label{cor:R3-BMW-localization}
The representation of Proposition~\ref{prop:R3-BMW} induces a faithful
representation
\[
  C_n^{\mathrm{ss}}(q_{\mathrm B}^3,q_{\mathrm B})
  \longrightarrow \End(V^{\ot n})
  \qquad(n\geq2).
\]
Consequently, \(R_3\) is a strict localization of the semisimple BMW
trace-quotient tower at \(q_{\mathrm B}=e^{\pi\ii/6}\).
\end{cor}

\begin{proof}
Let \(\pi_n:C_n(r,q_{\mathrm B})\to\End(V^{\ot n})\) be the representation
of Proposition~\ref{prop:R3-BMW}, and let
\(\operatorname{tr}^{\mathrm{BMW}}_n\) be the BMW Markov trace.  That
proposition gives
\[
  \operatorname{tr}^{\mathrm{BMW}}_n(x)
  =3^{-n}\Tr(\pi_n(x)).
\]
If \(x\in\ker\pi_n\), then
\(\operatorname{tr}^{\mathrm{BMW}}_n(xy)=0\) for every \(y\), so \(x\)
lies in the trace radical.  Conversely, the standard involution satisfies
\(g_i^*=g_i^{-1}\) and \(e_i^*=e_i\), and \(\pi_n\) preserves it because
the \(G_i\) are unitary and the \(E_i\) are self-adjoint.  If \(x\) lies in
the trace radical, then
\[
  0=\operatorname{tr}^{\mathrm{BMW}}_n(xx^*)
   =3^{-n}\Tr\bigl(\pi_n(x)\pi_n(x)^*\bigr),
\]
which implies \(\pi_n(x)=0\).  Thus the kernel equals the trace radical,
and the induced representation of the semisimple quotient is faithful.
Since \(G_i=q_{\mathrm B}^2R_{3,i}\), the local copies of \(R_3\) generate
the same unital image algebras, proving strict localization.
\end{proof}

Let
\[
  \widetilde T_3(L)=\frac{1}{3}T_3(L).
\]
Then \(\widetilde T_3(\bigcirc)=1\), and the loop value is \(3\), matching the
BMW trace in Proposition~\ref{prop:R3-BMW}.

For \(L=\widehat\xi\), where \(\xi\in\B_n\) has writhe \(w\), let
\(G_\xi\) and \(g_\xi\) denote its BMW and Temperley--Lieb images,
respectively.  Denote by \(K_L(r,q_{\mathrm B})\) the ambient-isotopy
invariant obtained from the BMW Markov trace in the normalization of
Proposition~\ref{prop:R3-BMW}; thus
\[
  K_L(r,q_{\mathrm B})
  =3^{n-1}r^{-w}\operatorname{tr}^{\mathrm{BMW}}_n(G_\xi),
  \qquad r=q_{\mathrm B}^3.
\]

We use the Jones normalization associated with the normalized
Temperley--Lieb Markov trace:
\[
  V_L(q_{\mathrm B}^{-2})
  =(q_{\mathrm B}+q_{\mathrm B}^{-1})^{n-1}
    (-q_{\mathrm B}^{-1})^w
    \operatorname{tr}^{\mathrm{TL}}_n(g_\xi).
\]
Then \(V_{\bigcirc}=1\), and
\(q_{\mathrm B}^{-2}=e^{-\pi\ii/3}\).

\begin{theorem}\label{thm:R3-symmetric-square}
For every oriented link \(L\),
\[
  \widetilde T_3(L)
  =K_L(r,q_{\mathrm B})
  =V_L(q_{\mathrm B}^{-2})^2.
\]
\end{theorem}

\begin{proof}
Choose a braid presentation \(L=\widehat\xi\), with \(\xi\in\B_n\) of
writhe \(w\).
Since \(G_\xi=q_{\mathrm B}^{2w}\rho_{R_3,n}(\xi)\), the definition of
\(T_3\) and Proposition~\ref{prop:R3-BMW} give
\[
  \widetilde T_3(L)
  =\frac13q_{\mathrm B}^{-3w}\Tr(G_\xi)
  =3^{n-1}r^{-w}\operatorname{tr}^{\mathrm{BMW}}_n(G_\xi)
  =K_L(r,q_{\mathrm B}).
\]

At \(r=q_{\mathrm B}^3\), the Larsen--Rowell symmetric-square isomorphism
for the semisimple trace quotient
\cite[Lemma~5.4 and Theorem~6.2]{LarsenRowell} sends
\[
  G_i\longmapsto q_{\mathrm B}(g_i\otimes g_i)
\]
and sends the BMW trace to the tensor square of the Temperley--Lieb Markov
trace.  Therefore
\[
  K_L(r,q_{\mathrm B})
  =3^{n-1}q_{\mathrm B}^{-2w}
    \operatorname{tr}^{\mathrm{TL}}_n(g_\xi)^2.
\]
Squaring the Jones normalization above and using
\(3=(q_{\mathrm B}+q_{\mathrm B}^{-1})^2\) gives
\(K_L(r,q_{\mathrm B})=V_L(q_{\mathrm B}^{-2})^2\).
\end{proof}

The qutrit examples realize the Temperley--Lieb and rank-one BMW cases; the
Family II example strictly localizes the semisimple BMW trace quotient.  The
smallest Family III specialization gives a non-Temperley--Lieb Hecke
representation with strict localization.

\section{The quaternionic Family III specialization}
\label{sec:quaternionic-localization}

Take
\[
  A=\Z/2\Z,
  \qquad q_0:A\longrightarrow\Uone,
  \quad q_0(0)=1,
  \quad q_0(1)=-\ii,
  \qquad \pair{a}{b}=(-1)^{ab}\quad(a,b\in A).
\]
This is the smallest Family III example.

This section recasts the quaternionic strict-localization construction of
\cite{RowellQuaternionic,GustafsonKimballRowellZhang} in the uniform Family
III language developed above.  We first give the direct quaternionic tower
presentation and then identify it with the corresponding twisted
group-algebra tower.

\subsection{The quaternionic tower}

Write
\[
  Q_8=\langle u,v,z:u^2=v^2=z,\ z^2=1,\
  [z,u]=[z,v]=1,\ uv=zvu\rangle,
  \qquad
  \overline Q_8=Q_8/\langle z\rangle\cong\F_2^2.
\]
The specialization imposes \(z=-1\) by passing to the quotient algebra
\[
  \mathcal L=\C[Q_8]/\langle z+1\rangle.
\]
Choose the section \(s:\overline Q_8\to Q_8\),
\(s(a,b)=u^av^b\) for \(a,b\in\F_2\).  We use \(u,v\) also for their
images in \(\mathcal L\), and
let \(U_{(a,b)}\) be the image of \(s(a,b)\) in
\(\mathcal L\).  Since \(u^2=v^2=-1\) and
\(vu=-uv\), its basis elements satisfy
\begin{equation}\label{eq:quaternion-local-algebra}
  U_{(a,b)}U_{(c,d)}
  =(-1)^{ac+bc+bd}U_{(a+c,b+d)}
  \qquad(a,b,c,d\in\F_2).
\end{equation}
Thus
\[
  \mathcal L\cong\C_\nu[\overline Q_8],
  \qquad
  \nu((a,b),(c,d))=(-1)^{ac+bc+bd}.
\]
The commutator bicharacter is
\[
  \tau((a,b),(c,d))=(-1)^{bc+ad},
  \qquad a,b,c,d\in\F_2,
\]
which is exactly the Family III adjacent bicharacter.  Consequently the
whole quaternionic tower is
\begin{equation}\label{eq:whole-tower-twisted-group-algebra}
  \A_n(Q_8)
  =\T_n(\overline Q_8;\nu,\tau)
  \cong\C_{\Omega_n^{\nu,\tau}}[\overline Q_8^{\,n-1}]
  \qquad(n\geq1).
\end{equation}
Equivalently, \(\A_n(Q_8)\) is the iterated twisted tensor product of
\(n-1\) copies of \(\mathcal L\), with adjacent twisting determined by
\(\tau\).  If \(u_i,v_i\) denote the images of \(u,v\) in the \(i\)-th
factor, then
\begin{align}
 u_i^2=v_i^2&=-1,& u_iv_i&=-v_iu_i,
 \label{eq:quaternion-same-factor}\\
 u_iu_{i+1}&=u_{i+1}u_i,&
 v_iv_{i+1}&=v_{i+1}v_i,
 \label{eq:quaternion-adjacent-same}\\
 u_iv_{i+1}&=-v_{i+1}u_i,&
 v_iu_{i+1}&=-u_{i+1}v_i,
 \label{eq:quaternion-adjacent-cross}
\end{align}
and generators at distance at least two commute.  The normal monomials
\[
  u_1^{\epsilon_1}v_1^{\delta_1}\cdots
  u_{n-1}^{\epsilon_{n-1}}v_{n-1}^{\delta_{n-1}},
  \qquad \epsilon_i,\delta_i\in\{0,1\},
\]
form a basis.  In particular,
\(\dim\A_n(Q_8)=4^{n-1}\).

By \eqref{eq:whole-tower-twisted-group-algebra}, the quotient by \(z+1\) is
the twisted group algebra determined by the section cocycle \(\nu\).  A
different section changes \(\nu\) by a coboundary and
hence gives the gauge-equivalent tower of
Remark~\ref{rem:cocycle-gauge}.

For \(n\geq2\) and \(1\leq i<n\), let \(u_i=U_i^{(1,0)}\) and
\(v_i=U_i^{(0,1)}\).  The Family III braid
element is
\begin{equation}\label{eq:abstract-r}
  r_{0,i}=\frac12(1+u_i+v_i+u_iv_i)
  =\frac12(1+u_i)(1+v_i).
\end{equation}
Proposition~\ref{prop:family-III-coefficients} gives its unitarity and braid
compatibility; this is the quaternionic braid element used in
\cite{RowellQuaternionic}.  For the unnormalized element
\[
  r_i=2r_{0,i}=1+u_i+v_i+u_iv_i
\]
one has
\begin{equation}\label{eq:quaternion-Hecke-relation}
  r_i^2-2r_i+4=0.
\end{equation}
If \(\zeta=e^{\pi\ii/6}\), then
\begin{equation}\label{eq:quaternion-Hecke-normalization}
  g_i=\frac{\ii}{2}r_i
\end{equation}
satisfies \((g_i-\zeta)(g_i+\zeta^{-1})=0\).

\subsection{The local operator and strict localization}

Let \(K=\C^2\), \(W=K\otimes K\), and let \(X,Z\in\End(K)\) be defined by
\(Xe_j=e_{j+1}\) and \(Ze_j=(-1)^je_j\) for \(j\in\Z/2\Z\).  On
\(W\otimes W\), put
\begin{equation}\label{eq:PQ}
  P_Z=(I\otimes Z)\otimes(Z\otimes I),
  \qquad
  P_X=(X\otimes I)\otimes(X\otimes X),
\end{equation}
and
\begin{equation}\label{eq:UV}
  U=\ii P_Z,
  \qquad V=\ii P_X.
\end{equation}
Since \(q_0(1)^{-1}=\ii\), these are the operators \(Y^{(1,0)}\) and
\(Y^{(0,1)}\) in the general Family III realization.  For \(n\geq2\) and
\(1\leq i<n\), let \(U_i\) and \(V_i\) denote their copies acting on the
\(i\)-th and \((i+1)\)-st factors of \(W^{\ot n}\).

\begin{prop}\label{prop:tower-representation}
For every \(n\geq2\), the assignments \(u_i\mapsto U_i\) and
\(v_i\mapsto V_i\), \(1\leq i<n\), extend to faithful homomorphisms
\[
  \pi_n:\A_n(Q_8)\longrightarrow\End(W^{\ot n}).
\]
Together with the scalar representation at \(n=1\), these homomorphisms are
compatible with the tower inclusions.
\end{prop}

\begin{proof}
This is the specialization of the faithful compatible representation
\(\Phi_n^{\mathrm{III}}\) constructed in
Subsection~\ref{subsec:tensor-power-representations}.
\end{proof}

Put \(h=\zeta^2=e^{\pi\ii/3}\) and define
\begin{equation}\label{eq:R-pauli}
  R=\frac{\ii\zeta}{2}\bigl(I_{16}+U+V+UV\bigr).
\end{equation}

\begin{theorem}\label{thm:strict-ybo}
The operator \(R\) is unitary and satisfies the Yang--Baxter equation.  It
also satisfies
\[
  (R+I_{16})(R-hI_{16})=0,
\]
and the two eigenvalues \(-1\) and \(h\) both have multiplicity \(8\).
\end{theorem}

\begin{proof}
The Yang--Baxter equation follows by applying \(\pi_3\) to the Family III
braid relation.  Unitarity follows either from
Theorem~\ref{thm:metric-families-YBE} or directly from
\[
  R=\ii\zeta
  \left(\frac{I+U}{\sqrt2}\right)
  \left(\frac{I+V}{\sqrt2}\right),
\]
since \(U^2=V^2=-I\) and \(U^*=-U\), \(V^*=-V\).
Equation~\eqref{eq:quaternion-Hecke-relation} gives the quadratic relation
in the theorem after the scalar normalization.  Finally, \(U,V,UV\) are
traceless and \(h-1=\ii\zeta\), which gives
\[
  \Tr(R)=8(h-1).
\]
Together with \(\dim(W\otimes W)=16\), this forces both multiplicities to
be \(8\).
\end{proof}

We recall the notation of \cite[Section~3]{LechnerTwoEigenvalues}.
Let \(S\in\mathrm{U}(D\otimes D)\) be a unitary Yang--Baxter operator with
\[
  \operatorname{Spec}(S)=\{-1,q\},
  \qquad q\neq\pm1,
\]
and put
\[
  p_S=\frac{qI-S}{1+q},
  \qquad
  \eta(S)=\frac{\rank(p_S)}{(\dim D)^2}.
\]
Two such operators are equivalent in Lechner's sense if their associated
representations of every braid group are unitarily equivalent.  The notation
\([q,\eta,d]\) denotes the possibly empty equivalence class with Hecke
parameter \(q\), spectral density \(\eta\), and local dimension \(d\).
This differs from the local unitary equivalence in
\eqref{eq:unitary-local-equivalence}: Lechner's relation compares the entire
tower of braid representations, whereas the latter is defined by a unit
scalar and a single local change of basis.

Following \cite{RowellQuaternionic}, let \(H_n(3,6)\) denote the semisimple
Hecke braid-image algebra on \(n\) strands associated with \(SU(3)_3\).

\begin{theorem}\label{thm:quaternionic-Lechner-localization}
The operator \(R\) lies in the non-Temperley--Lieb class
\[
  [e^{\pi\ii/3},1/2,4]
\].
Its local copies give a strict localization of the tower \(H_n(3,6)\).
\end{theorem}

\begin{proof}
The \(-1\)-eigenspace has rank \(8\) in local dimension \(4^2=16\), and
the spectral density is therefore \(1/2\).  For \(h=e^{\pi\ii/3}\), the
Temperley--Lieb value is
\[
  \frac{h}{(1+h)^2}=\frac13;
\]
hence the two-eigenvalue criterion of
\cite[Lemma~3.5]{LechnerTwoEigenvalues} rules
out the Temperley--Lieb quotient and identifies the stated class.

Rowell uses the Hecke parameter \(q=h=e^{2\pi\ii/6}\) and the braid
generators
\[
  s_i=-\frac1{2h}(1+u_i+v_i+u_iv_i)
  \qquad(1\leq i<n);
\]
see \cite[(3.3)]{RowellQuaternionic}.  Since \(h=\zeta^2\) and
\(-h^{-1}=\ii\zeta\), our normalizations satisfy
\begin{equation}\label{eq:Rowell-generator-comparison}
  s_i=\frac{\ii\zeta}{2}r_i=\zeta g_i.
\end{equation}
Thus the image of \(s_i\) under the faithful tower representation
\(\pi_n\) is exactly the local copy of the operator \(R\) in
\eqref{eq:R-pauli}.  Rowell's Lemma~3.4
\cite[Lemma~3.4]{RowellQuaternionic} identifies \(H_n(3,6)\) with the unital
subalgebra of \(\A_n(Q_8)\) generated by the \(s_i\).  Since multiplication
of every generator by the same nonzero scalar does not change the generated
unital algebra, this is also the subalgebra generated by the \(g_i\), or
equivalently by the \(r_i\).  Proposition~\ref{prop:tower-representation}
represents the entire tower faithfully and compatibly.  Restricting to this
braid-generated subalgebra and applying
Proposition~\ref{prop:faithful-tower-localization-criterion} proves
strictness.
\end{proof}

\begin{cor}\label{cor:Lechner-minimal-dimension}
The class \([e^{\pi\ii/3},1/2,4]\) is nonempty, and \(4\) is the smallest
local dimension in Lechner's exceptional family
\([e^{\pi\ii/3},1/2,2m]\) for which the class is nonempty.
\end{cor}

\begin{proof}
Theorem~\ref{thm:quaternionic-Lechner-localization} gives nonemptiness in
local dimension \(4\).  Lechner's classification restricts the exceptional
family to even local dimension, while the dimension-\(2\) class is empty
\cite[Proposition~1.2 and Theorem~3.4]{LechnerTwoEigenvalues}.
\end{proof}

\subsection{The three low-dimensional quotient patterns}

The qutrit and quaternionic specializations are summarized by
\[
\begin{array}{@{}cccl@{}}
\toprule
\textup{source} & \textup{local dimension} &
\textup{spectrum multiplicities} & \textup{quotient}\\
\midrule
\mathrm{Family\ I}^{-},\ A=\F_3
&3&3,6&\textup{Temperley--Lieb/Hecke}\\
\mathrm{Family\ II},\ A=\F_3
&3&4,4,1&\textup{BMW at }r=\ii,\ q_{\mathrm B}=e^{\pi\ii/6}\\
\mathrm{Family\ III},\ A=\Z/2\Z
&4&8,8&\textup{Hecke, non-Temperley--Lieb}\\
\bottomrule
\end{array}
\]
The three rows correspond to a Temperley--Lieb projection, a rank-one BMW
projection, and the exceptional density-\(1/2\) Hecke class.

\section{Two-ququint Clifford orbits}
\label{sec:ququint}

We compare the two-ququint Clifford search with the metric families of
Section~\ref{sec:gaussian}.  Let
\[
  \xi_{60}=e^{2\pi\ii/60},
  \qquad \zeta_5=\xi_{60}^{12}=e^{2\pi\ii/5}.
\]
The projective two-ququint Clifford group is
\[
  \F_5^4\rtimes\Sp(4,5),
  \qquad |\Sp(4,5)|=9{,}360{,}000.
\]
The strict matrix Yang--Baxter tests are exact over \(\Q(\xi_{60})\).  The
group-type tests and the comparisons up to local gauge are Gr\"obner
computations over \(\F_{1201}\), where a primitive \(60\)-th root is fixed.
They do not prove the corresponding statements over \(\C\).

The enumeration proceeds through symplectic solutions, their projective
Clifford lifts, the strict matrix Yang--Baxter solutions among those lifts,
and finally reduction under local one-ququint Clifford conjugation.  In the
spectral table below, \(e\in\Z/60\Z\), and \(e^m\) denotes the eigenvalue
\(\xi_{60}^e\) with multiplicity \(m\).

\begin{computation}\label{comp:ququint-Clifford-enumeration}
The exact enumeration returned \(3001\) symplectic solutions, \(97{,}201\)
projective Clifford lifts, and \(6027\) strict Yang--Baxter matrices over
\(\Q(\xi_{60})\).  Reduction under local one-ququint Clifford conjugation
gave \(438\) orbits, of which \(423\) contain a monomial representative and
\(15\) do not.  The \(15\) non-monomial local Clifford orbits have the
following projective spectral types:
\[
\begin{array}{@{}cclc@{}}
\toprule
\textup{type}&\textup{number of orbits}&\textup{spectrum multiplicities}&
\textup{projective order}\\
\midrule
S_1&2&0^{10},12^{10},36^5&5\\
S_2&6&0^9,12^4,24^4,36^4,48^4&5\\
S_3&4&0^5,10^3,20^5,30^3,40^5,50^4&6\\
S_4&3&0^3,\ (5k)^2\ (1\leq k\leq11)&12\\
\bottomrule
\end{array}
\]
The strict Yang--Baxter tests, orbit counts, and spectra are exact in
characteristic zero.
\end{computation}

\begin{computation}\label{comp:ququint-finite-field-group-type}
After reduction to \(\F_{1201}\), the group-type tests identify all four
\(S_3\) orbits and all three \(S_4\) orbits as group type.  The non-monomial
output not identified as group type therefore consists of the two \(S_1\)
orbits and the six \(S_2\) orbits.  This is a finite-field statement and does
not establish the corresponding group-type identifications over \(\C\).
\end{computation}

\subsection{Gaussian and product representatives}

On the basis \(\{e_j:j\in\F_5\}\) of \(\C^5\), define
\(Xe_j=e_{j+1}\) and \(Ze_j=\zeta_5^je_j\).  For \(a,b\in\F_5\) and
\(D\in\End(\C^5\otimes\C^5)\) satisfying \(D^5=I\), set
\[
  G_{a,b}(D)=\sum_{j\in\F_5}\zeta_5^{aj^2+bj}D^j.
\]
For \(a,b,c,b_2\in\F_5\), the cyclic and product ansatzes are
\begin{align*}
  R^{\mathrm{Gauss}}_{a,b}
  &=5^{-1/2}G_{a,b}(X\otimes Z),\\
  R^{\mathrm{prod}}_{a,b,c,b_2}
  &=5^{-1}G_{a,b}(X\otimes X^{-1})G_{c,b_2}(Z\otimes Z).
\end{align*}
Define
\[
  S_a(z)=\sum_{x\in\F_5}\zeta_5^{2ax^2-zx}
  \qquad(a,z\in\F_5).
\]

\begin{prop}\label{prop:ququint-cyclic-gaussian}
For \(a,b\in\F_5\), the cyclic ansatz is a unitary Yang--Baxter operator
exactly when
\[
  a=2\ \text{or}\ 3,
  \qquad b\in\F_5.
\]
These are the two signed Family I branches.
\end{prop}

\begin{proof}
Put \(D=X\otimes Z\), and let \(D_i\) denote its local copies.  Since
\(ZX=\zeta_5XZ\),
\[
  D_i^xD_{i+1}^y
  =\zeta_5^{xy}D_{i+1}^yD_i^x
  \qquad(x,y\in\F_5).
\]
Thus the relevant cyclic tower has
\(\tau(x,y)=\zeta_5^{xy}\).  The normalization \(5^{-1/2}\) is a common
nonzero scalar, so the coefficient equation for
\(f(x)=\zeta_5^{ax^2+bx}\) is
\[
  f(t)\sum_{x\in\F_5}f(s-x)f(x)\zeta_5^{-xt}
  =f(s)\sum_{x\in\F_5}f(x)f(t-x)\zeta_5^{-sx}.
\]
The summands on the two sides are, respectively,
\begin{align*}
  f(t)f(s-x)f(x)\zeta_5^{-xt}
  &=\zeta_5^{a(s^2+t^2)+b(s+t)}
    \zeta_5^{2ax^2-(2as+t)x},\\
  f(s)f(x)f(t-x)\zeta_5^{-sx}
  &=\zeta_5^{a(s^2+t^2)+b(s+t)}
    \zeta_5^{2ax^2-(2at+s)x}.
\end{align*}
After canceling the common factor, the coefficient equation is therefore
\[
  S_a(2as+t)=S_a(2at+s)
  \qquad(s,t\in\F_5).
\]
The coefficient sum determining unitarity at the shift \(u\) is
\[
  5^{-1}\sum_{x\in\F_5}\overline{f(x)}f(x+u)
  =5^{-1}\zeta_5^{au^2+bu}\sum_{x\in\F_5}\zeta_5^{2aux}
  \qquad(u\in\F_5).
\]
Thus unitarity forces \(a\neq0\), and for such \(a\) this sum is
\(\delta_{u,0}\).  Completing the square gives
\[
  S_a(z)=S_a(0)\zeta_5^{-z^2/(8a)},
  \qquad S_a(0)\neq0.
\]
The coefficient equation is therefore equivalent to
\[
  (4a^2-1)(s^2-t^2)=0
  \qquad(s,t\in\F_5).
\]
Hence \(4a^2=1\), which gives \(2a=\pm1\) and therefore \(a=2\) or \(3\).
Conversely, these two equations say precisely that the polarization is the
positive or negative polarization of the fixed Pauli realization.  Therefore
Proposition~\ref{prop:gaussian-coefficients} gives the Yang--Baxter
equation, while the preceding coefficient calculation gives unitarity.
\end{proof}

\begin{prop}\label{prop:ququint-product}
For \(a,b,c,b_2\in\F_5\), the product ansatz is a unitary Yang--Baxter
operator exactly when
\[
  ac=-1\pmod5.
\]
In that case it is a unit scalar multiple of a Family \(\mathrm{II}_E\)
operator; \(b\) and \(b_2\) are character parameters.
\end{prop}

\begin{proof}
The coefficient sums determining unitarity separate into two one-variable
quadratic sums; hence unitarity is equivalent to \(a,c\neq0\).  Assume this
and use the sums \(S_a\) defined before
Proposition~\ref{prop:ququint-cyclic-gaussian}.

Set
\[
  D=X\otimes X^{-1},
  \qquad H=Z\otimes Z.
\]
The operators \(D\) and \(H\) commute.  If
\(U_i^{(x,y)}=D_i^xH_i^y\), then the middle tensor factor gives
\[
  U_i^{(x,y)}U_{i+1}^{(r,\ell)}
  =\zeta_5^{yr+x\ell}
   U_{i+1}^{(r,\ell)}U_i^{(x,y)}.
\]
Hence the adjacent bicharacter is
\[
  \tau((x,y),(r,\ell))=\zeta_5^{yr+x\ell}.
\]
Ignoring the common normalization \(5^{-1}\), the product coefficient
function is
\[
  F(x,y)=\zeta_5^{ax^2+bx+cy^2+b_2y}.
\]
Put \(s=(u,v)\), \(t=(r,\ell)\), and use \((x,y)\) as the summation
variable in \eqref{eq:coefficient-ybe}.  The \(\tau^{-1}\)-factors on the
left and right are
\[
  \zeta_5^{-(yr+x\ell)},
  \qquad
  \zeta_5^{-(vx+uy)},
\]
respectively.  With
\[
  K=a(u^2+r^2)+b(u+r)+c(v^2+\ell^2)+b_2(v+\ell),
\]
the summands in the two coefficient sums become
\begin{align*}
  F(r,\ell)F(u-x,v-y)F(x,y)\zeta_5^{-(yr+x\ell)}
  &=\zeta_5^K
    \zeta_5^{2ax^2-(2au+\ell)x}
    \zeta_5^{2cy^2-(2cv+r)y},\\
  F(u,v)F(x,y)F(r-x,\ell-y)\zeta_5^{-(vx+uy)}
  &=\zeta_5^K
    \zeta_5^{2ax^2-(2ar+v)x}
    \zeta_5^{2cy^2-(2c\ell+u)y}.
\end{align*}
Thus \(b,b_2\) occur only in the common factor \(\zeta_5^K\), and the
coefficient equation separates as
\[
  S_a(2au+\ell)S_c(2cv+r)
  =S_a(2ar+v)S_c(2c\ell+u),
  \qquad u,v,r,\ell\in\F_5.
\]
Completing both squares shows that this identity is equivalent to
\[
  (4a-c^{-1})(u^2-r^2)
  +(4c-a^{-1})(v^2-\ell^2)=0
\]
for all \(u,v,r,\ell\in\F_5\).  Varying the first pair of variables and then
the second shows that both coefficients vanish; either equality gives
\(4ac=1\), or \(ac=-1\) in \(\F_5\).  Conversely, this condition makes the
coefficient identity hold, while the separated coefficient sums give
unitarity.

For the Family II identification, set
\[
  E:\F_5\longrightarrow\F_5,
  \qquad E(x)=2ax,
  \qquad Q_E(x)=\zeta_5^{ax^2+bx},
  \qquad P_E(x)=\zeta_5^{cx^2+b_2x}
  \quad(x\in\F_5).
\]
Then \(E^{-1}(x)=2cx\), and comparison with the normalized formula of
Proposition~\ref{prop:family-II-coordinate-change} gives
\[
  R^{\mathrm{prod}}_{a,b,c,b_2}
  =\gamma(P_E)R^{\mathrm{II},E}_{Q_E,P_E}.
\]
The scalar has modulus one.
\end{proof}

\begin{computation}\label{comp:ququint-finite-field-coverage}
After reduction to \(\F_{1201}\),
the comparison matches the two \(S_1\) local Clifford orbits with the cyclic
Gaussian representatives of Proposition~\ref{prop:ququint-cyclic-gaussian}
and the six \(S_2\) orbits with the product representatives of
Proposition~\ref{prop:ququint-product}, for which \(ac=-1\).  Consequently,
every local Clifford orbit left unidentified as monomial or group type by
Computation~\ref{comp:ququint-finite-field-group-type} is matched by a
signed Family I or Family \(\mathrm{II}_E\) operator.  These comparisons hold
over \(\F_{1201}\); they do not establish local-gauge equivalence over
\(\C\).
\end{computation}

\begin{remark}
As additional finite-field checks, we performed \(25\) tests for local
\(\GL_5\)-gauge equivalence between representatives with the same
projective spectrum and \(100\) tests for identifications by conditional
Doikou--Smoktunowicz \(T\)-moves.  None produced a further identification.
These are finite, nonexhaustive tests over \(\F_{1201}\), and their negative
outcome does not give a characteristic-zero obstruction.
\end{remark}

\section{Conclusion and outlook}
\label{sec:conclusion}

We have developed a common tower framework for several classes of unitary
Yang--Baxter operators.  The coefficient form of the braid relation isolates
the finite algebraic data behind the constructions, while the external
product, direct sum, and admissible internal orthogonal sum describe three
ways of combining solutions.  For finite metric groups, this
framework produces the three uniform families, both polarization signs,
the Family \(\mathrm{II}_E\) extension, compatible tensor-power realizations,
factorization formulas,
finite braid images, and exact polynomial-time evaluation of their trace
link invariants.

The cyclic constructions exhibit a complementary phenomenon.  For the fixed
dimension-\(8\) Pauli realization they form generically non-Gaussian one-parameter
families containing signed-Gaussian specializations.  Away from roots of
unity their braid images can be infinite, but the images remain virtually
abelian.  Thus finiteness and virtual abelianness separate naturally in the
examples, whereas the metric families satisfy both properties.

The low-dimensional calculations provide evidence within explicitly stated
Pauli and Clifford search spaces.  The qutrit representatives lead to a
Temperley--Lieb quotient with a HOMFLYPT specialization and to a BMW quotient
whose Kauffman invariant is a squared Jones specialization; the latter also
gives a strict localization of the semisimple BMW trace quotient.  The
quaternionic Family III
operator supplies a strict tensor-power realization of the
smallest-dimensional case in Lechner's exceptional two-eigenvalue family,
and the ququint calculation matches its remaining finite-field output with
signed Family I and Family \(\mathrm{II}_E\) representatives.  These
computations support, but do not replace, the structural arguments.

The three conjectures stated in the Introduction remain the main organizing
questions.  Natural next steps are to obtain
characteristic-zero versions of the ququint gauge comparisons, to understand
the admissible internal orthogonal-sum construction intrinsically rather than
one Pauli realization at a time, and to extend the searches beyond the present Pauli and
Clifford settings.
These problems separate the classification question from the computational
evidence and provide concrete tests for the proposed framework.

\section*{Data and code availability}

The scripts and data supporting the computational statements are available
from the authors upon reasonable request.  A versioned permanent archive for
the final public release is intended to include the software versions,
cyclotomic and finite-field conventions, enumeration scripts, orbit
representatives, gauge data, and machine-readable output certificates.

\section*{Acknowledgments}
C.G. was partially supported by Grant INV-2025-213-3452 from the School of Science of Universidad de los Andes.
E.C.R. was partially supported by U.S. NSF grant DMS-2205962 and a Royal
Society Wolfson Visiting Fellowship.  We thank Paul Martin for useful
conversations.

\bibliographystyle{amsplain}
\bibliography{unified_ybe_references}

@article{ApplebyBengtssonBrierleyGrasslGrossLarsson,
  author  = {Appleby, D. M. and Bengtsson, Ingemar and Brierley, Stephen and
             Grassl, Markus and Gross, David and Larsson, Jan-{\AA}ke},
  title   = {The monomial representations of the {Clifford} group},
  journal = {Quantum Inf. Comput.},
  volume  = {12},
  number  = {5--6},
  pages   = {404--431},
  year    = {2012},
  doi     = {10.26421/QIC12.5-6-3}
}

@article{BirmanWenzl,
  author  = {Birman, Joan S. and Wenzl, Hans},
  title   = {Braids, link polynomials and a new algebra},
  journal = {Trans. Amer. Math. Soc.},
  volume  = {313},
  number  = {1},
  pages   = {249--273},
  year    = {1989},
  doi     = {10.1090/S0002-9947-1989-0992598-X}
}

@inproceedings{CaiChenLiptonLu,
  author    = {Cai, Jin-Yi and Chen, Xi and Lipton, Richard J. and Lu, Pinyan},
  title     = {On tractable exponential sums},
  booktitle = {Frontiers in Algorithmics},
  series    = {Lecture Notes in Computer Science},
  volume    = {6213},
  pages     = {148--159},
  publisher = {Springer},
  year      = {2010},
  doi       = {10.1007/978-3-642-14553-7_16}
}

@article{DS,
  author  = {Doikou, Anastasia and Smoktunowicz, Agata},
  title   = {Set-theoretic {Yang--Baxter} \& reflection equations and quantum group symmetries},
  journal = {Lett. Math. Phys.},
  volume  = {111},
  pages   = {105},
  year    = {2021},
  doi     = {10.1007/s11005-021-01437-7}
}

@article{Durfee,
  author  = {Durfee, Alan H.},
  title   = {Bilinear and quadratic forms on torsion modules},
  journal = {Adv. Math.},
  volume  = {25},
  number  = {2},
  pages   = {133--164},
  year    = {1977},
  doi     = {10.1016/0001-8708(77)90002-0}
}

@article{GJ,
  author  = {Goldschmidt, David M. and Jones, Vaughan F. R.},
  title   = {{Metaplectic} link invariants},
  journal = {Geom. Dedicata},
  volume  = {31},
  number  = {2},
  pages   = {165--191},
  year    = {1989},
  doi     = {10.1007/BF00147477}
}

@unpublished{GalindoClifford,
  author        = {Galindo, C{\'e}sar},
  title         = {Splitting of {Clifford} groups associated to finite abelian groups},
  year          = {2026},
  eprint        = {2603.24743},
  archivePrefix = {arXiv},
  primaryClass  = {math.GR},
  note          = {Accepted for publication; arXiv:2603.24743}
}

@article{GustafsonKimballRowellZhang,
  author  = {Gustafson, Paul and Kimball, Andrew and Rowell, Eric C. and Zhang, Qing},
  title   = {Braid group representations from twisted tensor products of algebras},
  journal = {Peking Math. J.},
  volume  = {3},
  number  = {2},
  pages   = {103--130},
  year    = {2020},
  doi     = {10.1007/s42543-020-00023-5}
}

@article{Hietarinta92,
  author  = {Hietarinta, Jarmo},
  title   = {All solutions to the constant quantum {Yang--Baxter} equation
             in two dimensions},
  journal = {Phys. Lett. A},
  volume  = {165},
  number  = {3},
  pages   = {245--251},
  year    = {1992},
  doi     = {10.1016/0375-9601(92)90044-M}
}

@article{JonesCMP,
  author  = {Jones, Vaughan F. R.},
  title   = {On a certain value of the {Kauffman} polynomial},
  journal = {Comm. Math. Phys.},
  volume  = {125},
  number  = {3},
  pages   = {459--467},
  year    = {1989},
  doi     = {10.1007/BF01218412}
}

@article{JonesHecke,
  author  = {Jones, Vaughan F. R.},
  title   = {{Hecke} algebra representations of braid groups and link polynomials},
  journal = {Ann. of Math. (2)},
  volume  = {126},
  number  = {2},
  pages   = {335--388},
  year    = {1987}
}

@book{KasselTuraev,
  author    = {Kassel, Christian and Turaev, Vladimir},
  title     = {Braid Groups},
  series    = {Graduate Texts in Mathematics},
  volume    = {247},
  publisher = {Springer},
  address   = {New York},
  year      = {2008}
}

@article{LarsenRowell,
  author  = {Larsen, Michael J. and Rowell, Eric C.},
  title   = {An algebra-level version of a link-polynomial identity of {Lickorish}},
  journal = {Math. Proc. Cambridge Philos. Soc.},
  volume  = {144},
  number  = {3},
  pages   = {623--638},
  year    = {2008},
  doi     = {10.1017/S0305004107000424}
}

@article{LechnerPennigWood,
  author  = {Lechner, Gandalf and Pennig, Ulrich and Wood, Simon},
  title   = {{Yang--Baxter} representations of the infinite symmetric group},
  journal = {Adv. Math.},
  volume  = {355},
  pages   = {106769},
  year    = {2019},
  doi     = {10.1016/j.aim.2019.106769}
}

@misc{LechnerTwoEigenvalues,
  author        = {Lechner, Gandalf},
  title         = {The classification problem for unitary {R}-matrices
                   with two eigenvalues},
  year          = {2026},
  eprint        = {2603.20158},
  archivePrefix = {arXiv},
  primaryClass  = {math.QA},
  note          = {Preprint, arXiv:2603.20158}
}

@unpublished{MartinRowellTorzewska,
  author        = {Martin, Paul P. and Rowell, Eric C. and Torzewska, Fiona},
  title         = {A categorical perspective on braid representations},
  year          = {2026},
  eprint        = {2506.07950},
  archivePrefix = {arXiv},
  primaryClass  = {math.QA},
  note          = {Preprint, arXiv:2506.07950}
}

@unpublished{MartinRowell,
  author        = {Martin, Paul P. and Rowell, Eric C.},
  title         = {Classification of spin-chain braid representations},
  year          = {2026},
  note          = {To appear in Comm. Math. Phys.},
  eprint        = {2112.04533},
  archivePrefix = {arXiv},
  primaryClass  = {math.QA}
}

@article{MirandaMorrison,
  author  = {Miranda, Rick and Morrison, David R.},
  title   = {The number of embeddings of integral quadratic forms. {I}},
  journal = {Proc. Japan Acad. Ser. A Math. Sci.},
  volume  = {61},
  number  = {10},
  pages   = {317--320},
  year    = {1985},
  doi     = {10.3792/pjaa.61.317}
}

@article{Nikulin1980,
  author  = {Nikulin, V. V.},
  title   = {Integral symmetric bilinear forms and some of their applications},
  journal = {Math. USSR-Izv.},
  volume  = {14},
  number  = {1},
  pages   = {103--167},
  year    = {1980},
  doi     = {10.1070/IM1980v014n01ABEH001060}
}

@article{RowellWenzl,
  author  = {Rowell, Eric C. and Wenzl, Hans},
  title   = {{$SO(N)_2$} braid group representations are {Gaussian}},
  journal = {Quantum Topol.},
  volume  = {8},
  number  = {1},
  pages   = {1--33},
  year    = {2017},
  doi     = {10.4171/QT/85}
}

@article{RowellWangBAMS,
  author        = {Rowell, Eric C. and Wang, Zhenghan},
  title         = {Mathematics of topological quantum computing},
  journal       = {Bull. Amer. Math. Soc. (N.S.)},
  volume        = {55},
  number        = {2},
  pages         = {183--238},
  year          = {2018},
  doi           = {10.1090/bull/1605},
  eprint        = {1705.06206},
  archivePrefix = {arXiv},
  primaryClass  = {math.QA}
}

@article{RW,
  author  = {Rowell, Eric C. and Wang, Zhenghan},
  title   = {Localization of unitary braid group representations},
  journal = {Comm. Math. Phys.},
  volume  = {311},
  number  = {3},
  pages   = {595--615},
  year    = {2012},
  doi     = {10.1007/s00220-011-1386-7}
}

@misc{TaylorGaussSums,
  author        = {Taylor, Laurence R.},
  title         = {{Gauss} sums in algebra and topology},
  year          = {2022},
  eprint        = {2208.06319},
  archivePrefix = {arXiv},
  primaryClass  = {math.AT},
  note          = {Preprint, arXiv:2208.06319}
}

@article{Turaev1988,
  author  = {Turaev, Vladimir G.},
  title   = {The {Yang--Baxter} equation and invariants of links},
  journal = {Invent. Math.},
  volume  = {92},
  number  = {3},
  pages   = {527--553},
  year    = {1988},
  doi     = {10.1007/BF01393746}
}

@article{Wall,
  author  = {Wall, C. T. C.},
  title   = {Quadratic forms on finite groups, and related topics},
  journal = {Topology},
  volume  = {2},
  pages   = {281--298},
  year    = {1963},
  doi     = {10.1016/0040-9383(63)90012-0}
}

@article{RowellQuaternionic,
  author  = {Rowell, Eric C.},
  title   = {A quaternionic braid representation
             (after {Goldschmidt} and {Jones})},
  journal = {Quantum Topology},
  volume  = {2},
  number  = {2},
  year    = {2011},
  pages   = {173--182},
  doi     = {10.4171/QT/18},
  eprint  = {1006.4808},
  archivePrefix = {arXiv},
  primaryClass  = {math.QA}
}

@article{EtingofSchedlerSoloviev,
  author        = {Etingof, Pavel and Schedler, Travis and
                   Soloviev, Alexandre},
  title         = {Set-theoretical solutions to the quantum
                   {Yang--Baxter} equation},
  journal       = {Duke Math. J.},
  volume        = {100},
  number        = {2},
  pages         = {169--209},
  year          = {1999},
  doi           = {10.1215/S0012-7094-99-10007-X},
  eprint        = {math/9801047},
  archivePrefix = {arXiv}
}

@article{CarterElhamdadiSaito,
  author        = {Carter, J. Scott and Elhamdadi, Mohamed and
                   Saito, Masahico},
  title         = {Homology theory for the set-theoretic {Yang--Baxter}
                   equation and knot invariants from generalizations of
                   quandles},
  journal       = {Fund. Math.},
  volume        = {184},
  pages         = {31--54},
  year          = {2004},
  doi           = {10.4064/fm184-0-3},
  eprint        = {math/0206255},
  archivePrefix = {arXiv}
}

@article{GalindoRowellBVS,
  author        = {Galindo, C{\'e}sar and Rowell, Eric C.},
  title         = {Braid representations from unitary braided vector spaces},
  journal       = {J. Math. Phys.},
  volume        = {55},
  number        = {6},
  pages         = {061702},
  year          = {2014},
  doi           = {10.1063/1.4880196},
  eprint        = {1312.5557},
  archivePrefix = {arXiv},
  primaryClass  = {math.QA}
}

\end{document}